\documentclass[11pt,reqno,a4paper]{amsart}
 
\usepackage{graphicx}
\usepackage{color} 
\usepackage[dvips]{epsfig}
\usepackage[usenames,dvipsnames]{xcolor}
\usepackage{transparent}
\usepackage[latin1]{inputenc}  
\usepackage{dsfont}
\usepackage{gensymb}
\usepackage[active]{srcltx}\usepackage[colorlinks,pdfpagelabels,pdfstartview=FitH,bookmarksopen=true,bookmarksnumbered=true,linkcolor=blue,plainpages=false,hypertexnames=false,citecolor=red]{hyperref}
\usepackage{mathrsfs}
\usepackage{verbatim, amssymb,amsmath, amsfonts, amsthm, cases}
\usepackage{latexsym}
\usepackage[dvips]{epsfig}
\usepackage{float}
\usepackage{soul}
\allowdisplaybreaks 

\newtheorem{theorem}{Theorem}[section]
\newtheorem{proposition}[theorem]{Proposition}
\newtheorem{lemma}[theorem]{Lemma}
\newtheorem{corollary}[theorem]{Corollary}
\newtheorem{remark}[theorem]{Remark}
\newtheorem{definition}[theorem]{Definition}



\newcommand{\ep}{\varepsilon}

\newcommand{\N}{\mathbb{N}}

\newcommand{\R}{\mathbb{R}}
\renewcommand{\to}{\rightarrow}

\newcommand{\upp}{u_p}
\newcommand{\rad}{{\text{\upshape rad}}}
\providecommand{\norm}[1]{\lVert#1\rVert}
\def\ra{\rightarrow}
\def\sideremark#1{\ifvmode\leavevmode\fi\vadjust{\vbox to0pt{\vss
 \hbox to 0pt{\hskip\hsize\hskip1em
 \vbox{\hsize2.1cm\tiny\raggedright\pretolerance10000
  \noindent #1\hfill}\hss}\vbox to15pt{\vfil}\vss}}}%

\begin{document}

\numberwithin{equation}{section}
\parindent=0pt
\hfuzz=2pt
\frenchspacing

\title[Quasi-radial nodal solutions]{
Quasi-radial nodal solutions  \\ for the Lane-Emden problem in the ball
}

\author[]{F. Gladiali, I. Ianni}

\address{Francesca Gladiali, Polcoming, Universit\`a  di Sassari  - Via Piandanna 4, 07100 Sassari, Italy}
\address{Isabella Ianni, Dipartimento di Matematica e Fisica, Universit\`a degli Studi della Campania \emph{Luigi Vanvitelli}, V.le Lincoln 5, 81100 Caserta, Italy}

\thanks{2010 \textit{Mathematics Subject classification:}  }


\thanks{Research partially supported by: PRIN $201274$FYK7$\_005$ grant and INDAM - GNAMPA}

\begin{abstract}
We consider 
the semilinear elliptic  problem
\begin{equation}\label{problemAbstract}
\left\{\begin{array}{lr}-\Delta u= |u|^{p-1}u\qquad  \mbox{ in }B\\
u=0\qquad\qquad\qquad\mbox{ on }\partial B
\end{array}\right.\tag{$\mathcal E_p$}
\end{equation}
where $B$ is the unit ball of $\R^2$ centered at the origin and $p\in (1,+\infty)$. We prove the existence of non-radial sign-changing solutions to \eqref{problemAbstract} which are \emph{quasi-radial}, namely solutions whose nodal line is the union of a finite number of disjoint simple closed curves, which are the boundary of nested domains contained in $B$. In particular the nodal line of these solutions doesn't touch $\partial B$.
\\
The result is obtained with two different approaches:  via nonradial bifurcation from the least energy sign-changing radial solution $u_p$ of \eqref{problemAbstract} at certain values of $p$ and by investigating the qualitative properties, for $p$ large, of the least energy nodal solutions in spaces of functions invariant by the action of the dihedral group generated by the reflection with respect to the $x$-axis   and the rotation about the origin of angle $\frac{2\pi}{k}$ for suitable integers $k$.\\
We  also prove that for certain integers $k$ the least energy nodal solutions in these spaces of symmetric functions are  instead radial, showing in particular a breaking of symmetry phenomenon in dependence on the exponent $p$.
\end{abstract}

\maketitle

\section{Introduction}\label{Introduction}
We consider the semilinear Lane-Emden problem
\begin{equation}\label{problem}
\left\{\begin{array}{lr}-\Delta u= |u|^{p-1}u\qquad  \mbox{ in }B\\
u=0\qquad\qquad\qquad\mbox{ on }\partial B\\
\end{array}\right.
\end{equation}
where $B\subset\R^2 $ is the unit ball centered at the origin and $p>1$.
\\
It is well known that \eqref{problem} admits a unique positive ground state solution which is radially symmetric. 
Observe that the oddness of the nonlinearity implies that $u$ is a solution of \eqref{problem} if and only if $-u$ is a solution,  so there is also a unique negative  solution to \eqref{problem}.
\\
Moreover, due to the oddness of the nonlinear term, standard variational methods
give the existence of infinitely many sign-changing solutions.

While the ground state solution of \eqref{problem} has been widely investigated, not much is known for nodal ones. 
Among these one can select the least energy nodal solutions, which  can be obtained by minimizing the associated energy functional 
\begin{equation}\label{functional}E_p(u)\ : \ =\ \frac 12 \int_{B}|\nabla u |^2 -\frac 1{p+1}\int_B |u|^{p+1}
\end{equation}
on the nodal Nehari set in the Sobolev space $H^1_0(B)$ (see \cite{CastroCossioNeuberger} for details). We denote a least energy sign-changing solution by $\widetilde u_p$. In \cite{BartschWeth} it has been shown that 
\begin{equation}
\label{leastRegioEMorse} \sharp(\widetilde u_p)=2\quad \mbox{ and }\quad m(\widetilde u_p)=2,
\end{equation}
where $\sharp(u)$ is the number of nodal regions of $u$ and $m(u)$ is the Morse index of the solution $u$ (see Section \ref{section:linearizedOperator} for the definition). Moreover in \cite{BartschWethWillem} 
it has been proved that $\widetilde u_p$ partially inherits the symmetries of the domain, being \emph{foliated Schwarz symmetric}, namely  axially symmetric with respect to an axis passing through the origin and nonincreasing in the polar angle from this axis (see also \cite{PacellaWeth}). 
\

Since the domain $B$ is radially symmetric one  can restrict to the Sobolev space of radial functions $H^1_{0, \rad}(B)$ and prove the existence of infinitely many sign-changing radial solutions for \eqref{problem}. More precisely it can be proved that for every $m\in\mathbb N_0:=\mathbb N\setminus \{0\}$ 
there exists a unique radial solution to \eqref{problem} that satisfies
\begin{equation}\label{segno-in-zero}
u(0)>0
\end{equation}
and such that $\sharp(u)=m$ (see \cite{NN}, \cite{KAJIKIYA}). 
We denote by $u_p$ the unique radial least energy sign changing solution to \eqref{problem}  which satisfies \eqref{segno-in-zero}, clearly
\begin{equation}
\label{radduenod}
\sharp(u_p)=2.
\end{equation}
Morover it has been proved in \cite{AftalionPacella} that 
\begin{equation}
\label{StimaAPMorseR} 
m(u_p)\geq4
\end{equation}
(see also \cite{DeMarchisIanniPacellaMathAnn} where $m(u_p)$ has been explicitly computed for $p$ large and also \cite{DeMarchisIanniPacellaAdv} where the previous estimate on the Morse index  has been generalized to any radial solutions with $m$ nodal regions, with bound given by the number $3m-2$).

\

Comparing the information on the Morse index in  \eqref{leastRegioEMorse} with the one in \eqref{StimaAPMorseR} one gets that the radial solution $u_p$ is not the least energy sign-changing solution in the whole space $H^1_0(B)$, namely that $u_p\neq \widetilde u_p$. As a consequence the monotonicity of $\widetilde u_p$ with respect to the polar angle (as recalled above $\widetilde u_p$ is foliated Schwarz symmetric) must be strict at some region,    and in \cite{PacellaWeth} it is actually proved that, for $p>2$,  the monotonicity is always strict.
 Moreover in \cite{AftalionPacella} it has been also proved that 
the nodal set of $\widetilde u_p$ touches the boundary of $B$.

\

One can also restrict to the Sobolev space $H^1_{0,k}(B)$ of the functions in  $H^1_0(B)$ which are even and $\frac{2\pi}{k}$-periodic in the angular variable, for $k\in\mathbb N_0$ and similarly show the existence of infinitely many sign-changing symmetric solutions in $H^1_{0,k}(B)$, among which we denote by $u_p^k$ the least energy ones.

\

 Anyway \emph{a priori} it is not clear whether this procedure produces \emph{new solutions} or not. 
Indeed, clearly $u_p^1=\widetilde u_p$ (since $\widetilde u_p $ is axially symmetric) and, even though $u_p^k\neq \widetilde u_p$ for $k\geq 2$ (since if they  coincide then $\widetilde u_p$ would be $\frac{2\pi}{k}$-periodic in the angular variable and so necessarily radial by the Schwarz symmetry, getting a contradiction), $u_p^k$ \emph{could be radial}. 

\

In particular it would be interesting to show the existence of sign-changing solutions to \eqref{problem} which belong to $H^1_{0,k}(B)$
 but are \emph{not radially symmetric}, having nevertheless a \emph{quasi-radial shape}, in the sense of the following definition:
\begin{definition} 
\label{def:quasiradial} 
We say that a solution of \eqref{problem} is \emph{quasi-radial} if its nodal set is the union of a finite number of  disjoint simple closed curves which are the boundary of nested domains contained in $B$.
\end{definition} 
Observe that the nodal line of a  \emph{quasi-radial} solution doesn't  touch the boundary of the ball $B$. Clearly any radial solution is quasi-radial.

\

By the asymptotic estimates for the energy of the solutions of \eqref{problem}  in \cite{RenWei}, the obvious inequality $E_p(u_p^k)\leq E_p(u_p)$ and the upper bound 
\[pE_p(u_p)\leq \alpha\cdot 4\pi e, \quad\mbox{ for }p \mbox{ large},\]
proved in \cite{GGP2} for a certain value $\alpha\in (4.5,5)$,  one derives the following upper bound on the number of nodal regions of $u_p^k$:  
\[\sharp(u_p^k)\leq 4 \quad \forall k\in\mathbb N_0, \mbox{ for }p \mbox{ large}.\] 
Combining this bound with the results in \cite{DeMarchisIanniPacellaJDE} (which hold in symmetric and simply connected domains, more general than the ball) it then follows that the least energy symmetric solution 
\begin{equation}\label{primaquadratino}
u_p^k \mbox{ is \emph{quasi-radial} when }k\geq 4\mbox{ and }p\mbox{  is large,}
\end{equation}
from which  in particular one also derives
\begin{equation}
\label{quadratino}
\sharp(u_p^k)=2 \quad \mbox{ and }\quad  m(u_p^k)\geq4, \quad \mbox{ for }k\geq 4, \mbox{ for }p \mbox{ large.}
\end{equation}
Observe that the properties in \eqref{primaquadratino}, \eqref{quadratino} are satisfied also by $u_p$ (see \eqref{radduenod}, \eqref{StimaAPMorseR}), hence the question of the existence of symmetric but non-radial solutions of \eqref{problem}  which are \emph{quasi-radial} is still open. 
 \\  
Moreover, as $p\in (1,+\infty)$ and $k\in\mathbb N_0$ vary,  one would like to investigate whether $u_p^k$  coincides with the radial  least energy nodal solution $u_p$ or not.

\begin{figure}[h]
  \centering\Huge
\resizebox{80pt}{!}{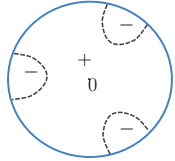}
\caption{$k=3$. Symmetric and not quasi-radial function with $4$ nodal regions}
\end{figure}

\

We start by giving a positive answer to the first question, showing the existence of three distinct solutions to \eqref{problem} which belong respectively to $H^1_{0,k}(B)\setminus H^1_{0,\rad}(B)$,  for $k=3,4,5$. Each solution   bifurcates from the least energy radial nodal solution $u_p$ at certain values of $p$ and close to the bifurcation point it is \emph{quasi-radial}.
The result is the following, where $\mathcal X_k:=H^1_{0,k}(B)\cap C^{1,\alpha}(\bar B)$ (and $C^{1,\alpha}(\bar B)$ denotes the space of $C^1(\bar B)$ functions with H\"older derivatives):
\begin{theorem}\label{teo1}
For any $k=3,4,5$ there exists at least one exponent $p^k\in(1,+\infty)$ such that $(p^k, u_{p^k})$ is a nonradial bifurcation point for problem \eqref{problem}. The bifurcating solutions are sign-changing,  belong to $ \mathcal X_k$ and close to the bifurcation point they have two nodal domains and are  quasi-radial. Moreover the bifurcation is global and, letting $\mathcal{C}_k$ be the continuum that branches out of $( p^k, u_{p^k})$, then either $\mathcal{C}_k$ is unbounded in $(1,+\infty)\times\mathcal X_k$ or it intersects $\{1\}\times\mathcal X_k$. Finally at any point along each  branch $\mathcal C_k$ either the solution belongs to $\mathcal X_k\setminus\mathcal X_j$, $\forall j>k$ or it is radial, in particular the continua bifurcating from different values of $k$ can intersect only at radial solutions. 
\end{theorem}

\begin{figure}[h]
  \centering\Huge
\resizebox{240pt}{!}{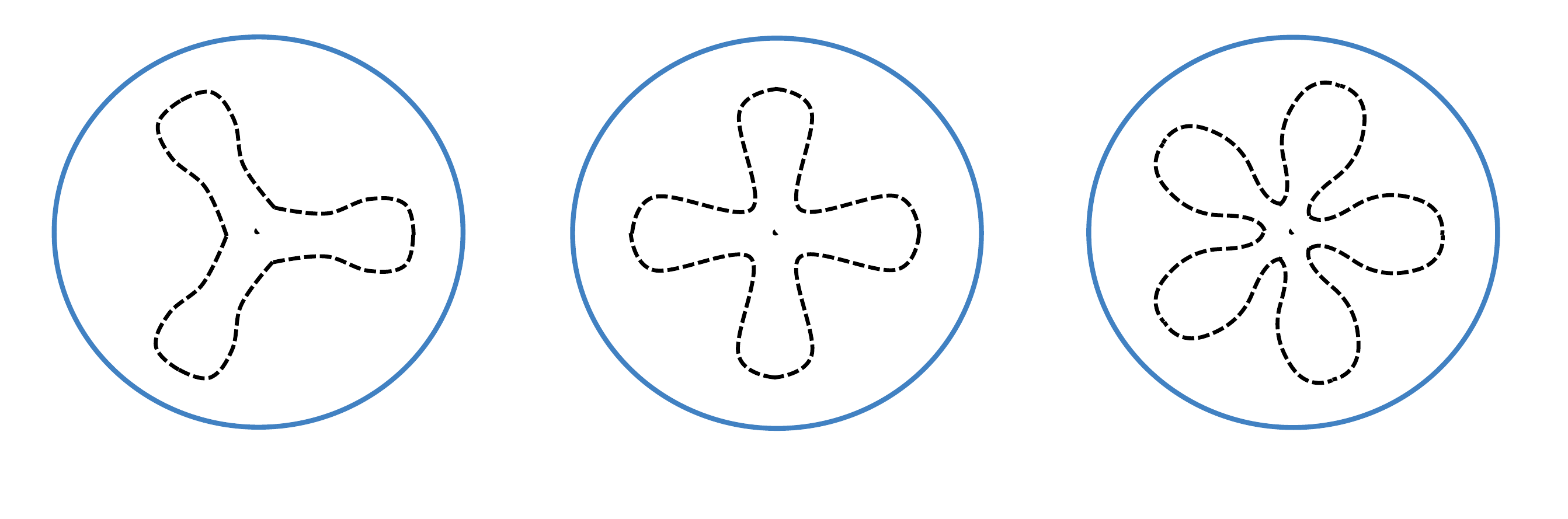}
\caption{}
\end{figure}

Our second goal is to understand  whether the least energy symmetric  solution $u_p^k$, $k\in\mathbb N_0$, $p\in (1,+\infty)$, coincides with the radial  least energy nodal  solution $u_p$ or not, and this is analyzed in next result:
\begin{theorem} \label{prop1.4}
Let  $u_p^k$ be the least energy sign-changing solution of \eqref{problem} in the space $H^1_{0,k}(B)$, $k\in\mathbb N_0$, then there exist $\delta >0$ and $p^{\star} >1$ such that:
\begin{itemize} 
\item[$i)$] for $k=2$: $\quad u_p^k$  is non-radial both for $p\in (1,1+\delta)$ and $p\geq p^{\star}$;
\item[$ii)$] for $k=3,4,5$: $\quad u_p^k$  is radial for $p\in (1,1+\delta)$ and non-radial when $p\geq p^{\star}$;
\item[$iii)$] for $k\geq 6$: $\quad u_p^k$  is radial for $p\in (1,1+\delta)$.
\end{itemize}
Clearly when $u_p^k$ is radial then it coincides with $u_p$ (up to the sign).
\end{theorem}

The fact that the symmetry of the domain is not totally caught by these least energy solutions is reasonable, since we are dealing with sign-changing solutions, anyway the symmetry breaking phenomenon when $k=3,4,5$ (case $ii)$) and its dependence  on the value of the exponent $p$  were totally unexpected. It is also interesting that we can identify the symmetries of the solutions at which this symmetry breaking phenomenon occurs.

Theorem \ref{prop1.4}-$ii)$ combined with \eqref{primaquadratino} and \eqref{quadratino} provides another example for non-radial symmetric  sign-changing solution of \eqref{problem} which are \emph{quasi-radial} in the sense of Definition \ref{def:quasiradial}. Differently with respect to Theorem \ref{teo1}, this result is now for any $p$ large enough:
\begin{corollary}\label{cor:teoleast}
Let $k=4,5$, then   $u_p^k$ is not radial but it is \emph{quasi-radial} for $p$ large  enough. 
In particular $u_p^k\neq u_p$. Moreover $u_p^k\neq\widetilde u_p$ and \eqref{quadratino} holds.
\end{corollary}
\begin{figure}[h]
  \centering
  \resizebox{410pt}{!}{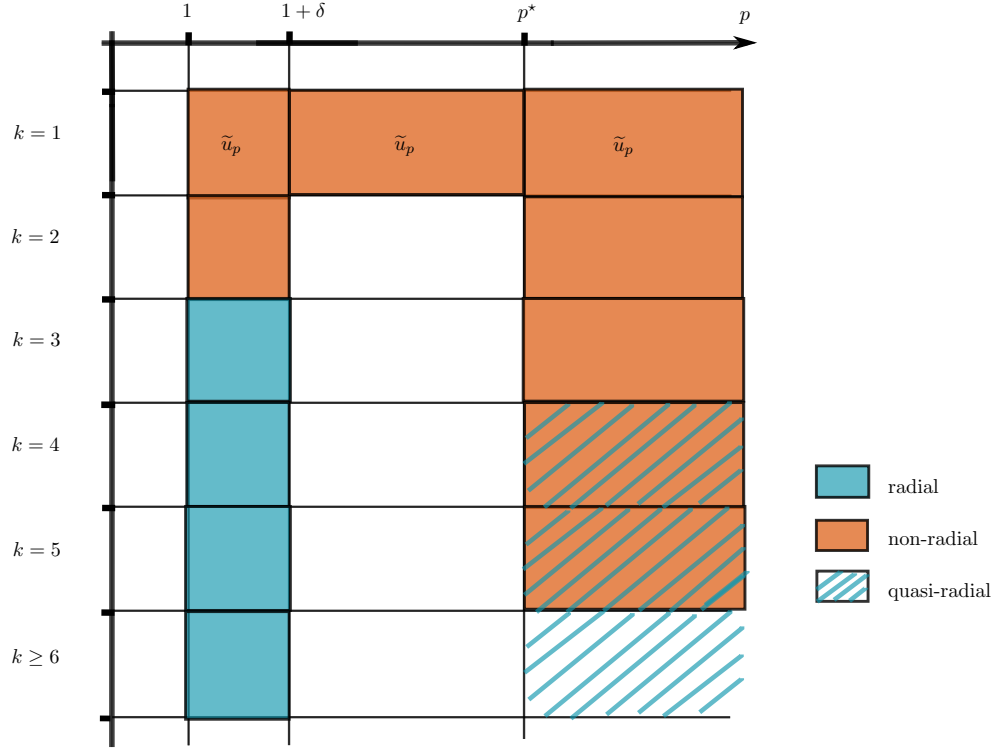}
  \caption{Symmetry of $u_p^k$ from Theorem \ref{prop1.4} and Corollary \ref{cor:teoleast}}
\end{figure}
We conjecture that  the bifurcating solution in $\mathcal X_k$ found in Theorem \ref{teo1}  not only exists for any $p\geq p^k$ but also coincides with $u_p^k$, when $k=4,5$ and even $3$. 
Differently from the higher symmetry cases considered in Corollary \ref{cor:teoleast}, when $k=3$ we do not expect $u_p^k$ to keep the \emph{quasi-radial} shape for large $p$.
For $k=2$ we believe that $u_p^k$  is not radial for all $p$ and also not \emph{quasi-radial} (when $p$ is close to $1$ it could be proved  rigorously,  see Remark \ref{RemarkNoQuaiRadk=2}),  for $k=1$  we recall that $u_p^k=\widetilde u_p$ for any $p\in (1,+\infty)$.  
The case $k\geq 6$ and $p$ large is not covered by the previous result, we conjecture that $u_p^k$ is radial, observe that this is not in contrast with \eqref{quadratino}. 
The asymptotic behavior, as $p\rightarrow +\infty$, of the least energy sign-changing solution $u_p^k$ of \eqref{problem} in the spaces $H^{1}_{0,k}(B)$ will be object of a subsequent paper \cite{GIP}.

\

Next we briefly explain the main ideas to get Theorem \ref{teo1} and Theorem \ref{prop1.4}.

\

The bifurcation in Theorem \ref{teo1} is  with respect to the exponent $p$ of the nonlinearity, previous results in this direction can be found for instance in \cite{GladialiGrossiPacellaSrikanth} and \cite{G10}. 
Observe that the bifurcation   can occur only at values $p$ at which the least energy nodal radial solution $u_p$ is degenerate and that a sufficient condition to identify degeneracy  points is to have a change in the Morse index of $u_p$.
  \\
This paper starts then from the recent results in \cite{DeMarchisIanniPacellaMathAnn} where the Morse index of the radial least energy sign-changing solution $u_p$ is computed for large values of $p$, proving the existence of an exponent  $p^{\star}>1$ such that:
\begin{equation}\label{morseplargeintro}
m(u_p)=12 \qquad \forall\ p\geq p^{\star}.
\end{equation}
This result is only for large $p$ and it strongly relies on the asymptotic behavior of $u_p$ as $p\rightarrow +\infty$, which has been described in  \cite{GGP2}. Indeed, an asymptotic analysis of the behavior of the solution $u_p$ as $p\rightarrow 1$ shows that a suitable re-normalization of $u_p$ converges to the second radial eigenfunction of the Laplace operator with Dirichlet boundary conditions (see Lemma \ref{lemma-pvicino1}) and this allows  to compute the Morse index of $u_p$ for $p$ close to $1$, showing that it has a different value in this range. More precisely in Proposition \ref{risultatoMorse_pvicino1} we get the existence of $\delta>0$ such that
\begin{equation}\label{morse-index-6}
 m(u_p)=6 \qquad \forall\ p\in (1,1+\delta).
\end{equation}
Hence \eqref{morseplargeintro} and \eqref{morse-index-6} prove  that along the branch of radial solutions $(p,u_p)$ of \eqref{problem} there should be points at which the Morse index increases and this change of the Morse index of $u_p$ in the interval $(1,+\infty)$  suggests  bifurcation  from $u_p$. 

\

We underline that in the convex domain $B$ this phenomenon is specific of sign-changing solutions, since the positive solution in $B$ is unique and non-degenerate  (for the uniqueness in more general convex domains see \cite{DeMarchisGrossiIanniPacellaprogress}).
\\
Anyway this is the first time that a non-radial bifurcation result from sign-changing solutions in convex domains is  observed and there was no chance to get it
before the study of the Morse index in  \cite{DeMarchisIanniPacellaMathAnn}.

\

To prove the result in Theorem \ref{teo1} we need first to analyze the degeneracies of the solution $u_p$. This is the goal of Sections \ref{section:generale}, \ref{section p large} and \ref{sse:pvicino1}.
We first consider in Section \ref{se:auxiliary} an auxiliary singular weighted eigenvalue problem 
\begin{equation} \label{problemaPesatoPallaINTRO}
\begin{cases}
\begin{array}{ll}
-\Delta \psi - p|u_p(x)|^{p-1} \psi =\frac{\beta}{|x|^2} \psi\qquad & \text{ in } B\setminus\{0\}, \\
\psi= 0 & \text{ on } \partial B\\
\int_B|\nabla \psi|^2+\frac{\psi^2}{|x|^2}<+\infty,
\end{array} 
\end{cases}
\end{equation}
which has the same kernel and the same number of negative eigenvalues of the linearized operator at $u_p$ (see Lemma \ref{lemma:Morse=numeroAutovaloriesatoPalla}) and 
whose main advantage relies on the fact that, in addition, a classical spectral decomposition into radial and angular part may be applied to it (Lemma \ref{lemma:decomposizionePalla}). 
The weighted eigenvalue problem \eqref{problemaPesatoPallaINTRO} belongs to the class of eigenvalue problems which has been studied in \cite{GGN2}, where the  eigenvalues for \eqref{problemaPesatoPalla} have been variationally characterized in the case when they are \emph{negative}. 
\\
Since  $u_p$ is the radial least energy nodal solution, then in the space of  radial functions its Morse index is $2$, in  Section \ref{subsec:mo}, in view of the spectral decomposition,  we  estimate the two negative radial eigenvalues of problem \eqref{problemaPesatoPallaINTRO} from above and from below by certain consecutive eigenvalues of $-\Delta_{S^1}$. The proof is based on the approximation of the negative eigenvalues of problem \eqref{problemaPesatoPallaINTRO} by the negative eigenvalues of a family of  weighted eigenvalue problems in annuli already studied in \cite{DeMarchisIanniPacellaMathAnn}, in particular we can extend some previous estimates in \cite{DeMarchisIanniPacellaMathAnn} related to the negative radial eigenvalues in annuli  to the negative eigenvalues for the singular problem \eqref{problemaPesatoPallaINTRO}. As a consequence of our estimates we get  information about the Morse index of the solution $u_p$ (Lemma   \ref{proposition:autovaloriRadiaiGenerale}) and a general characterization of its degeneracy (Proposition \ref{p4.7}), for any $p>1$. Finally, thanks to \eqref{morseplargeintro} and \eqref{morse-index-6}, we get more specific results both in the case   
 $p$ large and  $p$ close to $1$ (see Sections \ref{section p large} and \ref{sse:pvicino1}).
\\
Observe that, due now to the spectral decomposition, we can decompose any solution of the linearized equation at $u_p$ (and more in general each solution of the eigenvalue problem  \eqref{problemaPesatoPallaINTRO}) along spherical harmonics,  which in $\R^2$ are  the functions $\cos ( j\theta), \ \sin (j \theta)$ with $j\in \N$, getting in particular an explicit representation of the solutions of the linearized equation when they are nontrivial (and more in general of the eigenfunctions of \eqref{problemaPesatoPallaINTRO} associated with \emph{negative} eigenvalues).
As a consequence we can then \emph{identify the symmetries} of those functions which  are responsible of the degeneracy of $u_p$ (or which give rise to negative eigenvalues for the linearized operator at $u_p$). This aspect has been investigated in Section \ref{se:symmspaces}, where the symmetric spaces ($H^1_{0,k}(B)$ and) $\mathcal X_k$ have been introduced and the degeneracy and Morse index  of $u_p$ in these spaces studied (see Proposition \ref{Morse-simmetrico1}, \ref{Morse-simmetrico2} and \ref{lemma:degenerazioneSimmetria}).
\\
The reason for restricting to the spaces $\mathcal X_k$ is to isolate a \emph{unique}  function in the kernel of the linearized operator; more precisely, on one side it allows to select one suitable spherical harmonic (between $\sin$ and $\cos$) that produces degeneracy and, on the other side it avoids a possible double degeneracy due to the contemporary vanishing of two eigenvalues,  possibility that cannot be ruled out and it is specific of sign-changing solutions. 
Since we do not know explicitly the solution $u_p$, it is not clear whether the transversality condition of the well-known Crandall-Rabinowitz Theorem (for one dimensional kernel) is satisfied or not. Anyway the bifurcation result may be obtained here using a degree argument. 
The separation of the  branches  is obtained defining suitable cones $\mathcal K_k\subset\mathcal X_k$ of monotone functions introduced by Dancer in \cite{D} and using the degree in cones, see \cite{Amman} (see Section \ref{se:bifurcation} for the definitions of the cones).
The \emph{quasi-radiality}  is inherited from the radial least energy solution $u_p$, since  near the bifurcation point the bifurcating solution is a small perturbation of it (see  Remark \ref{remark:shape}).\\
Along the branch instead the number of nodal regions and the shape of the solutions may change, anyway the characterization of the behavior for branches of non-radial solutions may be a very difficult task to investigate, we also conjecture that the branches exist for every $p\geq p^k$.

\

In this paper we have focused on the radial least energy sign-changing solution $u_p$ of \eqref{problem}. A bifurcation result similar to Theorem \ref{teo1} could be obtained from any nodal radial solution $u_p^m$ of  \eqref{problem} with $m> 2$ nodal regions, provided information about its Morse index when $p$ is large is available. In this case  we expect that the symmetries which cause the degeneracy and hence produce branches of bifurcating solutions, should be of the same type of the one for functions in $\mathcal X_k$ (which derive by the symmetry groups of spherical harmonics), but with different values of $k$, probably $k\geq 6$.
  
  \

Moreover one could think to extend the bifurcation result in Theorem \ref{teo1} also to higher dimension $N\geq 3$, when $p\in (1,\frac{N+2}{N-2})$.
Indeed the behavior of all the radial sign-changing solutions of \eqref{problem} has been studied in \cite{DeMarchisIanniPacellaAdv} and in particular their Morse index has been explicitly computed when $p$ is sufficiently close to $\frac{N+2}{N-2}$, giving for instance,  for the radial least energy sign-changing solution $u_p$:
\[m(u_p)=2+N, \quad \mbox{ for $p$ close to $\frac{N+2}{N-2}$. }\]
Similarly as in the $2$-dimensional case, we expect a change in the Morse index of $u_p$ as $p$ varies from $1$ to $\frac{N+2}{N-2}$. Indeed $u_p$ should converge as $p\rightarrow 1$  to the radial Dirichlet eigenfunction with $2$ nodal regions of the Laplace operator in $B$ and this would imply
\[m(u_p)= 2+N+\frac{(N+2)(N-1)}2, \quad\mbox{ for $p$ close to $1$.}\]  
Again a change in the Morse index  should give a nonradial bifurcation result. An extra difficulty in dimension $N\geq 3$  would be to identify the symmetry groups of the spherical harmonics, which are much more involved than those of the $2$-dimensional spherical harmonics, see for instance \cite{AG}. 


\

Next we discuss the main ideas behind the proof of Theorem \ref{prop1.4}, which is contained in Section  \ref{section:proofTeoLeastEnergy}.

\

The \emph{non-radial} part is a byproduct of the study of the symmetry groups that cause the degeneracy and the bifurcation from $u_p$. Indeed in order to prove that $u_p$ and $u_p^k$ do not coincide one would like to compare their Morse indexes and show that they are different. However the computation of $m(u_p^k)$ may be very difficult, but if we restrict to  the symmetric spaces $H^1_{0,k}(B)$ then the Morse index of $u_p^k$ is always $2$ (see Lemma \ref{lem-10-1}). On the other side 
we are able to compute the  \emph{symmetric Morse index} also for the radial solution  $u_p$  (Proposition \ref{Morse-simmetrico1} and \ref{Morse-simmetrico2}). Observe that it is computed only for $p$ close to $1$ and $p$ large since it is deduced, among other things, from the asymptotic analysis of $u_p$ as $p\rightarrow 1$ and as $p\rightarrow +\infty$ respectively.

\

The proof of the {\sl radial part} of Theorem \ref{prop1.4} is more involved. It relies on a careful blow-up procedure in the spirit of \cite{GS} for showing $L^{\infty}$ bounds for the solutions $u_p^k$ (see Proposition \ref{prop7.4}). Once an $L^{\infty}$ bound
is available one can deduce the result by studying the asymptotic behavior of the solutions $u_p^k$ as $p\rightarrow 1$ (see the proof of Proposition \ref{leastRadiale}). In particular a delicate expansion  of $\norm{u_p^k}_{\infty}$ at $p=1$ up to the second order is needed.
\\
Getting a uniform  $L^{\infty}$ bound is somehow standard for solutions with uniformly bounded Morse index, since  one shows that the bound on the Morse index is preserved as $p\rightarrow 1$, while  the blow-up analysis of 
unbounded  solutions in $L^{\infty}$-norm  leads to solutions to limit problems in unbounded domains, whose  Morse index is not finite, thus reaching a contradiction.
\\
The main problem here is that for the least energy symmetric solutions  $u_p^k$ we do not have a bound for the full Morse index, but {\sl  only for the $k$-Morse index} (see Lemma \ref{lem-10-1}), while in the rescaling procedure the symmetries are not preserved.
\\
To overcome this technical difficulty we exploit the symmetry of $u_p^k$ and reduce problem \eqref{problem} to the circular sector $S_k$ of the ball of amplitude $\frac{\pi}k$, for $k\in \N_0$.
In particular we are able to convert  the bound on the  $k$-Morse index   to a bound on the full Morse index of $u_p^k$ in the sector $S_k$ (Morse index for a mixed Dirichlet-Neumann problem, see Lemma \ref{lem-10-nuovo}) and finally we perform the blow-up argument in $S_k$.
\\
Also the blow-up procedure in $S_k$ requires special care, since we have to deal with mixed boundary conditions and, above all, with the angular points of $S_k$. For these reasons  the analysis of the rescaled solutions includes several different cases, depending upon the location of the maximum points in the sector. Anyway in all the cases we end-up with solutions to  a limit linear problem in  unbounded domains 
with either Dirichlet or Neumann or mixed boundary conditions, whose Morse index is finite. Finally studying the Morse index of solutions for these limit problems (Proposition \ref{prop7.3}) we get a contradiction.

\tableofcontents

\section{Preliminary results} \label{section:preliminaries}   
\begin{proposition}\label{PropositionUnicoMaxeMin}  \eqref{problem} admits a unique radial solution $(u_p)$ having $2$ nodal regions and satisfying \eqref{segno-in-zero}. Moreover:
\begin{itemize}
\item[$(i)$] $u_p(0)=\|u\|_{\infty}$
\item[$(ii)$] in each nodal region there is exactly one critical point (namely the maximum and the minimum points)
\end{itemize}
\end{proposition}

In \cite[Lemma 5.2]{HRS} the authors proved the following estimate that can be useful in the sequel: 
\begin{lemma}
For any $p_*\in (1,+\infty)$ there exist constants $m,M$ such that, for any $p\in(1,p_*]$
\begin{equation}\label{*}m\leq  \left(\norm{ u_p}_{\infty} \right)^{p-1}\leq M.
\end{equation}
\end{lemma}
Finally we state a Proposition which provides the behavior, at the singularity, of solutions to a singular ordinary differential equation. This result is partially contained in \cite[Lemma 2.4]{GGN2}, although one implication is new and  proved here.
\begin{proposition}\label{p2.2}
Let $\psi$ be a solution to 
\begin{equation}\label{f1a}
\begin{cases}
-\psi'' - \frac{1}{r}\psi' +\beta^2\frac{\psi}{r^2}=h\psi, \quad \text{in} \ \ (0,1)\\
\psi(1)=0,        \ \int_0^1 r (\psi')^2dr< \infty
\end{cases}
\end{equation}
with $h\in L^\infty(0,1)$ and $\beta>0$. Assume that $\psi$  satisfies one of the following conditions:
\begin{eqnarray*}
a)  & \psi(0)=0\\
b) &  \int_0^1\frac{\psi^2}{r} dr< \infty.
\end{eqnarray*}
Then $\psi\in L^\infty(0,1)$ and 
\begin{equation}\label{psi-in-zero}
\psi(r)=O(r^{\beta}) \ \ \text{ as } \ r\to 0.
\end{equation} 
\end{proposition}
\begin{proof}
When $\psi$ satisfies condition $b)$ then the thesis follows from Lemma 2.4 in \cite{GGN2} (see 
estimate (2.28)). When $\psi$ satisfies condition $a)$ we observe that since $\psi$ solves \eqref{f1a} then $\psi\in L^{\infty}(0,1)$. We can proceed as 
in the proof of Lemma 2.4 in \cite{GGN2}. 
Then, multiplying by $r_n$ \eqref{f1a} and integrating in $(r_n,1)$ we get 
\[r_n^{\beta+1}\psi'(r_n)-r_n^\beta \psi'(1)+\beta^2r_n^{\beta}\int_{r_n}^1\frac{\psi}r \ dr=r_n^{\beta}\int_{r_n}^1r h(r)\psi(r)\ dr.\]
Using the fact that along a sequence $r_n\to 0$ it holds
\[\big| r_n^{\beta}\int_{r_n}^1 \frac {\beta^2}s\psi(s) \ ds \Big| \leq C r_n^{\beta}|\log{r_n}|=o(1)\]
we get as $n\to \infty$
\[r_n^{\beta+1}\psi'(r_n)=o(1).\]
Observe now that the function $v(r)=r^{\beta}$ satisfies
\begin{equation}\label{inter}
-v''-\frac 1r v'+\frac{\beta^2}{r^2}v=0 \ \text{ in }(0,1) \ , \ v(0)=0
\end{equation}
We multiply \eqref{f1a} by $v$, we multiply \eqref{inter} by $\psi$, we integrate on $(r_n,R)$, with $R\in (0,1)$, we subtract the two equations and we get
\[\int_{r_n}^R r^{\beta+1}h(r)\psi(r) \ dr=r_n^{\beta+1}\psi'(r_n)-\beta r_n^{\beta}\psi(r_n)-R^{\beta+1}\psi'(R)+\beta R^{\beta}\psi(R)\]
and, passing to the limit as $n\to \infty$
\[
 \int_0^R r^{\beta +1}h(r)\psi(r) \ dr=-R^{\beta+1}\psi'(R)+\beta R^{\beta}\psi(R)\]
which implies for any $t\in(0,1)$
\begin{equation}\label{fra1}
\frac{\psi(t)}{t^{\beta}}=\int_t^1\frac1{R^{2\beta+1}}
\left(\int_0^Rs^{\beta+1}h(s)\psi(s)ds\right)dR.
\end{equation}
The boundedness of $h(s)$ and $\psi(s)$ then gives
\begin{equation}\label{fra2}
\Big|\int_0^R s^{\beta +1}h(s)\psi(s)ds\Big|\leq CR^{\beta+2}
\end{equation}
which, together with \eqref{fra1} gives
\[
\frac{|\psi(t)|}{t^{\beta}}\leq \begin{cases}
C |1-t^{2-\beta}| & \text{ if }\beta\neq 2\\
C(1-\log t) & \text{ if }\beta= 2
\end{cases}\]
and this implies the thesis in case $\beta<2$. When $\beta\geq 2$ instead we have $ |\psi(t)|\leq Ct^2$ for $\beta>2$ and $ |\psi(t)|\leq Ct^{\beta-\varepsilon}$ for $\beta=2$ where $0<\varepsilon<<1$. Inserting these estimates into \eqref{fra2} then we have
\[  
\Big|\int_0^R s^{\beta +1}h(s)\psi(s)ds\Big|\leq \begin{cases}
CR^{\beta+4} & \text{ if } \beta>2\\
CR^{2\beta+1-\varepsilon}& \text{ if } \beta=2
\end{cases}
\]
which, together with \eqref{fra1} gives
\[
\frac{|\psi(t)|}{t^{\beta}}\leq \begin{cases}
C |1-t^{4-\beta}| & \text{ if }\beta\neq 4\\
C(1-\log t) & \text{ if }\beta= 4\\
C(1-t^{1-\varepsilon})& \text{ if }\beta= 2
\end{cases}\]
which implies the thesis when $\beta<4$. We can repeat the procedure. At each step the set of values of $\beta$ at which 
\eqref{psi-in-zero} is satisfied increases by $2$. Then for every value of $\beta$ the thesis follows after a finite number of steps. 
\end{proof}

\

\section{Linearized operator}\label{section:linearizedOperator}
Let $L_{p}: H^2(B)\cap H^1_0(B)\rightarrow L^2(B)$ be the linearized operator at $u_p$, namely
\begin{equation}\label{linearizedOperator} L_{p} v: =   -\Delta v-p|u_p(x)|^{p-1}v.
\end{equation}
It is well known that $L_p$ admits a sequence of eigenvalues which, counting them according to their multiplicity, we denote by
\[\mu_1(p)< \mu_2(p)\leq\ldots\leq\mu_i(p)\leq\ldots,\quad \mu_i(p)\rightarrow +\infty  \mbox{ as }i\rightarrow +\infty,\]
where the first inequality is strict because it is known that $\mu_1(p)$ is simple. 
We also recall their min-max characterization
\begin{eqnarray}\label{CourantCharEigenv}
\mu_i(p) &=& \min_{\substack{
W\subset H^1_{0}(B)\\ dim W=i}}   \max_{\substack{v\in W\\v\neq 0}}\ \ \
R_p[v],\qquad i\in\N_0
\end{eqnarray}
where $R_p[v]$ is the Rayleigh quotient
\begin{equation}\label{Rayleigh}
R_p[v]:=\frac{Q_p(v)}{\int_B v(x)^2 dx}
\end{equation}
and $Q_p: H^1_0(B)\rightarrow \mathbb R$ denotes the quadratic form associated to $L_p$, namely
\begin{equation}\label{formaQuadratica}
Q_p (v):=\int_B \left[|\nabla v(x)|^2 -p|u_p(x)|^{p-1}v(x)^2  \right]dx.
\end{equation}

\

Since $u_p$ is a radial solution to \eqref{problem} we can also consider the subsequence of $(\mu_i(p))_{i\in\mathbb N_0}$ of the radial eigenvalues of $L_p$
(i.e. eigenvalues which are associated to a radial eigenfunction) that we denote by  \[\mu_{i,\rad}(p), \quad i\in\N_0\]
and which are all simple in the space of radial functions. 

For the eigenvalues $\mu_{i,\rad}(p)$ an analogous characterization holds:
\begin{eqnarray}\label{CourantCharEigenvRad}
\mu_{i,\rad}(p) &=& \min_{\substack{
W\subset H^1_{0,rad}(B)\\ dim W=i}}   \max_{\substack{v\in W\\v\neq 0}}\ \ \
R_p[v]
\end{eqnarray}
where $R_p$ is as in \eqref{Rayleigh} and $H^1_{0,\rad}(B)$ is the subspace of the radial functions of $H^1_0(B)$. Moreover it is known that $\mu_{1,\rad}(p)=\mu_1(p).$

\

The {\sl Morse index of $u_p$}, denoted by $m(u_p)$, is the maximal dimension of a subspace $X\subseteq H^1_0(B)$ such that $Q_p(v)<0,  \ \forall v\in X\setminus\{0\}$. Since $B$ is a bounded domain this is equivalent to say that $m(u_p)$ is  the number of the negative eigenvalues of $L_p$ counted with their multiplicity.
\\
The {\sl radial Morse index of $u_p$}, denoted by $m_{\rad}(u_p)$,
is instead the number of the negative radial eigenvalues $\mu_{i,\rad}(p)$ of $L_{p}$.

\

By the results in  \cite{AftalionPacella} we have
\begin{lemma}\label{LemaAftalionPacella} 
For any $p>1$
\[(+\infty >)\ m(u_p)\geq 4.\]
\end{lemma}

Moreover it is well known (see for instance \cite{BartschWeth}, see also \cite{HRS}) the following 

\begin{lemma}\label{LemmaMorseIndexRadiale}
For any $p>1$
\begin{equation}
m_{\rad}(u_p)=2.
\end{equation}
\end{lemma}

\

The previous lemma means that for any $p>1$
\[\mu_{1,\rad}(p)<\mu_{2,\rad}(p)<0\leq \mu_{3,\rad}(p)<\ldots, \] 
next we show that 
\[
\mu_{3,\rad}(p)>0,
\]
namely that the problem
\begin{equation}\label{linearizedProblem}
\left\{\begin{array}{lr}
L_pv=0 & \mbox{ in }B\\
v=0 & \mbox{ on }\partial B
\end{array}
\right.
\end{equation}
doesn't admit nontrivial radial solutions, indeed the following result holds:

\begin{lemma}\label{lemma:radiallyNonDeg} For any  $p>1$ $u_p$ is radially non-degenerate.
\end{lemma}

\begin{proof} 
Given a solution $w_{\alpha}$ for the problem  
\begin{equation}\label{problvaliniz}
\left\{
\begin{array}{lr}
w_{\alpha}''+\frac{1}{r}w_{\alpha}'+|w_{\alpha}|^{p-1}w_{\alpha}=0  &\quad \mbox{ in } (0,T)\\
w_{\alpha}(0)=\alpha>0
\\
w_{\alpha}'(0)=0\\
w_{\alpha} \quad \mbox{ has exactly $1$ zero in $(0,T)$}\\
w_{\alpha}(T)=0
\end{array}
\right.
\end{equation}  
where $T>0$, it is not difficult to see (see \cite{SmollerWassarman}) that  $w_{\alpha}$  is differentiable with respect to $\alpha$ and that it is radially non-degenerate in $(0, T)$  if and only if $\frac{\partial w_{\alpha}}{\partial\alpha}|_{r=T}\neq 0$.
\\
Observe that  $u_p$  solves \eqref{problvaliniz} with $\alpha=u_p(0)>0$ and $T=1$. 
\\
Moreover for any $\alpha>0$ \eqref{problvaliniz} has a unique solution $w_{\alpha}$ which is obtained by scaling $u_p$ as
\[w_{\alpha}(r):=T(\alpha)^{-\frac{2}{p-1}}u_p(\frac{r}{T(\alpha)}),\]
where $T=T(\alpha):=\left(\frac{u_p(0)}{\alpha}\right)^{\frac{p-1}{2}}$. 
\\
Hence it  is immediate to check that  $\frac{\partial w_{\alpha}}{\partial\alpha}|_{r=T(\alpha)}\neq 0$, from which it then follows that $u_p$ is radially non-degenerate.
\end{proof}

\

\section{Morse index and degeneracy of $u_p$}\label{section:generale}

The section is organized as follows: we first consider an auxiliary weighted eigenvalue problem (problem \eqref{problemaPesatoPalla} below), whose main advantage, as we will see,  relies on the fact that it shares with $L_p$ the same spectral properties (see Lemma \ref{lemma:Morse=numeroAutovaloriesatoPalla}) and,  in addition, a classical spectral decomposition into radial and angular part may be applied to it (Lemma \ref{lemma:decomposizionePalla} in the section). 
The study of the auxiliary problem is carried out for any $p>1$, getting  information about the Morse index of the solution $u_p$ (Lemma   \ref{proposition:autovaloriRadiaiGenerale}) and a general characterization of its degeneracy (Proposition \ref{p4.7}).

\

\subsection{An auxiliary weighted eigenvalue problem}\label{se:auxiliary}

\

We consider the auxiliary eigenvalue problem
 \begin{equation} \label{problemaPesatoPalla}
\begin{cases}
\begin{array}{ll}
-\Delta \psi - p|u_p(x)|^{p-1} \psi =\frac{\beta}{|x|^2} \psi\qquad & \text{ in } B\setminus\{0\}, \\
\psi= 0 & \text{ on } \partial B\\
\int_B|\nabla \psi|^2+\frac{\psi^2}{|x|^2}<+\infty,
\end{array} 
\end{cases}
\end{equation} where $\beta\in\mathbb R$ and $p>1$.\\

Observe that, since $p|u_p|^{p-1}\in L^{\infty}(B)$, \eqref{problemaPesatoPalla} belongs to the class of eigenvalue problems which has been studied in \cite{GGN2}, where the  eigenvalues for \eqref{problemaPesatoPalla} have been variationally characterized in the case when they are \emph{negative}.

\

In the following we recall the variational characterization obtained in \cite{GGN2}. In particular they have observed that when the associated Rayleigh quotient is greater or equal than zero there is a compactness problem, but as far as the quotient is strictly negative, the eigenvalues and eigenfunctions maintain the usual properties of the classical ones.
\\
Let us denote by  $\mathcal H$  the closure of $C^{\infty}_0(B)$ with respect to the norm $\|v\|^2_{\mathcal H}=\int_{B}\left( |\nabla v|^2+\frac{v^2}{|x|^2}\right) dx$. Notice that $\mathcal H
%
\subset H^1_0(\Omega)$  and the inclusion is strict (consider for instance the function $w(x)=1-|x|^2$).\\
For $\eta, \xi\in\mathcal H$ we write
\begin{equation}\label{perpH}\eta\perp_{\mathcal H}\xi \quad\Leftrightarrow \quad \int_{B}\frac{\eta\xi}{|x|^2}dx=0.   
\end{equation}

\

Observe that if $\psi,\widetilde{\psi}\in\mathcal H$ are weak solutions to \eqref{problemaPesatoPalla} related respectively to the   eigenvalues $\beta$ and $\widetilde{\beta}$, $\beta\neq\widetilde{\beta}$ then
\begin{equation}\label{ortogonalita_auto_deboli}
\psi\perp_{\mathcal H}\widetilde{\psi}
\end{equation}
(just multiply  \eqref{problemaPesatoPalla} by $\widetilde{\psi}$, the equation \eqref{problemaPesatoPalla} for the eigenvalue $\widetilde{\beta}$  by $\psi$, integrate and subtract).

\

We define 
\begin{eqnarray}\label{CourantCharEigenvpesatipalla1}
\beta_1(p) &:=&   \inf_{\substack{v\in \mathcal H,\,\,v\neq 0}}\ \ \
\widetilde{R_p}[v]
\end{eqnarray}
where  
$\widetilde{R_p}[v]$ is the Rayleigh quotient
\begin{equation}\label{Rayleigh2}
\widetilde{R_p}[v]:=\frac{Q_p(v)}{\int_B \frac{v(x)^2}{|x|^2} dx}
\end{equation}
and $Q_p$ is as in \eqref{formaQuadratica}.

\

From \cite[Proposition 2.1]{GGN2} we know that when $\beta_1(p)<0$ then this infimum  is achieved at a radial function $\psi_1 \in \mathcal H$, $\psi_1>0$ in $B\setminus\{0\}$, which solves 
\begin{equation}\label{problemaPesatoPallaDebole1}
\int_B\nabla\psi_1\nabla v-p|u_p|^{p-1}\psi_1v\, dx= \beta_1(p)\int_B \frac{\psi_1 v}{|x|^2}\, dx, \qquad \forall v\in \mathcal H.
\end{equation}  Moreover $\beta_1(p)$ is simple (in $\mathcal H$). In this case we can then define
\begin{eqnarray}\label{CourantCharEigenvpesatipalla2}
\beta_2(p) &:=&   \inf_{\substack{v\in \mathcal H,\,\,v\neq 0\\ v\perp_{\mathcal H}\psi_1}}\ \ \
\widetilde{R_p}[v]
\end{eqnarray}
which again is achieved when it is negative (see \cite[Proposition 2.3]{GGN2}) and any function $\psi_2\in\mathcal H$  at which $\beta_2(p)$ is achieved solves 
\begin{equation}\label{problemaPesatoPallaDebole2}
\int_B\nabla\psi_2\nabla v-p|u_p|^{p-1}\psi_2v\, dx= \beta_2(p)\int_B \frac{\psi_2 v}{|x|^2}\, dx, \qquad \forall v\in \mathcal H,
\end{equation}  
and by definition  $\psi_1\perp_{\mathcal H}\psi_2$, then $\psi_2$ must change sign. 
\\
\\
 More in general, by iterating, if $\beta_j(p)<0$  and $\psi_j\in\mathcal H$ is a function where it is achieved, for $j=1,\ldots, i-1$, we can define
\begin{eqnarray}\label{CourantCharEigenvpesatipalla3}
\beta_i(p) &:=&   \inf_{\substack{v\in \mathcal H,\,\,v\neq 0\\ v\perp_{\mathcal H}span\{\psi_1,\ldots,\psi_{i-1}\}}}\ \ \
\widetilde{R_p}[v],\qquad i\in\N, \ i\geq 2
\end{eqnarray}
 which (again \cite[Proposition 2.3]{GGN2}) is achieved if it is negative and, in such a case, any function $\psi_i\in\mathcal H$  at which $\beta_i(p)$ is achieved solves
 \begin{equation}\label{problemaPesatoPallaDebolei}
\int_B\nabla\psi_i\nabla v-p|u_p|^{p-1}\psi_iv\, dx= \beta_i(p)\int_B \frac{\psi_i v}{|x|^2}\, dx, \qquad \forall v\in \mathcal H,
\end{equation}  and changes sign. 

\

Similarly, restricting to  the subspace
$\mathcal H_{\rad}$ of the radial functions of $\mathcal H$, we can also define:
\begin{eqnarray}\label{CourantCharEigenvpesatipallaRADIALE1}
\beta_{1,\rad}(p) &:=&    \inf_{\substack{v\in \mathcal H_{\rad},\,\,v\neq 0}}\ \ \
\widetilde{R_p}[v]  \ (= \beta_1(p))
\end{eqnarray}
and, if $\beta_{j,\rad}(p)<0$ for $j=1,\ldots, i-1$
\begin{eqnarray}\label{CourantCharEigenvpesatipallaRADIALE}
\beta_{i,\rad}(p) &:=&   \inf_{\substack{v\in \mathcal H_{\rad},\,\,v\neq 0\\ v\perp_{\mathcal H}span\{\phi_{1},\ldots,\phi_{i-1}\}}}\ \ \
\widetilde{R_p}[v],\qquad i\in\N, \ i\geq 2
\end{eqnarray}
where $\phi_j\in\mathcal H_{\rad}$ is the function where $\beta_{j,\rad}(p)$ is achieved for  $j=1,\ldots, i-1$ (observe that $\phi_1=\psi_1$) and solve
\begin{equation}\label{problemaPesatoPallaDeboleiRadiale}
\int_B\nabla\phi_j\nabla v-p|u_p|^{p-1}\phi_jv\, dx= \beta_{j,\rad}(p)\int_B \frac{\phi_j v}{|x|^2}\, dx, \qquad \forall v\in \mathcal H_{\rad}.
\end{equation}

\

\begin{lemma}[Variational characterization \cite{GGN2}] 
\label{valoriVariaz=AutovaloriesatoPalla}
The negative eigenvalues (resp. negative radial eigenvalues) of problem \eqref{problemaPesatoPalla} coincide with the negative numbers  $\beta_i(p)$'s  defined in \eqref{CourantCharEigenvpesatipalla1}-\eqref{CourantCharEigenvpesatipalla3} (resp. with the numbers $\beta_{i,\rad}(p)$'s defined in \eqref{CourantCharEigenvpesatipallaRADIALE1}-\eqref{CourantCharEigenvpesatipallaRADIALE}). Moreover, by \eqref{ortogonalita_auto_deboli}, the corresponding 
eigenfunctions, which solve \eqref{problemaPesatoPalla}, are in $\mathcal H$ and can be  chosen to be orthogonal in the sense of \eqref{perpH}. 
\end{lemma}

The following relation holds between the Morse index of $u_p$ and the number of negative eigenvalues of the weighted problem \eqref{problemaPesatoPalla}:
\begin{lemma}[\cite{GGN2}, Lemma 2.6]\label{lemma:Morse=numeroAutovaloriesatoPalla} 
The  Morse index (resp. radial Morse index) of $u_p$ coincides with the number of negative eigenvalues (resp. negative radial eigenvalues) of problem \eqref{problemaPesatoPalla} counted according to their multiplicity.
\end{lemma}

As a consequence we have:

\begin{lemma}\label{2AutovRadNegativiPesatoPalla}
For any  $p>1$
\[ \beta_{1,\rad}(p)<\beta_{2,\rad}(p)<0.\] 
Moreover $\beta_{3,\rad}(p)=0$ and it is not an eigenvalue for \eqref{problemaPesatoPalla}. 
\end{lemma}
\begin{proof}

The first statement is a consequence of Lemma \ref{LemmaMorseIndexRadiale} and  Lemma \ref{lemma:Morse=numeroAutovaloriesatoPalla}.
\\ Observe that the value $\beta_{3,\rad}(p)$ is well defined by \eqref{CourantCharEigenvpesatipallaRADIALE}, being both $\beta_{1,\rad}(p)$ and $\beta_{2,\rad}(p)$ negative, moreover $\beta_{3,\rad}(p)\geq 0$ from Lemma \ref{valoriVariaz=AutovaloriesatoPalla} and Lemma \ref{lemma:Morse=numeroAutovaloriesatoPalla}, since 
$m_{\rad}(u_p)=2$ by Lemma \ref{LemmaMorseIndexRadiale}.
In particular even if $\beta_{3,\rad}(p)=0$ it cannot be  an eigenvalue for \eqref{problemaPesatoPalla} because $\mathcal H\subset H^1_0(B)$ and $u_p$ is radially nondegenerate by  Lemma \ref{lemma:radiallyNonDeg}.\\
To show that $\beta_{3,\rad}(p)=0$ we let $\phi_j\in \mathcal H_{\rad}$ be the function where $\beta_{j,\rad}(p)$ is achieved for  $j=1,2$, we choose the test functions
\begin{equation*}
\eta_{\ep}(x):=
\begin{cases}
1-|x|&\hbox{if }\ep\le|x|\le1\\
\frac{2(1-\ep)}\ep|x|+\ep-1&\hbox{if }\frac\ep 2\le|x|\le\ep\\
0&\hbox{if }|x|\le\frac\ep 2
\end{cases}
\end{equation*}
defined for $0<\ep<1$ and we let 
\[\widetilde\eta_\ep(x):=\eta_\ep(x)-a_\ep\phi_1-b_\ep\phi_2\]
where $a_\ep,b_\ep\in \R$ are given by
\[
a_\ep:=\frac{\int_B \frac{\eta_\ep\phi_1} {|x|^{2}}}{\int_B \frac{\phi_1^2}{|x|^{2}} }\quad , \quad
b_\ep:=\frac{\int_B \frac{\eta_\ep\phi_2}{|x|^{2}}}{\int_B \frac{\phi_2^2}{|x|^{2}}}\]
so that $\widetilde\eta_\ep$ is orthogonal in the sense of \eqref{perpH} to $\phi_j$, $j=1,2$ for any $\ep\in (0,1)$.\\
Moreover observe that by our choice of the test functions $\eta_\ep$ there exists $C=C_p>0$ such that
\begin{equation}
\label{etaepUnifBounded}
\int_B \left(|\nabla \eta_\ep|^2-p|u_p|^{p-1}\eta_\ep^2\right)\leq C,
\end{equation}
for any $\ep\in (0,1)$.\\
Since $\beta_{j,\rad}(p)<0$ for $j=1,2$, by Proposition \ref{p2.2} we have that 
\begin{equation}\label{numero}
 \phi_j(r)=O\left( r^{\sqrt {-\beta_{j,\rad}(p)}}\right) \quad \text{ as }r\rightarrow 0.
\end{equation}
This last estimate together with the definition of $\eta_\ep$ then implies that 
\[
\begin{split}
\int_0^1 \frac { \eta_\ep\phi_j}r\, dr=&\frac {2(1-\ep)} {\ep}\int_{\frac {\ep}2}^{\ep} \phi_j(r)\, dr+(\ep-1)\int_{\frac {\ep}2}^{\ep}\frac { \phi_j(r)}r\, dr+\int_{\ep}^1\frac{(1-r)\phi_j(r)}r\, dr\\
&\hspace{-4em}\overset{\eqref{numero}} {\leq} C+O\left(\ep^{\sqrt {-\beta_{j,\rad}(p)}}\right)\leq C
\end{split}
\]
so that $a_\ep$ and $b_\ep$ are uniformly bounded.

From \eqref{CourantCharEigenvpesatipallaRADIALE} and the orthogonality between $\widetilde\eta_\ep$ and $\phi_j$, $j=1,2$ then  $\beta_{3,\rad}(p)\leq    \widetilde{R_p}[\widetilde\eta_\ep]$
where 
\begin{equation}\label{R_peta-tilde}
\widetilde{R_p}[\widetilde\eta_\ep]=\frac{Q_p(\widetilde\eta_\ep )}{\int_B \frac{\widetilde\eta_\ep ^2}{|x|^2} dx}.
\end{equation}
An easy computation shows that 
\begin{eqnarray*}
&Q_p(\widetilde\eta_\ep )=&\int_B \left(|\nabla \eta_\ep|^2-p|u_p|^{p-1}\eta_\ep^2\right)+a_\ep^2\int_B\left(|\nabla \phi_1|^2 -p|u_p|^{p-1}\phi_1^2\right)\\
&&+b_\ep^2\int_B\left(|\nabla \phi_2|^2 -p|u_p|^{p-1}\phi_2^2\right)-2a_\ep\int_B\left(\nabla \eta_\ep\cdot \nabla \phi_1-p|u_p|^{p-1}\eta_\ep\phi_1\right)\\
&&-2b_\ep\int_B\left(\nabla \eta_\ep\cdot \nabla \phi_2-p|u_p|^{p-1}\eta_\ep\phi_2\right) -2a_{\ep}b_\ep\int_B\left(\nabla \phi_1\cdot \nabla \phi_2-p|u_p|^{p-1}\phi_1\phi_2\right)
\end{eqnarray*}
and, using that $\phi_j$, $j=1,2$ solves \eqref{problemaPesatoPallaDeboleiRadiale}, that $\phi_1\perp_{\mathcal H}\phi_2$  and recalling the definition of $a_\ep,b_\ep$, we then get
\[Q_p(\widetilde\eta_\ep )= \int_B \left(|\nabla \eta_\ep|^2-p|u_p|^{p-1}\eta_\ep^2\right)-a_\ep^2\beta_{1,\rad}(p)\int_B\frac{\phi_1^2}{|x|^2}-b_\ep^2\beta_{2,\rad}(p)\int_B\frac{\phi_2^2}{|x|^2}.\]
The last equality, together with \eqref{etaepUnifBounded} and the boundedness of $a_\ep,b_\ep$ implies that 
\[Q_p(\widetilde\eta_\ep )\leq C\]
for any $\ep\in (0,1)$.
Finally, using again the definition of  $a_\ep,b_\ep$ we have
\begin{eqnarray*}
\int_B \frac{\widetilde\eta_\ep ^2}{|x|^2} dx &= &\int_B \frac{\eta_\ep ^2}{|x|^2}
-a_\ep^2 \int_B\frac{\phi_1^2}{|x|^2}-b_\ep^2\int_B\frac{\phi_2^2}{|x|^2}\\
&\overset{\mbox{$a_\ep,  b_\ep$ bounded}}{\geq} & \int_B \frac{\eta_\ep ^2}{|x|^2}-C\\
&= & 2\pi \left(\frac{(1-\ep)^2}{\ep^2} 
\int_{\frac {\ep}2}^{\ep}\frac {(2r-\ep)^2}r \, dr+\int_{\ep}^1\frac {(1-r)^2}r\, dr\right)-C\\
&= & 2\pi\left(-\log \ep +\ep \log 2+(1-\ep)(\ep-2)\right)-C\\
&=& -2\pi \log \ep \left(1+o(1)\right) \quad \text{ as }\ep \rightarrow 0.
\end{eqnarray*}
The conclusion then follows using \eqref{R_peta-tilde} and $0\leq \beta_{3,\rad}(p)\leq    \widetilde{R_p}[\widetilde\eta_\ep]$.
\end{proof}

\

Here and in the following we denote by  $\alpha_k$, $k\in\mathbb N$ the {\sl spherical harmonics} in dimension $2$, namely
 the
homogeneous harmonic polynomials of degree $k$ considered on the unit sphere $S^1\subset\mathbb R^2$.
They can be written explicitly, using the polar coordinates $x=(r\cos\theta, r\sin\theta)$ 
\begin{equation}
\label{espressioneArmonicheSferiche}
\alpha_k(\theta)=\left\{\begin{array}{ll}c & k=0
\\c_1\cos (k\theta)+c_2\sin (k\theta) & k= 1,2,3,\ldots\end{array} \right.
\end{equation}
for $c,c_1,c_2\in\mathbb R$.

Recall that the set $(\alpha_k)_{k\in\mathbb N}$ is a complete orthogonal system for $L^2(S^1)$, hence 
any function $v\in L^2(B)$ can be written as
\begin{equation}\label{spectralDec1}
v(r,\theta)=\sum_{k=0}^{+\infty}h_k(r)\alpha_k(\theta)
\end{equation}
where 
\begin{equation}\label{spectralDec2}
h_k(r):=\int_0^{2\pi}\alpha_k(\theta)v(r,\theta)d\theta, \qquad r\in (0,1).
\end{equation}
Moreover if $v(r,\theta)$ is continuous in the origin, then $2 \pi c v(0)=h_0(0)$ (where $c$ is the constant in \eqref{espressioneArmonicheSferiche}) and 
\begin{equation}\label{h_kInZero} h_k(0)=0, \quad \forall k\geq1.
\end{equation}

Recall also that the  eigenvalues of the  Laplace-Beltrami operator $-\Delta_{S^1}$ on the unit sphere $S^1$ are the numbers $k^2$, $k\in\mathbb N$, that they have multiplicity  $1$ if $k=0$ and multiplicity  $2$ if $k\geq 1$,  and that the spherical harmonics 
$\alpha_k$ are the eigenfunctions associated to the eigenvalue $k^2$.

\

 For the negative eigenvalues  of \eqref{problemaPesatoPalla} we  then have the
following spectral decomposition into radial and angular part, where the angular part is given by the eigenvalues of $-\Delta_{S^1}$: 
\begin{lemma}\label{lemma:decomposizionePalla} 
Let $p>1$. For any $i=1,\ldots, m(u_p)$ there exists $(j,k)\in \{1,2\}\times\mathbb N$ ($(j,k)$ depending also on $p$)
such that  
\begin{equation}\label{spectralDecomposition}
\beta_i(p)= \beta_{j,\rad}(p)+k^2. 
\end{equation} 
Conversely for every $(j,k)\in \{1,2\}\times\mathbb N$ such that  $\beta_{j,\rad}(p)+k^2<0$ there exists $i\in\{1,\ldots, m(u_p)\}$ ($i$ depending also on $p$) for which \eqref{spectralDecomposition} holds.\\

Moreover the eigenspace associated to each negative eigenvalue $\beta(p)$ of \eqref{problemaPesatoPalla}
is spanned by the functions
\begin{equation}\label{eigenspaceSpectralDec}
\phi_j(r)\cos (k\theta)\ \mbox{ and }\ \phi_j(r)\sin (k\theta),\quad \forall\, (j,k) \mbox{ such that }\  \beta_{j,\rad}(p)+k^2=\beta(p),
\end{equation}
where $\phi_j$ is the radial eigenfunction to \eqref{problemaPesatoPalla} associated to the radial eigenvalue $\beta_{j,\rad}(p)$ (which is simple in the space of radial functions).
%
\end{lemma} 
\begin{proof} 
{\bf Step 1.}
{\sl We show the first statement.}
\\
By Lemma \ref{valoriVariaz=AutovaloriesatoPalla} and Lemma \ref{lemma:Morse=numeroAutovaloriesatoPalla} the value $\beta_i(p)$, for any $i=1,\ldots, m(u_p)$,  is a (negative) eigenvalue for problem  \eqref{problemaPesatoPalla} and so there exists a function $\psi\neq 0$ which satisfies \eqref{problemaPesatoPalla} with $\beta=\beta_i(p)$. 
Decomposing $\psi$ along spherical harmonics (see \eqref{spectralDec1}, \eqref{spectralDec2}), we write
\[\psi(r,\theta)=\sum_{k=0}^{+\infty} h_k(r)\alpha_k(\theta)\]
where 
\begin{equation}\label{decompos}
h_k(r):=\int_0^{2\pi}\alpha_k(\theta)\psi(r,\theta)d\theta, \qquad r\in (0,1).
\end{equation}
Then, since $\psi\neq 0$ and $(\alpha_k)_k$ is a complete orthogonal system for $L^2(S^1)$, it follows that $h_k\neq 0$ for some $k\in\mathbb N$, moreover
it  satisfies 

\begin{eqnarray*}
-h_k''-\frac{1}{r}h_k'&=& 
\int_0^{2\pi}\left(-\psi_{rr}-\frac{1}{r}\psi_r\right)\alpha_k \,d\theta
\\
&=&\int_{0}^{2\pi} \left(-\Delta\psi+\frac{1}{r^2}\Delta_{S^1}\psi\right)\alpha_k\,d\theta
\\
&=& p |u_p|^{p-1}h_k +\frac{\beta_i(p)}{r^2} h_k+
\frac{1}{r^2}\int_0^{2\pi} \left(\Delta_{S^1}\psi\right) \alpha_k\,d\theta.
\end{eqnarray*}
Integrating the last term by parts we get
\begin{equation}\label{A}
\begin{cases}
-h_k''-\frac 1r h_k'-p|u_p|^{p-1}h_k =\frac{\beta_i(p)-k^2}{r^2} h_k & \text{ in }(0,1)\\
h_k(1)=0,
\end{cases}
\end{equation}
where $\beta_i(p)-k^2\leq \beta_i(p)<0$. Next we show that it satisfies also the condition
\begin{equation} \label{radsphint}
\int_0^1 r(h_k')^2+\frac{h^2_k}{r}<+\infty.
\end{equation}


Indeed using \eqref{decompos} we get 
\begin{eqnarray}\label{primop}
 \int_0^1 \frac{h_k(r)^2}{r}\; dr
&=&\int_0^1 \frac 1 r 
\left(\int_{0}^{2\pi}\alpha_k(\theta)\psi(r,\theta) \;d \theta\right)^2 dr\\\nonumber
&\overset{\text{Jensen ineq.}}{\leq}& \int_0^1\frac 1 r\int_{0}^{2\pi}\alpha^2_k(\theta)\psi^2(r,\theta) \; d \theta\, dr\\
\nonumber
&\overset{\text{ $\alpha_k$ are bounded}}{\leq}& C\int_0^1\int_0^{2\pi} \frac{\psi^2(r,\theta)}{r^2} r\; dr\, d\theta
= C\int_B \frac{\psi^2(x)}{|x|^2}<\infty,
\end{eqnarray}
where last estimate follows from \eqref{problemaPesatoPalla}. In the same way we obtain 
\begin{eqnarray}\label{secondop}
 \int_0^1r\left(h'_k(r)\right)^2dr &=&\int_0^1 r \left(\int_{0}^{2\pi}\alpha_k(\theta)\frac {\partial \psi(r,\theta)}{\partial r}d \theta\right)^2dr\\\nonumber
&\leq & C\int_0^1\int_0^{2\pi}r \left|\frac {\partial \psi(r,\theta)}{\partial r}\right|^2 dr d\theta\leq C\int_B|\nabla \psi(x)|^2 dx<\infty,
\end{eqnarray} 
showing \eqref{radsphint}.
\\
By Lemma \ref{valoriVariaz=AutovaloriesatoPalla}, Lemma \ref{lemma:Morse=numeroAutovaloriesatoPalla} and Lemma \ref{2AutovRadNegativiPesatoPalla}  problem \eqref{A}-\eqref{radsphint} 
admits only two negative eigenvalues which coincide with $\beta_{1,\rad}(p)$ and $\beta_{2,\rad}(p)$. Then \eqref{A}-\eqref{radsphint} has a nontrivial solution $h_k$ (related to a negative eigenvalue) if and only if $\beta_{j,\rad}(p)=\beta_i(p)-k^2$ for some $j=1,2$. This ends the proof of the existence of $(j,k)\in \{1,2\}\times \mathbb N$  which satisfies \eqref{spectralDecomposition}.
 
\

{\bf Step 2.}
{\sl We show the converse statement.}
\\
Let $(j,k)\in \{1,2\}\times\mathbb N$ be such that  $\beta_{j,\rad}(p)+k^2<0$, let   $\phi_j$ be an eigenfunction associated to the radial eigenvalue $\beta_{j,\rad}(p)$ (which is simple in the space of the radial functions) and $\alpha_k$ be an eigefunction of $-\Delta_{S^1}$ associated to the eigenvalue $k^2$ (see \eqref{espressioneArmonicheSferiche}). 
Then easy computation shows that the number $\beta_{j,\rad}(p)+k^2$ is a negative eigenvalue for the weighted problem \eqref{problemaPesatoPalla} with eigenfunction given by 
\begin{equation} 
\label{moltiplicazione}
\psi_{j,k}(x):=\phi_j(|x|)\alpha_k(\frac{x}{|x|}). 
\end{equation}
As a consequence, by Lemma \ref{valoriVariaz=AutovaloriesatoPalla} and Lemma \ref{lemma:Morse=numeroAutovaloriesatoPalla}, there exists $i\in\{1,\ldots, m(u_p)\}$ for which \eqref{spectralDecomposition} holds.

\

{\bf Step 3.}
{\sl We prove that the eigenspace of a negative eigenvalue $\beta(p)$ of problem \eqref{problemaPesatoPalla} is spanned by the functions in  \eqref{eigenspaceSpectralDec}.}
\\
Let $m\in\N_0$ be the multiplicity of $\beta(p)$, so there exists an index $\ell\in\N$, $\ell\geq 1$ such that
 \[\beta(p)=\ \beta_{\ell}(p)=\beta_{\ell+1}(p)=\cdots\beta_{\ell+m-1}(p)\ < \beta_{\ell+m}(p)\]
 and if $\ell \geq 2$ also \[\beta_{\ell-1}(p)<\ \beta(p)\] ($m$ is the number of subsequent indexes $i$ in our notation). \\
By {\bf Step 1.} for every $i=\ell,\ldots, \ell+m-1$ there exists a couple $(j,k)\in \{1,2\}\times\N$ for which \eqref{spectralDecomposition} holds (some of the couples may coincide).
\\
Then considering the set
\[\mathcal I:=\{(j,k)\in\{1,2\}\times\N\ :\   \beta_i(p)=\beta(p)=\beta_{j,\rad}(p)+k^2,\quad i=\ell,\ldots \ell +m\},\] as seen in {\bf Step 2.} all the functions in 
 \eqref{moltiplicazione} with $(j,k)\in\mathcal I$ are eigenfunctions for \eqref{problemaPesatoPalla}. 
Observe that since $\beta_{j,\rad}(p)$ is simple in the space of radial functions and $\alpha_k$ are the functions in \eqref{espressioneArmonicheSferiche} one obtains all the functions in  \eqref{eigenspaceSpectralDec}, which are linearly independent. \\
Last we prove by contradiction that the eigenspace of $\beta(p)$ consists only of the functions in \eqref{eigenspaceSpectralDec}. So let us assume  the existence of another eigenfunction $\psi\neq 0$, 
\begin{equation}\label{orto}
\psi\perp_{\mathcal H} \mbox{ span}\left\{\phi_j(r)\cos (k\theta),\ \phi_j(r)\sin (k\theta)\ :\ (j,k) \in\mathcal I\right\},
\end{equation} 
then similarly as in {\bf Step 1.} we can write
\begin{equation}
\label{expa}
\psi(r,\theta)=\sum_{s=0}^{+\infty} h_s(r)\alpha_s(\theta)
\end{equation}
where
\[
h_s(r):=\int_0^{2\pi}\alpha_s(\theta)\psi(r,\theta)d\theta, \qquad r\in (0,1).
\]
Since $\psi\neq 0$ then there exists $s\in\N$ such that $h_s\neq 0$. Then, as in {\bf Step 1.} we can prove that for any $s$ such that $h_s\neq 0$  there exists $t_s\in \{1,2\}$ such that 
\begin{equation}\label{PerAssurdo}
\beta(p)=\beta_{t_s,\rad} +s^2\quad \mbox{ and }\quad h_s=\phi_{t_s}.
\end{equation}
As a consequence \eqref{expa} becomes 
\[\psi(r,\theta)=\sum_{s=0, \ h_s\neq 0}^{+\infty} \phi_{t_s}(r)\alpha_s(\theta)
\]
and so the orthogonality condition \eqref{orto} gives 
\[0=\sum_{s=0}^{\infty} \int_0^{1}\frac{\phi_{t_s}\phi_j}{r} dr\ \int_{0}^{2\pi}\alpha_s\alpha_k d\theta=\sum_{s=0, \ h_s\neq 0}^{+\infty} \delta_{t_s,j}\delta_{s,k}, \qquad \forall (j,k)\in\mathcal I.\]
As a consequence, for any $(j,k)\in\mathcal I$ either $s\neq k$ or if $s=k$ then necessarily $t_s\neq j$, namely the couple
$(t_s,s)\not\in\mathcal I$.
Since \eqref{PerAssurdo} holds this contradicts the definition of the set $\mathcal I$.
%
%
\end{proof}

\

\subsection{Morse index and characterization of the degeneracy of $u_p$}\label{subsec:mo}

\

In the next result we estimate the two negative radial eigenvalues of the auxiliary weighted eigenvalue problem \eqref{problemaPesatoPalla} from above and from below by consecutive eigenvalues of $-\Delta_{S^1}$. The proof is based on the approximation of problem \eqref{problemaPesatoPalla} by a family of  weighted eigenvalue problems in annuli already studied in \cite{DeMarchisIanniPacellaMathAnn}, in particular we exploit some previous estimates in \cite{DeMarchisIanniPacellaMathAnn} related to the negative radial eigenvalues to this family of approximating auxiliary problems. As a consequence of our estimates we  also get that the Morse index of $u_p$ is even for any $p>1$ and uniformly bounded in $p$. Moreover the estimate of the two negative radial eigenvalues of \eqref{problemaPesatoPalla} is the starting point to characterize the degeneracy of $u_p$, this last result is contained in Proposition \ref{p4.7} at the end of the section.
\begin{lemma} \label{proposition:autovaloriRadiaiGenerale}
\begin{equation}\label{nu2controllato}
-1\leq \beta_{2,\rad}(p)<0 \quad \forall p>1.
\end{equation}
For any $p>1$ there exists a unique $j=j(p)\in\mathbb N$, $j\geq 2$ such that
\begin{equation}\label{nu1controllato}
-j^2\leq \beta_{1,\rad}(p)<-(j-1)^2
\end{equation} 
and
\begin{equation}\label{MorsePari} 
m(u_p)=2j
\end{equation}
Moreover $j(p)\leq C$ for any $p>1$, where the constant $C>0$ does not depend on $p$.
\end{lemma}

\begin{proof}
Let us consider the set 
\begin{equation}\label{def:anello}
A_n:=\{x\in \mathbb R^2\ :\ \frac{1}{n}<|x|<1 \}, \quad n\in \mathbb N_0
\end{equation}
and the weighted eigenvalue problem
 \begin{equation} \label{problemaPesatoAnello}
\begin{cases}
\begin{array}{ll}
-\Delta \psi - p|u_p(x)|^{p-1} \psi =\frac{\beta}{|x|^2} \psi\qquad & \text{ in } A_n, \\
\psi= 0 & \text{ on } \partial A_n
\end{array} 
\end{cases}
\end{equation} 
and let us denote by  
 \[
 \beta_i^n(p), \quad \beta_{i,\rad}^n(p), \quad i\in\N_0
 \] 
its eigenvalues, counted with their multiplicity, and the radial eigenvalues, which are simple in the space of radial functions, respectively. 

\

Clearly the following characterizations hold
\begin{eqnarray}\label{defmutilde1n}
\beta_i^n(p)&=&
\min_{\substack{
V\subset H^1_{0}(A_n)\\ dim V=i}}   \max_{\substack{v\in V\\v\neq 0}} \widetilde{R^n_p}[v]
\end{eqnarray} 
\begin{eqnarray}\label{defbetatilde1n}
\beta_{i,\rad}^n(p)&=&
\min_{\substack{
V\subset H^1_{0,rad}(A_n)\\ dim V=i}}   \max_{\substack{v\in V\\v\neq 0}} \widetilde{R^n_p}[v]
\end{eqnarray} 
where $\widetilde{R^n_p}$ is the Rayleigh quotient
\begin{equation}\label{Rayleigh3}
\widetilde{R^n_p}[v]:=\frac{Q^n_p(v)}{\int_{A_n} \frac{v(x)^2}{|x|^2} dx}
\end{equation}
and $Q^n_p(v):=\int_{A_n}\left(|\nabla v(x)|^2-p|u_p(x)|^{p-1}v(x)^2\right) dx$.

\

{\bf Step 1.}
{\sl We show that for any $p>1$ there exists a unique $j=j(p)\in\mathbb N$, $j\geq 2$ and $n_p\in\mathbb N_0$ such that
\begin{equation}\label{morsePari}m(u_p)=2j,
\end{equation} 
and for $n\geq n_p$
\begin{equation}\label{posizionebetatilden1}
-j^2\leq\beta_{1,\rad}^n(p)<-(j-1)^2
\end{equation}
\begin{equation}\label{posizionebetatilden2}
-1< \beta_{2,\rad}^n(p)<0.
\end{equation}
}
As already proved in \cite[Proposition 4.3]{DeMarchisIanniPacellaMathAnn}  there exists $n_p'\in\mathbb N_0$ such that
\begin{equation}\label{morseradCoincide}
2\overset{Lemma \ \ref{LemmaMorseIndexRadiale}}{=} m_{\rad}(u_p)=\#\{\mbox{negative eigenvalues $\beta_{i,\rad}^n(p)$}\}, \quad\mbox{ for }n\geq n_p',
\end{equation}
namely
\begin{equation}
\label{ahah}\beta_{1,\rad}^n(p)< \beta_{2,\rad}^n(p)<0\leq \beta_{3,\rad}^n(p)<\ldots, \quad\mbox{ for }n\geq n_p',
\end{equation}
where the strict inequalities are due to the fact that the radial eigenvalues are simple in the space of the radial functions.\\ 
From \cite[Proposition 4.5]{DeMarchisIanniPacellaMathAnn}  we also know that for any $p>1$ there exists $n_p''\in\mathbb N_0$ such that 
\begin{equation}
\label{dip}
\beta_{2,\rad}^n(p)> -1,\ \ \mbox{ for any }n\geq  n_p''.
\end{equation}
Hence \eqref{posizionebetatilden2} follows immediately from  \eqref{ahah} and \eqref{dip}. In order to conclude the proof observe that from \cite[Proposition 4.3]{DeMarchisIanniPacellaMathAnn} we  also know that
\begin{equation}\label{morseCoincide}
4\overset{ Lemma \ \ref{LemaAftalionPacella}}{\leq } m(u_p)=\#\{\mbox{negative eigenvalues $\beta_i^n(p)$}\}, \quad\mbox{ for }n\geq n_p'.
\end{equation}
We show that, as a consequence of \eqref{morseCoincide}, using \eqref{ahah}-\eqref{dip} and a spectral decomposition for the eigenvalues $\beta_i^n(p)$ in $A_n$, necessarily also \eqref{morsePari} and  \eqref{posizionebetatilden1}  hold. Indeed recall that for the eigenvalues in $A_n$ the spectral decomposition holds:
\begin{equation}\label{spectralDecompositionAnello}
\beta_i^n(p)= \beta_{j,\rad}^n(p)+k^2\quad i,j \in\mathbb N_0,\ k\in\mathbb N, 
\end{equation} 
where as before $k^2$ are  the  eigenvalues of the  Laplace-Beltrami operator on the unit sphere $S^1$, and the eigenfunctions $\psi$ associated to the eigenvalue  $\beta_i^n(p)$ may be obtained by multiplying the spherical harmonics $\alpha_k$ associated to $k^2$ (given in \eqref{espressioneArmonicheSferiche}) with the radial $j$-th eigenfunction $\phi_j^n$ for problem \eqref{problemaPesatoAnello}, similarly as we already did in \eqref{moltiplicazione}:
\begin{equation}
\label{moltiplicazione2}
\psi(x)=\phi_j^n(|x|)\alpha_k(\frac{x}{|x|}). 
\end{equation}

\

By \eqref{spectralDecompositionAnello} and \eqref{ahah} it follows that the modes $k$ that contributes to the Morse index of $u_p$ are those such that 
\begin{equation}\label{spectralDecompositionAnelloModiCheContano}
\beta_i^n(p)= \beta_{j,\rad}^n(p)+k^2 <0,\qquad j=1,2.
\end{equation}
The case $j=2$ in \eqref{spectralDecompositionAnelloModiCheContano} is possible only when $k=0$ by \eqref{dip}. Hence by \eqref{moltiplicazione2} and recalling  that there is only $1$ spherical harmonic for $k=0$ (see \eqref{espressioneArmonicheSferiche}) we get only $1$ contribution to the Morse index in this case.
\\
The case $j=1$ gives instead $1$ contribution (for $k=0$) and moreover,  by  \eqref{morseCoincide},  it must also give other contribution ($k\geq 1$). As a consequence  \eqref{posizionebetatilden1} holds.
Hence by \eqref{moltiplicazione2} and recalling  that there are two spherical harmonics for $k\geq 1$, and only $1$ for $k=0$ (see \eqref{espressioneArmonicheSferiche}) we get in this case
that the total contribution of $\beta_{1,\rad}^n(p)$ to the Morse index is then $2(j-1)+1$. 
\\
Summing up all the contributions from both $j=1$ and $j=2$ we get \eqref{morsePari}.

\

{\bf Step 2.}
{\sl We show that for any $p>1$ the sequence  $(\beta_{i,\rad}^n(p))_n$ is monotone non-increasing and
\[\beta_{i,\rad}(p)=\lim_{n\rightarrow +\infty}\beta_{i,\rad}^n(p)=\inf_n \beta_{i,\rad}^n(p), \quad i=1,2.\]
}
The monotonicity of $(\beta_{i,\rad}^n(p))_n$  follows by the variational characterization \eqref{defbetatilde1n} of these eigenvalues and by the
canonical embeddings $H^1_0(A_{n})\subset H^1_0(A_{n+1})$, $\forall n\in\N_0$.
\\
By the monotonicity of $(\beta_{i,\rad}^n(p))_n$ we can define the values
 \begin{equation}\label{deflim}\delta_i(p):=\lim_{n\rightarrow +\infty} \beta_{i,\rad}^n(p)=\inf_n \beta_{i,\rad}^n(p)\overset{\mbox{{\scriptsize{{\bf Step 1}}}}}{<}0, \quad i=1,2.
\end{equation}
Then the proof of {\bf Step 2.}   consists in proving that 
\begin{equation}\label{tesiSTEP2}
\delta_i(p)=\beta_{i,\rad}(p), \quad i=1,2.
\end{equation}
Let $\widetilde{\phi _i^n}$ be the radial eigenfunction of problem  \eqref{problemaPesatoAnello} corresponding to the radial eigenvalue $\beta_{i,\rad}^n(p)$ and normalized so that
\begin{equation}
\label{normalizzazione}
\int_{A_n} \left(\widetilde{\phi _i^n}\right)^2 =1.
\end{equation}
Then, extending $ \widetilde{\phi _i^n}$ to zero in $B\setminus A_n$, from \eqref{normalizzazione} we have
\begin{equation}\label{deltaip}\int _{B} \frac{\left(\widetilde{\phi _i^n}\right)^2}{|x|^2}>\int_{B} \left(\widetilde{\phi _i^n}\right)^2 =1.
\end{equation}
By {\bf Step 1.} we know in particular that 
$\beta_{i,\rad}^n(p)<0$, so from \eqref{problemaPesatoAnello} we get
\begin{equation}\label{betaiequ}
\int_{B} | \nabla \widetilde{\phi _i^n}|^2-p \int_{B}|u_p|^{p-1} \left(\widetilde{\phi _i^n}\right)^2=\beta_{i,\rad}^n(p)\int _{B} \frac{\left(\widetilde{\phi _i^n}\right)^2}{|x|^2}< 0,
\end{equation}
from which it follows that the sequence $(\widetilde{\phi _i^n})_n$ is bounded in $H^1_0(B)$, indeed 
\begin{equation}\label{c}
\int_{B} | \nabla \widetilde{\phi _i^n}|^2<p \int_{B}|u_p|^{p-1} \left(\widetilde{\phi _i^n}\right)^2\leq p\|u_p\|_{\infty}^{p-1}\int_{B}\left(\widetilde{\phi _i^n}\right)^2\overset{\eqref{normalizzazione}}{\leq} C_p.
\end{equation} 
Hence, up to a subsequence, $\widetilde{\phi _i^n}\rightarrow \widetilde{\phi_i}$ as $n\rightarrow +\infty$ weakly in $H^1_0(B)$, strongly in $L^2(B)$ and pointwise a.e. in $B$ and then as a consequence
\begin{equation}
\label{stainH10}
\int_{B} | \nabla {\widetilde{\phi_i}}|^2\leq \liminf_{n\rightarrow +\infty}\int_{B} | \nabla \widetilde{\phi _i^n}|^2\overset{\eqref{c}}{\leq} C_p
\end{equation}
and
\begin{equation}\label{convergenzapartel2}
p\int_{B}|u_p|^{p-1} \left({\widetilde{\phi_i}}\right)^2=\lim_{n\rightarrow +\infty} p \int_{B}|u_p|^{p-1} \left(\widetilde{\phi _i^n}\right)^2.
\end{equation}
$\widetilde{\phi_i}$ is radial and moreover, since $\int_{B} \left({\widetilde{\phi_i}}\right)^2 =1$, we have \begin{equation}\label{nonZ}\widetilde{\phi_i}\neq 0.
\end{equation} 
Furthermore, since $\widetilde{\phi _1^n}\geq 0$,  we also  have  
\begin{equation}
\label{nonegp}\widetilde{\phi_1}(x)\geq 0 \mbox{ in }B.
\end{equation}
Moreover there exists $C>0$, independent on $n\in\mathbb N$, such that 
\begin{equation}
\label{normpesfin}
\int_B  \frac{\left({\widetilde{\phi_i}}\right)^2}{|x|^2}\overset{\mbox{\footnotesize{Fatou}}}{\leq} \liminf_{n\rightarrow +\infty} \int _{B} \frac{\left(\widetilde{\phi _i^n}\right)^2}{|x|^2}\leq C,
\end{equation}
indeed, if by contradiction we have that $\int _{B} \frac{\left(\widetilde{\phi _i^n}\right)^2}{|x|^2}\rightarrow +\infty$, then by \eqref{c} and \eqref{convergenzapartel2} we derive
\[\delta_i(p)=\lim_{n\rightarrow +\infty} \beta_{i,\rad}^n(p)\overset{\eqref{betaiequ}}{=}\lim_{n\rightarrow +\infty}\frac{\int_{B} | \nabla \widetilde{\phi _i^n}|^2-p \int_{B}|u_p|^{p-1} \left(\widetilde{\phi _i^n}\right)^2}{\int _{B} \frac{\left(\widetilde{\phi _i^n}\right)^2}{|x|^2}}=0,\]
which is in contradiction with \eqref{deflim}.
Observe that by the bounds in \eqref{stainH10} and \eqref{convergenzapartel2} and the estimate in \eqref{deltaip} we also get that $\delta_i(p)>-\infty$.
By \eqref{stainH10} and \eqref{normpesfin} it follows that 
\begin{equation}\label{limInH}
\widetilde{\phi_i}\in\mathcal H_{\rad}. 
\end{equation}
Moreover \eqref{normpesfin} implies that, up to a subsequence,  
$\widetilde{\phi _i^n}$ converges to $\widetilde{\phi_i}$ also weakly in 
$L^2_{\frac 1{|x|^2}}(B)$.
Multiplying  \eqref{problemaPesatoAnello} by $\varphi\in \mathcal{H}$ and integrating on $A_n$, we have 
\begin{equation}\label{d}
\int_{A_n}\nabla  \widetilde{\phi _i^n}\cdot \nabla \varphi -\int_{|x|=\frac{1}{n}}\left(\nabla \widetilde{\phi _i^n}\cdot \nu \right)   \varphi-p \int_{A_n}  |u_p|^{p-1}\widetilde{\phi _i^n}\varphi=\beta_{i,\rad}^n(p)\int_{A_n} \frac{\widetilde{\phi _i^n}\varphi}{|x|^2}
\end{equation}
(where the only boundary term is the one on the interior part of $\partial A_n$ since $\varphi(x)=0$ when $|x|=1$).\\ 
Now by the weak convergence of  $\widetilde{\phi _i^n}\rightarrow \widetilde{\phi_i}$ in $H^1_0(B)$ and in
$L^2_{\frac 1{|x|^2}}(B)$, we have
\begin{equation}
\label{term1}
\int_{A_n}\nabla  \widetilde{\phi _i^n}\cdot \nabla \varphi=\int_{B}\nabla  \widetilde{\phi _i^n}\cdot \nabla \varphi\underset{n\rightarrow +\infty}{\longrightarrow} \int_{B}\nabla  \widetilde{\phi_i}\cdot \nabla \varphi,
\end{equation}
\begin{equation}
\label{term2}
\int_{A_n}  |u_p|^{p-1}\widetilde{\phi _i^n}\varphi=\int_{B}  |u_p|^{p-1}\widetilde{\phi _i^n}\varphi\underset{n\rightarrow +\infty}{\longrightarrow} \int_{B}  |u_p|^{p-1}\widetilde{\phi_i}\varphi 
\end{equation}
and  
\begin{equation}
\label{term3}
\int_{A_n} \frac{\widetilde{\phi _i^n}\varphi}{|x|^2}=\int_{B} \frac{\widetilde{\phi _i^n}\varphi}{|x|^2}\underset{n\rightarrow +\infty}{\longrightarrow}\int_{B} \frac{\widetilde{\phi_i}\varphi}{|x|^2}.
\end{equation}
Furthermore 
\begin{equation}
\label{convTermBordo}
\int_{|x|=\frac{1}{n}}\left(\nabla \widetilde{\phi _i^n}\cdot \nu \right)   \varphi \underset{n\rightarrow +\infty}{\longrightarrow} 0,
\end{equation}
indeed 
\begin{eqnarray}
\left|\int_{|x|=\frac{1}{n}}\left(\nabla \widetilde{\phi _i^n}\cdot \nu \right)   \varphi \right| &\leq & \frac{1}{n}\int_{|x|=\frac{1}{n}} |\nabla \widetilde{\phi _i^n}|\frac{|\varphi|}{|x|}
\leq \frac{1}{n}\int_{ A_n}|\nabla \widetilde{\phi _i^n}|\frac{|\varphi|}{|x|}\nonumber\\
 &\overset{\mbox{{\footnotesize H\"older}}}{\leq} &  \frac{1}{n}\left(\int_{ B}|\nabla \widetilde{\phi _i^n}|^2\right)^{\frac 12}\left(\int_B \frac{|\varphi|^2}{|x|^2}\right)^{\frac{1}{2}} 
 \nonumber\\
 &\overset{\eqref{c}}{\leq} & \frac{1}{n}\sqrt{C_p}\|\varphi\|_{\mathcal H}\underset{n\rightarrow +\infty}{\longrightarrow} 0.
\end{eqnarray}
Passing to the limit into \eqref{d} and using \eqref{term1}, \eqref{term2}, \eqref{term3} and \eqref{convTermBordo}, we get that $\widetilde{\phi_i}$ satisfies
\begin{equation}
\label{equadelta}
\int_{B}\nabla {\widetilde{\phi_i}}\cdot \nabla \varphi-p\int_{B}  |u_p|^{p-1}{\widetilde{\phi_i}}\varphi=\delta_i(p)\int_{B} \frac{{\widetilde{\phi_i}}\varphi}{|x|^2}, \quad\mbox{ for any }\varphi \in \mathcal{H},
\end{equation}
in particular by \eqref{limInH} we can choose $\varphi=\widetilde{\phi_i}\neq 0$ by \eqref{nonZ},  so since $\delta_i(p)<0$ we have from  \eqref{equadelta} that  the quadratic form evaluated at $\widetilde{\phi_i}$ is negative
\[Q_p(\widetilde{\phi_i})<0.\]
As a consequence, since $m_{\rad}(u_p)=2$, it must be
\[\widetilde{\phi_i}\in span\{\phi_1,\phi_2\}\]
where $0\neq \phi_j\in\mathcal H_{\rad}\subset H^1_0(B)$ is the function where the negative weighted radial eigenvalue $\beta_{j,\rad}(p)$ is achieved for  $j=1,2$, which satisfies
\eqref{problemaPesatoPallaDeboleiRadiale} (and so $Q_p(\phi_j)<0$) and such that $\phi_1\geq 0$ and $\phi_1\perp_{\mathcal H}\phi_2$ (so $\phi_2$ changes sign).
\\
Choosing the test function $\varphi=\phi_j$, $j=1,2$ into \eqref{equadelta} and using the equation \eqref{problemaPesatoPallaDeboleiRadiale} for $\phi_j$ we also get
\[[\delta_i(p)-\beta_{j,\rad}(p)]\int_B\frac{\widetilde{\phi_i}\phi_j}{|x|^2}=0\quad\mbox{ for }j=1,2,\]
hence the only possibility is that there exists $\alpha\in\mathbb R$ such that 
\begin{equation}\label{alternativa}\mbox{ either }\ \left\{\begin{array}{lr}\widetilde{\phi_i}=\alpha\phi_1\\\delta_i(p)=\beta_{1,\rad}(p)\end{array}\right. \ \mbox{ or }\ \left\{\begin{array}{lr}\widetilde{\phi_i}=\alpha\phi_2 \\\delta_i(p)=\beta_{2,\rad}(p)\end{array}\right.   
\end{equation}
for $i=1,2$. By \eqref{nonegp} it follows that  necessarily  there exists $\alpha\in\mathbb R$ such that
\[\left\{\begin{array}{lr}\widetilde\phi_1=\alpha \phi_1\\\delta_1(p)=\beta_{1,\rad}(p).\end{array}\right.\]

Moreover \eqref{posizionebetatilden1} and \eqref{posizionebetatilden2} proved in {\bf Step 1.} and the definition of $\delta_i(p)$ in \eqref{deflim} imply that
\[ 
\delta_1(p)< -1 \leq \delta_2(p),
\]
hence $\delta_1(p)\neq\delta_2(p)$ and so by \eqref{alternativa} necessarily there exists $\beta\in\mathbb R$ such that
 \[\left\{\begin{array}{lr}\widetilde\phi_2=\beta\phi_2 \\\delta_2(p)=\beta_{2,\rad}(p)\end{array}\right. \]
which concludes the proof of  \eqref{tesiSTEP2}.
\\
{\bf Step 3.} {\sl Conclusion.} 
\\
\eqref{MorsePari} is the same as \eqref{morsePari} in {\bf Step 1}. Moreover passing to the limit in \eqref{posizionebetatilden1} and in  \eqref{posizionebetatilden2} proved in {\bf Step 1} and using the results of {\bf Step 2} we obtain \eqref{nu1controllato} and \eqref{nu2controllato} respectively, where the strict inequalities are due to the monotonicity of the sequences $(\beta_{1,\rad}^n(p))_n$ and $(\beta_{2,\rad}^n(p))_n$.
\\
Last we show that there exists $C>0$ independent of $p$ such that 
\begin{equation}\label{nuUnoLimitato}
-C\leq \beta_{1,\rad}(p)\ (< 0)\quad \mbox{ for any $p>1$}
\end{equation}
from which the uniform bound on $j(p)$ then follows and this concludes the proof.
Let $\phi_p\in\mathcal H$ be a function where $\beta_{1,\rad}(p)$ is achieved,  then by \eqref{problemaPesatoPallaDeboleiRadiale}, choosing $v=\phi_p$, we have:
\begin{eqnarray*}
0\leq \int_{B}|\nabla\phi_p(y)|^2dy&=&\int_{B}p|\upp(y)|^{p-1}\phi_p(y)^2 dy+\beta_{1,\rad}(p)\int_{B}\frac{\phi_p(y)^2}{|y|^2}dy\\
&=&\int_{B}\left(p|\upp(y)|^{p-1}|y|^2+\beta_{1,\rad}(p)\right)\frac{\phi_p(y)^2}{|y|^2} dy\\
&\leq &  \left[\max_{y\in B} \left(p|\upp(y)|^{p-1}|y|^2\right)+ \beta_{1,\rad}(p)\right] \int_{B}\frac{\phi_p(y)^2}{|y|^2} dy,
\end{eqnarray*}
 As a consequence 
 \begin{equation}\label{semifin}
 \beta_{1,\rad}(p)\geq -\max_{y\in B} \left( p|\upp(y)|^{p-1}|y|^2\right).
 \end{equation} 
We recall the following pointwise estimate for $u_p$ which has been proved in  \cite{DeMarchisIanniPacellaJEMS}:
 \begin{equation}\label{Q3}
p|u_p(x)|^{p-1}|x|^2\leq C, \quad \forall p>1, \ \forall x\in B,
\end{equation}
 for a certain $C>0$  (see property $(P_3^k)$ in  \cite[Proposition 2.2]{DeMarchisIanniPacellaJEMS}, observing that in  the radial case the origin is the only absolute maximum point of $|u_p|$ and that $k=1$ by \cite[Proposition 3.6]{DeMarchisIanniPacellaJEMS}).
The conclusion follows combining  \eqref{Q3} with \eqref{semifin}.
\end{proof}

\begin{remark} \label{Remark:convergenzaAutovaloriDAnelloAPalla} In the proof of Lemma \ref{proposition:autovaloriRadiaiGenerale} we have introduced the auxiliary weighted eigenvalue problems
\eqref{problemaPesatoAnello} in the annuli $A_n$, $n\in\mathbb N_0$  and,
 as an intermediate step, we have shown that for any fixed $p>1$ the sequence  $(\beta_{i,\rad}^n(p))_n$ of the $i$-th radial eigenvalues for these problems is monotone non-increasing and 
\begin{equation}
\label{convergenzaAutovRadiali}
\beta_{i,\rad}(p)=\lim_{n\rightarrow +\infty} \beta_{i,\rad}^n(p)=\inf_n \beta_{i,\rad}^n(p), \quad i=1,2.
\end{equation}
We stress that by the spectral decomposition in
\eqref{spectralDecompositionAnello} we also get the non-increasing monotonicity of the sequence of the $j$-th eigenvalues $(\beta_j^n(p))_n$ of the problems  \eqref{problemaPesatoAnello}. Moreover combining  \eqref{convergenzaAutovRadiali}, the spectral decomposition in \eqref{spectralDecompositionAnello} and  Lemma \ref{lemma:decomposizionePalla} we also know the limit for the  negative ones:
\[
\beta_j(p)=\lim_{n\rightarrow +\infty} \beta_j^n(p)= \inf_n \beta_j^n(p), \quad j=1,\ldots, m(u_p). 
\]
The auxiliary weighted eigenvalue problems
\eqref{problemaPesatoAnello} in the annuli $A_n$, $n\in\mathbb N_0$ is the same already introduced and studied in \cite{DeMarchisIanniPacellaMathAnn} when computing the Morse index of $u_p$ for large $p$.
\end{remark}
%
%
%
%
%
%
%

Next we investigate the degeneracy of the solution $u_p$, for any $p>1$. This result will be useful to characterize the degeneracy of $u_p$ in the case of large $p$. Moreover we will need it to identify the possible bifurcation points and select the eigenfunctions related to them.

\begin{proposition}[Characterization of degeneracy]\label{p4.7}
For any $p\in (1, +\infty)$ let $j=j(p)\in\mathbb N$, $j\geq 2$ be as in Lemma \ref{proposition:autovaloriRadiaiGenerale}. The solution $u_p$ is degenerate if and only if 
\begin{equation}\label{1deg}\beta_{1,\rad}(p)=-j^2
\end{equation}
or 
\begin{equation}\label{2deg}\beta_{2,\rad}(p)=-1.
\end{equation}
Moreover the space of the solutions to the linearized problem \eqref{linearizedProblem} at a value $p$ that satisfies \eqref{1deg} and/or \eqref{2deg}
 is spanned by 
\begin{equation}\label{autof1}
v_{j}(r,\theta)=\phi_1(r) \left(A \sin(j\theta)+B\cos(j\theta)\right) \qquad  A,B\in \R
\end{equation}
and/or
\begin{equation}
\label{autof2}
v(r,\theta)=\phi_2(r) \left(A \sin(\theta)+B\cos(\theta)\right) \qquad  A,B\in \R,
\end{equation}
where $\phi_1, \phi_2$ are the  eigenfunctions associated to the first and second radial eigenvalue $\beta_{1,\rad}(p)$, $\beta_{2,\rad}(p)$ respectively.\\
Hence $Ker (L_p)$ has dimension $0$ when  neither  \eqref{1deg} nor \eqref{2deg} hold, dimension  $2$ when either  \eqref{1deg} or \eqref{2deg} holds, and dimension $4$ when   \eqref{1deg} and \eqref{2deg} hold. 
\end{proposition}
\begin{proof}
$u_p$ is degenerate if and only if there exists $v\in H^1_0(B)$, $v\neq 0$ such that
 \begin{equation}\label{nucleoNonVuoto}\left\{
\begin{array}{ll}
-\Delta v - p|u_p|^{p-1} v =0 \qquad & \text{ in } B, \\
v= 0 & \text{ on } \partial B.
\end{array} 
\right.
\end{equation}

{\bf Step 1.} {\sl We show that if $u_p$ is degenerate then \eqref{1deg} or \eqref{2deg} hold.}
\\
If $u_p$ is degenerate,  problem \eqref{nucleoNonVuoto} admits a solution $v$ which is continuous in $B$ by elliptic regularity. Then we can decompose $v$ along spherical harmonics, namely for $k\in\mathbb N$ we consider the radial function
\begin{equation}
\label{defhk}
h_k(r):=\int_0^{2\pi}\alpha_k(\theta)v(r,\theta)\, d\theta,\qquad r\in [0,1)
\end{equation}
where $\alpha_k$ is an eigefunction of $-\Delta_{S^1}$ associated to the eigenvalue $k^2$ (see \eqref{espressioneArmonicheSferiche}---\eqref{h_kInZero}). Since $(\alpha_k)_k$ is a complete orthogonal system for $L^2(S^1)$ and $v\neq 0$, then necessarily $h_k\neq 0$ for some $k\in\mathbb N$.
Moreover, similarly as in {\bf Step 1} in the proof of Lemma \ref{lemma:decomposizionePalla}, it is easy to show that 
 $h_k$, for these values of $k$, is a nontrivial solution to the problem 
\begin{equation}\label{radsph}
\left\{
\begin{array}{lr}
-h_k''-\frac 1r h_k'-p|u_p|^{p-1}h_k =\frac{-k^2}{r^2} h_k & \text{ in }(0,1)\\
h_k(1)=0
\end{array}\right. 
\end{equation}
Observe  that $k\geq 1$, since $u_p$ is radially nondegenerate by Lemma \ref{lemma:radiallyNonDeg}, so (see  \eqref{h_kInZero}), one has also
\begin{equation}\label{zeroh}h_k(0)=0.
\end{equation}
Next we show that $h_k$ satisfies also the condition
\begin{equation} 
\int_0^1 r(h_k')^2+\frac{h^2_k}{r}<+\infty. 
\end{equation}

Indeed, since $v\in H^1_0(B)$, we can argue as in the proof of \eqref{secondop} to get
\begin{equation} \label{radsphint1}
\int_0^1 r(h_k')^2<+\infty
\end{equation}
and moreover, using Proposition \ref{p2.2}, we also have that $h_k(r)=O(r^k)$,  as $r \to 0$, which implies
\begin{equation}\label{sommability}
\int_0^1\frac{h^2_k}{r}<+\infty.
\end{equation}
By  Lemma \ref{valoriVariaz=AutovaloriesatoPalla}, Lemma \ref{lemma:Morse=numeroAutovaloriesatoPalla} and Lemma \ref{2AutovRadNegativiPesatoPalla}  problem \eqref{radsph}-\eqref{radsphint1}-\eqref{sommability} admits only two negative eigenvalues which coincide with $\beta_{1,\rad}(p)$ and $\beta_{2,\rad}(p)$. Hence  we conclude that $h_k$ is nontrivial if and only if
  $\beta_{i,\rad}(p)=-k^2$ for some $i=1,2$ and $k\geq 1$. 
The equalities \eqref{1deg} and \eqref{2deg}  then follow remembering that, by Lemma \ref{proposition:autovaloriRadiaiGenerale}, $-1\leq\beta_{2,\rad}(p)<0$ and  $-j^2\leq \beta_{1,\rad}(p)<-(j-1)^2$ for some $j=j(p)\in\N$, $j\geq 2$.
\\
{\bf Step 2.} {\sl We show that if  \eqref{1deg} or \eqref{2deg} hold then $u_p$ is degenerate.}
\\
Let 
\begin{equation}\label{nuclei}
v_{i,k}(x):=\phi_i(|x|)\alpha_k(\frac{x}{|x|}), 
\end{equation}
where
 $\phi_i$ is an eigenfunction associated to the radial eigenvalue $\beta_{i,\rad}(p)$ and $\alpha_k$ is an eigefunction of $-\Delta_{S^1}$ associated to the eigenvalue $k^2$ (see \eqref{espressioneArmonicheSferiche}). Then easy computation shows that if \eqref{1deg} (resp. \eqref{2deg}) holds then $v_{i,k}$ with $i=1$ and $k=j$ (resp. $i=2$ and $k=1$) solves \eqref{nucleoNonVuoto}.
 \\ 
{\bf Step 3.} {\sl We show that the space of solutions of \eqref{nucleoNonVuoto} at a value $p$ that satisfies \eqref{1deg} (resp. \eqref{2deg}) is given by \eqref{autof1}  (resp. \eqref{autof2}).}
\\
The functions in \eqref{autof1}  (resp. \eqref{autof2}) clearly solve \eqref{nucleoNonVuoto}. This follows from {\bf Step 2}, recalling the explicit expression of $\alpha_k$ (see \eqref{espressioneArmonicheSferiche}). 
\\
To prove that the space of solutions to  \eqref{nucleoNonVuoto} is spanned by the functions in \eqref{autof1}  and/or  \eqref{autof2}, recall that $\alpha_k$ is an orthogonal basis for $L^2(S^1)$, hence
any nontrivial solution $v$ to \eqref{nucleoNonVuoto} may be written in $L^2(B)$ as
\begin{equation}\label{serie}v(r,\theta)=\sum_{k=0}^{+\infty}h_k(r)\alpha_k(\theta)
\end{equation}
with $h_k$ defined as in \eqref{defhk}. 
Then the same arguments used in {\bf Step 1} imply that when only \eqref{1deg} holds then $h_k=0$ for any $k\neq j$ and so   \eqref{serie} reduces to
\[v(r,\theta)=h_j(r)\alpha_j(\theta)\] with $h_j$ eigenfunction associated to the  radial eigenvalue $\beta_{1,\rad}(p)$, namely $h_j=\phi_1$; similarly when only \eqref{2deg} holds then $h_k=0$ for any $k\neq 1$ and so   \eqref{serie} reduces to \[v(r,\theta)=h_1(r)\alpha_1(\theta)\] where $h_1$ is now the eigenfunction associated to the radial eigenvalue $\beta_{2,\rad}(p)$, namely $h_1=\phi_2$.
\end{proof}

\

\section{The case $p$ large}\label{section p large}

In \cite{DeMarchisIanniPacellaMathAnn}, exploiting the asymptotic analysis of $u_p$ for $p\rightarrow +\infty$, it has been already proved that

\begin{proposition}\label{risultatoMorse_pGrande}
There exists $\hat{p}>1$ such that
\begin{equation}
m(u_p)=12 \qquad \forall\ p\geq \hat{p}.
\end{equation}
\end{proposition}

Moreover one can also improve some partial result in \cite{DeMarchisIanniPacellaMathAnn} about the asymptotic behavior of the first eigenvalue $(\beta^{n}_{1,\rad}(p))_n$ of the auxiliary weighted problem  \eqref{problemaPesatoAnello}  (cfr. \cite[Theorem 6.1]{DeMarchisIanniPacellaMathAnn}) and deduce the following asymptotic result for the first eigenvalue $\beta_{1,\rad}(p)$ in the ball (whose proof is sketched at the end of the section, see also \cite{AG2} for a detailed proof.)
%
%
%
\begin{lemma} \label{lemma:limiteBeta1}
\[\lim_{p\rightarrow +\infty}\beta_{1}(p)=-\frac{\ell^2+2}{2}\sim -26,9\]
where $\ell=\lim_{p\to \infty}\frac{s_p}{\varepsilon _p^-}\simeq 7.1979$ and  $s_p\in (0,1)$ is the minimum point of $u_p$, (i.e. the point such that $\norm{u_p^-}_{\infty}=u_p^-(s_p)=-u_p(s_p)$) and $\varepsilon _p^-$ is such that $\left(\varepsilon_p^-\right)^{-2}=p|u_p(s_p)|^{p-1}$.
\end{lemma}

Using the general  analysis previously done in Section \ref {section:generale} (Lemma \ref{proposition:autovaloriRadiaiGenerale} and Proposition \ref{p4.7}), combining it with Proposition \ref{risultatoMorse_pGrande} above and with the asymptotic result in Lemma \ref{lemma:limiteBeta1}, we completely characterize the degeneracy of the solution $u_p$ when  $p$ is large. Our result reads as follows:
\begin{proposition}
\label{lemma:autovaloriRadialipgrande} There exists $p^{\star}>1$ such that for any $p\geq p^{\star}$ 
\begin{equation}\label{nu1controllatopgrande}
-36 \ <\  \beta_{1,\rad}(p)<-25
\end{equation} 
\begin{equation}\label{nu2controllatopgrande}
-1\leq \beta_{2,\rad}(p)<0
\end{equation}
Moreover for $p\ge p^{\star}$ the solution $u_p$ is degenerate if and only if 
\begin{equation}\label{2degpgrande}\beta_{2,\rad}(p)=-1.
\end{equation}
The space of solutions to the linearized problem \eqref{linearizedProblem} at a value $p\geq p^{\star}$ that satisfies 
 \eqref{2degpgrande}
 is spanned by 
\begin{equation}
v(r,\theta)=\phi_2(r) \left(A \sin(\theta)+B\cos(\theta)\right) \qquad  A,B\in \R,
\end{equation}
where 
$\phi_2$ is the  eigenfunction associated to the 
second radial eigenvalue 
$\beta_{2,\rad}(p)$ 
.\\
Hence $Ker (L_p)$ for $p\geq p^{\star}$ has dimension $0$ when   \eqref{2degpgrande} does not hold and  dimension $2$ when it holds.
%
\end{proposition}

\begin{proof}
The proof follows from Lemma \ref{proposition:autovaloriRadiaiGenerale}, Proposition \ref{p4.7} and observing that by Proposition \ref{risultatoMorse_pGrande} $j(p)\equiv 6$ for $p\geq \hat{p}$  and that moreover by  Lemma \ref{lemma:limiteBeta1} there exists $p^{\star}(\geq\hat{p})$ such that the equality
\[\beta_{1,\rad}(p)=36\] is never attained when  $p\geq p^{\star}$.
\end{proof}

\

We conclude the section with:
\begin{proof}[Sketch of the proof of Lemma \ref{lemma:limiteBeta1}]$\;$\\ 
In \cite[Theorem 6.1]{DeMarchisIanniPacellaMathAnn} it has been proved that, if $\beta^n_1(p)$ is the first eigenvalue  of the auxiliary weighted problem  \eqref{problemaPesatoAnello}  in the annulus $A_n$, then there exists a sequence $n_p$ such that $n_p\to \infty$ as $p\to \infty$ and
\begin{equation}\label{provvisoria} \tag{A.1}
\lim_{p\to \infty}\beta_1^{n_p}(p)=-\frac{\ell^2+2}2.
\end{equation}
Here we want to show that the proof for the annulus $A_{n_p}$ in  \cite{DeMarchisIanniPacellaMathAnn} can be adapted to the ball $B$, so that one gets the same asymptotic result for the first eigenvalue $\beta_{1}(p)$ in the ball.

\

The proof of the convergence \eqref{provvisoria} in \cite{DeMarchisIanniPacellaMathAnn} deeply relies on the study of the behavior of the radial solution $u_p$ as $p\to \infty$,  it is quite long and involved and requires several steps, which we now try to retrace.\\

Let us first recall that $u_p$ admits two limit problems, one obtained when rescaling $u_p$ with respect to its maximum point, which is $0$ (the scaling parameter in this case is $\varepsilon^+_p$ defined by $\left(\varepsilon_p^+\right)^{-2}=pu_p(0)^{p-1}$) and the second one obtained rescaling $u_p$ with respect to its minimum point $s_p$ (with scaling parameter given by $\varepsilon_p^-$) (see \cite[Theorem 1]{GGP2} for the rigorous statement of the result). \\

As in \cite{DeMarchisIanniPacellaMathAnn} the idea to prove the result is now to consider the eigenvalue problem associated to $\beta_1(p)$, rescale properly the first eigenfunction $\psi_{1,p}$ using the scaling parameters $\varepsilon_p^\pm$
and pass to the limit into the rescaled equations. More precisely, similarly as in \cite{DeMarchisIanniPacellaMathAnn}, we will obtain again that the  rescaled eigenfunction $\widetilde  \psi_{1,p}^{+}(x):=\psi_{1,p}\left(\varepsilon_p^{+} x\right)$ vanishes, while the other  rescaled eigenfunction $\widetilde  \psi_{1,p}^{-}(x):=\psi_{1,p}\left(\varepsilon_p^{-} x\right)$ converges (in some sense) to the first eigenfunction of some eigenvalue limit problem, whose first eigenvalue is exactly the value $ -\frac{\ell^2+2}2$ in \eqref{provvisoria}.
\\

One of the main point, crucial to pass to the limit in the rescaled equation and get the limit eigenvalue problem, is to prove the analogous of Lemma 6.2, 6.3 and 6.4 in \cite{DeMarchisIanniPacellaMathAnn}, where  some estimates on the first eigenvalue $\beta_1^{n_p}(p)$ and on the associated rescaled eigenfunction are obtained.  
Similar estimates can be easily obtained exactly in the same way as in \cite{DeMarchisIanniPacellaMathAnn} directly for the first rescaled eigenfunctions $\widetilde  \psi_{1,p}^{\pm}$ in the ball (without restricting to $A_{n_p}$) and imply in turn the convergence
\begin{equation}\label{provvisoria2}\tag{A.2}
\widetilde  \psi_{1,p}^{\pm}\to C^{\pm}\widetilde  \psi_{1}^{\pm}\ \text{ as }p\to \infty
\end{equation}
in some sense (in particular uniformly on compact sets of $\R^2$), where $\widetilde\psi_{1}^{\pm}$ are  the first eigenfunctions of certain limit eigenvalue problems, associated respectively to eigenvalues $\widetilde \beta_1^{\pm}$. In particular, since we can prove that $\widetilde \beta_1^+=-1$ and we already know  that
\[\beta_1(p)\leq \beta_1^{n_p}(p)<-25 \ \text{ as $p$ large},\]
then necessarily
\[\widetilde  \psi_{1,p}^{+}\to 0 \ \text{ as }p\to \infty.\]

The other main point, following  the proof of \eqref{provvisoria}, is to show that $\widetilde  \psi_{1,p}^{-}$ does not vanishes.
This step requires a deep analysis on the behavior of the function 
\[
[0,1]\ni r\mapsto p|u_p(r)|r^2 \ \  \text{ as }p\rightarrow \infty\]
and luckily this behavior  does not depend on the annulus $A_{n_p}$ and this produces estimates in all of the ball $B$. As a consequence, we  can follow step by step the proof of  Proposition 6.6 in \cite{DeMarchisIanniPacellaMathAnn}, getting analogously  that
\[
\liminf_{p\to \infty}\int_{\{\frac 1K<|x|<K\}}\frac{\left(\widetilde  \psi_{1,p}^-\right)^2}{|x|^2}\ dx>0\]
for some $K>0$.
The rest of the proof then follows similarly as in \cite{DeMarchisIanniPacellaMathAnn}. One can find in \cite{AG2} all the details.
\end{proof}

\

\section{The case $p$ close to $1$}\label{sse:pvicino1}

Let us fix some notation.
We denote by $(\lambda_i)_i$ the sequence of the Dirichlet eigenvalues of $-\Delta$ in $B$, counted with their multiplicity.
Moreover let $(\varphi_i)_i$ be a basis of eigenfunctions in $L^2(B)$ associated to $\lambda_i$.
\\
We also denote by $(\lambda_{i,\rad})_i$ and $(\varphi_{i, \rad})_i$ the subsequences of the radial eigenvalues and eigenfunctions respectively (it is well known that $\lambda_{i,\rad}$ are simple in the space of radial functions and that $\varphi_{i,\rad}$ has $i-1$ zeros).
\\
The main result of this section is the following:
\begin{proposition}\label{risultatoMorse_pvicino1}
There exists $\delta>0$ such that 
\begin{equation}\label{morseindex-pvicino1}
 m(u_p)=6 \qquad \forall\ p\in (1,1+\delta)
\end{equation}
and $u_p$ is nondegenerate for $p\in (1,1+\delta)$  (namely $\mu_7(p)>0$).
\\
Moreover 
\begin{eqnarray}\label{iRadiali}
&& \mu_{i}(p) \underset{p\rightarrow 1}{\longrightarrow} \lambda_i- \lambda_{2,\rad} <0,\qquad i=1,\ldots, 5 
\\
\nonumber
&&\mu_{6}(p)=\mu_{2,\rad}(p)\underset{p\rightarrow 1}{\longrightarrow} \lambda_6- \lambda_{2,\rad}= 0^-
\end{eqnarray} 
and, up to a subsequence 
\begin{equation} \label{convAutofDefinitivaTeo}
v_{i,p}\underset{p\rightarrow 1}{\longrightarrow} C\frac{\varphi_i}{\|\varphi_i\|_\infty }\ \mbox{ in } C(\bar B),  \qquad i=1,\ldots, 6
\end{equation}
where $C=\pm1$ and $\mu_i(p)$, $\mu_{i,\rad}(p)$ are the Dirichlet eigenvalues and radial eigenvalues respectively of the linearized operator $L_p$ at $u_p$ (see \eqref{linearizedOperator}, \eqref{CourantCharEigenv} and \eqref{CourantCharEigenvRad}) and $v_{i,p}$ are the eigenfunctions of $L_p$ associated to the eigenvalues $\mu_{i,p}$ and normalized in $L^{\infty}(B)$ ($\|v_{i,p}\|_\infty =1$).
\end{proposition}

We observe that, combining  \eqref{morseindex-pvicino1} with the general results about the Morse index of $u_p$ and the characterization of its degeneracy  given  in  Section \ref{section p large} for any $p>1$ (Proposition \ref{p4.7} and Lemma \ref{proposition:autovaloriRadiaiGenerale} respectively), we  also have the following estimates for the $2$ negative radial eigenvalues of the auxiliary problem \eqref{problemaPesatoPalla}, when $p$ is close to $1$:
\begin{corollary}
\label{lemma:autovaloriRadialipvicino1}   
Let $\delta>0$ be as in Proposition  \ref{risultatoMorse_pvicino1}.  Then
for any $p\in (1,1+\delta)$ 
\begin{equation}\label{nu1controllatopvicino1}
-9  <  \beta_{1,\rad}(p)<-4
\end{equation} 
\begin{equation}\label{nu2controllatopvicino1}
-1  <  \beta_{2,\rad}(p)<0.
\end{equation}
\end{corollary}

\begin{proof}
From Lemma \ref{proposition:autovaloriRadiaiGenerale}, observing that \eqref{morseindex-pvicino1} implies $j(p)\equiv 3$ for $p\in (1,1+\delta)$, we have that
\begin{eqnarray*}
&& -9  \leq \beta_{1,\rad}(p)<-4
\\
&& 
-1  \leq \beta_{2,\rad}(p)<0,
\end{eqnarray*}
for $p\in (1,1+\delta)$. The strict inequalities in the left hand sides follow from the nondegeneracy of $u_p$ in $(1,1+\delta)$ (see Proposition \ref{risultatoMorse_pvicino1}) and from the characterization of the degeneracy in Proposition \ref{p4.7}.
\end{proof}

In order to obtain the previous result  we need to analyze the behavior of the solution $u_p$, as $p$ is close to $1$. 
We will show that  $u_p$ converges, as $p\rightarrow 1$, to the second radial Dirichlet eigenfunction of $-\Delta$ in the ball $B$ (Lemma \ref{lemma-pvicino1} below).\\
Hence let us recall some useful result for the Dirichlet eigenvalues and for the second radial eigenfunction of $-\Delta$ in $B$.

\begin{lemma}\label{lemma:morseSecAutofRadialeLap}  
One has
\[ m\left(\varphi_{2,\rad}\right)=5\]
and in particular  
\begin{equation}\label{ordinamentol}\lambda_1=\lambda_{1,\rad}<\lambda_2=\lambda_3<\lambda_4=\lambda_5<\lambda_6=\lambda_{2,\rad}< \lambda_7\leq \ldots.
\end{equation}
\end{lemma}
\begin{proof} This proof is classical, we write it for completeness.
The eigenfunctions of the Laplace operator $-\Delta$ with Dirichlet boundary conditions in $B$ are given, in radial coordinates, by
\begin{equation}\label{autof-laplace}
\widetilde\varphi_{n,k}(r,\theta)=J_n(\nu_{nk}r)\times \left\{\begin{array}{ll}
\cos (n\theta) &\\
\sin(n\theta) &\hbox{ for } n\neq 0
\end{array}\right.
\end{equation}
for $n\in\N,$ $k\in\N_0$, 
where $J_n$ are the Bessel functions of the first kind (see for instance \cite{Watson}) and  $\nu_{nk}$ is the $k$-th  positive root of $J_n$ (for any fixed $n$ there are infinitely many roots). The corresponding eigenvalues are given by 
\begin{equation}\label{eigenLap}
\widetilde\lambda_{nk}= \nu_{nk}^2,
\end{equation} hence they are simple for $n=0$ and have multiplicity $2$ when $n\geq 1$.\\
 From \eqref{autof-laplace} it follows that the second radial eigenfunction is 
 \[\varphi_{2,\rad}(r)=J_0(\nu_{02}r)\] and so by \eqref{eigenLap} the second radial eigenvalue is \begin{equation}\label{eigenLap2}\lambda_{2,\rad } = \nu_{02}^2 .
 \end{equation}
 The Morse index of $\varphi_{2,\rad}$ is the number of eigenvalues (counted with multiplicity) of the Laplace operator $-\Delta$ with Dirichlet boundary conditions in $B$ which are strictly less than $\lambda_{2,\rad}$. By \eqref{eigenLap} and \eqref{eigenLap2} this is equivalent to compute the number of the zeros $\nu_{nk}$  which are strictly less than $\nu_{02}$, recalling that when $n\geq 1$ each eigenvalue has multiplicity $2$.\\
It is known (see \cite[\uppercase{table vii}]{Watson})  that \begin{equation} \label{ordine}
\nu_{01}<\nu_{11}<\nu_{21}<\nu_{02},
\end{equation} while 
\begin{equation} 
\label{ordineancora}
\nu_{12},\nu_{22}, \nu_{h1}>\nu_{02},\qquad \forall h\geq 3
\end{equation} hence the Morse index of $\varphi_{2,\rad}$ is $5$.\\
By  \eqref{eigenLap}, \eqref{ordine} and \eqref{ordineancora} (recalling the multiplicities) it  follows that
\begin{eqnarray*}
\lambda_1 &=& \widetilde\lambda_{01},\\ 
\lambda_2=\lambda_3 &=& \widetilde\lambda_{11}, \\
\lambda_4=\lambda_5 &=&\widetilde\lambda_{21},\\
\lambda_6 &=& \widetilde\lambda_{02}
\ <\ \lambda_7,
\end{eqnarray*}
and that  \eqref{ordinamentol} holds.
\end{proof}

\

\subsection{Asymptotic behavior of $u_p$ as $p\rightarrow 1$}

\

We now analyze the asymptotic behavior of $u_p$, as $p\rightarrow 1$. In particular we obtain an expansion of its  $L^{\infty}$-norm up to the second order which will be useful for the proof of Theorem \ref{prop1.4} (see Proposition \ref{leastRadiale}).

\begin{lemma}\label{lemma-pvicino1}
Let $p_n$  be any sequence converging to $1$. Then 
\begin{equation}\label{u-n} 
\bar{u}_n:=\frac {u_{p_n}}{\norm{u_{p_n}}_{\infty}}\rightarrow 
\varphi_{2,\rad} =J_0(\nu_{02}|x|)\quad \text{ in }C( \bar B)
\end{equation}
(recall that, by the definition of $J_0$, we have that  $\norm{\varphi_{2,\rad}}_{\infty}=\varphi_{2,\rad}(0)=J_0(0)=1$)
and
\begin{equation}\label{gammanp-1} 
\norm{u_{p_n}}_{\infty}^{p_n-1}= \lambda_{2,\rad}\left(1-\widetilde c (p_n-1)\right)+o(p_n-1) \text{ as }n\to \infty
\end{equation}
where 
\begin{equation}\label{c-tilde}
\widetilde c:=\frac{\int_B\ \varphi_{2,\rad}^2\log|\varphi_{2,\rad}|dx}{\int_B\ \varphi_{2,\rad}^2dx}
\end{equation}
\end{lemma}
\begin{proof}
The function $\bar{u}_n$ defined in \eqref{u-n} satisfies  
\begin{equation}\label{eq:u-n}
\left\{\begin{array}{lr}
-\Delta \bar{u}_n= \gamma_n^{p_n-1}|\bar{u}_n  |^{p_n-1}\bar{u}_n\qquad  \mbox{ in }B\\
\bar{u}_n =0\qquad\qquad\qquad\qquad\qquad\mbox{ on }\partial B\\
 \bar{u}_n(0)=1
\end{array}\right.
\end{equation}
where $\gamma_n:=\norm{u_{p_n}}_{\infty}$. From \eqref{*} it easily follows 
\[\Arrowvert \gamma_n^{p_n-1}|\bar{u}_n  |^{p_n-1}\bar{u}_n\Arrowvert_{\infty}\leq M,\]
from which
\begin{equation}\label{bounGradL2}
\|\nabla\bar u_n\|_{L^2(B)}\leq M.
\end{equation}
Moreover  we have 
the following estimate
\begin{equation}\label{stima-agg}
|\left( |\bar{u}_n  |^{p_n-1}-1\right)\bar{u}_n|\leq c(p_n-1)
\end{equation}
 in $\bar B$, with $c$ independent on $n$. Estimate \eqref{stima-agg} obviously holds, for any fixed $n$, at the points at which $\bar{u}_n=0$. When $\bar{u}_n\neq 0$ instead it comes as in \cite[(3.10)]{AGG} from the identity $e^x-1=x\int_0^1e^{tx}\ dt$, from which 
\begin{equation}\label{6.12-bis}
|\bar{u}_n  |^{p_n-1}-1=(p_n-1)\log |\bar{u}_n  |\int_0^1 \left(|\bar{u}_n  |^{p_n-1}\right)^t\ dt,
\end{equation}
so that 
\[\big||\bar{u}_n  |^{p_n-1}-1\big|\leq (p_n-1)\big|\log |\bar{u}_n  |\big|,\]
which implies \eqref{stima-agg} by the boundedness of the function $x\mapsto x\log x$ in $(0,1)$. From \eqref{stima-agg} we get
\begin{equation}\label{6.17}
\left(|\bar{u}_n  |^{p_n-1}-1\right)\bar{u}_n\rightarrow 0\quad \text{ as }n\to +\infty\end{equation}
uniformly in $\bar B$.  Then, by \eqref{bounGradL2} and \eqref{6.17}, $\bar{u}_n$ converges, up to a subsequence, in $C(\bar B)$ to a solution to
\[
\left\{\begin{array}{lr}
-\Delta \bar{u}= \gamma\bar{u}\qquad\qquad  \mbox{ in }B\\
\bar{u} =0\qquad\qquad\qquad\mbox{ on }\partial B\\
 \bar{u}(0)=1
\end{array}\right.\]
where $\gamma:=\lim_{n\rightarrow +\infty} \gamma_n^{p_n-1}> 0$ by \eqref{*}. Moreover $\bar{u}$ is radial and we will prove that it has two nodal regions.
This implies that $\bar{u}=\varphi_{2,\rad}$  showing \eqref{u-n} and consequently $\gamma=\lambda_{2,\rad}$.
Since the convergence in \eqref{u-n} holds for every subsequence, then it holds directly for the sequence $\bar{u}_n$.\\
Next we show that $\bar u$ has $2$ nodal regions. Observe that  the number of nodal regions of $\bar u$ cannot be grater then $2$ since  $\bar u_n$ has $2$ nodal regions and it converges uniformly to $\bar u$. Let $r_n$ be the unique zero of $\bar{u}_n$ in $(0,1)$, up to a subsequence $r_n\rightarrow r_0$, then $\bar{u}$ has $2$ nodal regions if we show that $r_0\in (0,1)$.  The $C^0$ convergence  of $\bar{u}_n$ to $\bar{u}$ easily implies that $r_0>0$ since $\bar{u}(0)=1$. So by contradiction let us assume $r_n\rightarrow 1$ as $n\rightarrow +\infty$. By  Rolle  Theorem there exists $\xi_n\in (r_n,1)$ such that $\bar{u}_n'(\xi_n)=0$ for any $n$. By assumption $\xi_n\rightarrow 1$ as $n\rightarrow +\infty$. Moreover observe that the convergence in \eqref{u-n} holds also in $C^1(B)$, by standard regularity theory, so it follows that $\bar{u}'(\xi_n)\rightarrow 0$  and this is not possible since the Hopf boundary Lemma implies $\bar{u}'(r)\neq 0$ in a neighborhood of $r=1$.

\

We have shown so far that $  \gamma_n^{p_n-1}\to \lambda_{2,\rad}$ as $n\to \infty$.
To conclude we have to prove the expansion in \eqref{gammanp-1}. Let us multiply \eqref{eq:u-n} by $\varphi_{2,\rad}$ and integrate over $B$. We get
\[
\gamma_n^{p_n-1}\int_B |\bar{u}_n  |^{p_n-1}\bar{u}_n \varphi_{2,\rad}=\int_B \nabla 
\bar{u}_n\nabla \varphi_{2,\rad}=\lambda_{2,\rad}\int_B \bar{u}_n \varphi_{2,\rad}
\]
where last equality follows by the definition of $ \varphi_{2,\rad}$. This implies
%
that
\begin{equation}\label{8.19}
\lambda_{2,\rad} \int_B\left(|\bar{u}_n  |^{p_n-1}-1\right)\bar{u}_n \varphi_{2,\rad}=\left(\lambda_{2,\rad}-\gamma_n^{p_n-1}\right) \int_B |\bar{u}_n  |^{p_n-1}\bar{u}_n \varphi_{2,\rad}.
\end{equation}
By using the identity \eqref{6.12-bis}, which holds a.e. in $B$, we also have
\[
\int_B\left(|\bar{u}_n  |^{p_n-1}-1\right)\bar{u}_n \varphi_{2,\rad}=
(p_n-1)\int_B\bar{u}_n \varphi_{2,\rad} \log| \bar{u}_n|\int_0^1 |\bar{u}_n|^{t(p_n-1)}dt \ dx
\]
and so from \eqref{8.19} we get
\begin{equation}
\label{eqRapporto}
\frac{\lambda_{2,\rad}-\gamma_n^{p_n-1}}{\lambda_{2,\rad}(p_n-1)}=\frac{\int_B \bar{u}_n \varphi_{2,\rad} \log| \bar{u}_n|\int_0^1 |\bar{u}_n|^{t(p_n-1)}dt \ dx} {\int_B |\bar{u}_n  |^{p_n-1}\bar{u}_n \varphi_{2,\rad}\ dx}.
\end{equation}
To conclude the proof we show that the right hand side of \eqref{eqRapporto} converges to the constant $\widetilde c$ in \eqref{c-tilde}.
First we observe that the uniform convergence of $\bar u_n$ to $\varphi_{2,\rad}$ in $B$ implies
\begin{equation}
\label{laPrima}
\int_B |\bar{u}_n  |^{p_n-1}\bar{u}_n \varphi_{2,\rad}\to
\int_B \varphi_{2,\rad}^2\neq 0\  \text{ as }n\to \infty
\end{equation}
(recall that $\varphi_{2,\rad}(x)=J_0(\nu_{02}|x|)$). Moreover, 
 since $\|\bar u_n\|_{\infty}\leq 1$, ($\bar u_n\neq 0$ q.o.) and the function $x\mapsto x\log x $ is bounded in $(0,1)$, then  the term $\bar{u}_n \varphi_{2,\rad} \log| \bar{u}_n| \int_0^1 |\bar{u}_n|^{t(p_n-1)}dt \in L^{\infty}(B)$ and
 \[\norm{\bar{u}_n \varphi_{2,\rad} \log| \bar{u}_n|\int_0^1 |\bar{u}_n|^{t(p_n-1)}dt  }_{L^{\infty}(B)}\leq C,
  \]
  so by  the convergence of $\bar u_n$ to $\varphi_{2,\rad}$ and the dominated convergence theorem we also get
%
\begin{equation}
\label{laSeconda}\int_B\bar{u}_n \varphi_{2,\rad} \log| \bar{u}_n|\int_0^1 |\bar{u}_n|^{t(p_n-1)}dt \ dx\to
\int_B \varphi_{2,\rad}^2\log|\varphi_{2,\rad}|dx \ \text{ as }n\to \infty.
\end{equation}
%
Then, from \eqref{eqRapporto}, by \eqref{laPrima} and \eqref{laSeconda}, it follows that  $\frac{\lambda_{2,\rad}-\gamma_n^{p_n-1}}{\lambda_{2,\rad}(p_n-1)}$ is bounded and, up to a subsequence, 
\[
\frac{\lambda_{2,\rad}-\gamma_n^{p_n-1}}{\lambda_{2,\rad} (p_n-1)}\to \widetilde c \ \text{ as }n\to \infty.
\]
Since this convergence holds for every subsequence, then it holds for the sequence 
concluding the proof.

\end{proof}

\

\subsection{Proof of Proposition \ref{risultatoMorse_pvicino1}}

\

Using  Lemma \ref{lemma-pvicino1} and Lemma \ref{lemma:morseSecAutofRadialeLap}  we can finally prove  Proposition \ref{risultatoMorse_pvicino1}.

\begin{proof}[Proof of Proposition \ref{risultatoMorse_pvicino1}]  
The proof of \eqref{morseindex-pvicino1} consists in showing that for $p$ sufficiently close to $1$  \begin{equation}\label{tesiMorse}
m(u_p)=m\left(\varphi_{2,\rad}\right)+1,
\end{equation}
where $m\left(\varphi_{2,\rad}\right)=5$ by Lemma \ref{lemma:morseSecAutofRadialeLap}.  We divide it into three steps. First observe that for $\bar u_p$  defined from $u_p$ as in \eqref{u-n}
\begin{equation}
\label{stessoLinearizzato}|u_p|^{p-1}=\norm{ u_p}_{\infty}^{p-1}|\bar {u}_p|^{p-1}.
\end{equation}

\

{\bf Step 1.}
{\sl We show that $m(u_p)\geq m\left(\varphi_{2,\rad}\right)+1$, for $p$ sufficiently close to $1$.}
\\
Let  $Q_p\ : H^1_0(B)\rightarrow \R$ be the   quadratic form in \eqref{formaQuadratica} and let us consider  the first $5$ Dirichlet eigenfunctions  $\varphi_1,\dots,\varphi_5$ of  $-\Delta$  in $B$ and the corresponding eigenvalues $\lambda_1,\dots,\lambda_5$. Then by \eqref{stessoLinearizzato} we have that 
\begin{align*}
 Q_p(\varphi_i) & =  \int_B \left[|\nabla \varphi_i|^2 -p|u_p|^{p-1}\varphi_i^2  \right]dx\\
 & \overset{\eqref{stessoLinearizzato}}{=}\int_{B}\left[|\nabla \varphi_i|^2-p\norm{ u_p}_{\infty}^{p-1}|\bar{u}_p|^{p-1}\varphi_i^2\right]\, dx\\
&=\lambda_i \int_B\varphi_i^2\, dx -  p\norm{ u_p}_{\infty}^{p-1} \int_B |\bar u_p|^{p-1}\varphi_i^2\, dx
\\
&
\overset{(\star)}{=}\left(\lambda_i-\lambda_{2,\rad}\right)\int_B\varphi_i^2\, dx +o_p(1)<0
\end{align*}
for $i=1,\dots,5$ and $p$ sufficiently close to $1$, since $\lambda_i<\lambda_{2,\rad}$ by Lemma \ref{lemma:morseSecAutofRadialeLap}, where for the equality in $(\star)$ we have used \eqref{gammanp-1} and the Lebesgue dominated convergence theorem thanks to \eqref{u-n}.
Recalling that the eigenfunctions $\varphi_i$ are orthogonal in $L^2(B)$ and hence in $H^1_0(B)$
this means that the Morse index of $u_p$ is at least $5$ for $p$ sufficiently close to $1$. But
from \eqref{morsePari} in Lemma \ref{proposition:autovaloriRadiaiGenerale} we already know that $m(u_p)$ must be always even, then the Morse index of $u_p$ is at least $6$ for $p$ sufficiently close to $1$.

\

{\bf Step 2.} {\sl  Let $\mu_i(p)\leq 0$ be a non-positive Dirichlet eigenvalue of the operator $L_p$ for $p\in (1,1+\delta)$ and let $v_{i,p}$ be an associated eigenfunction with $\|v_{i,p}\|_{\infty}=1$.  We prove that as $p\rightarrow 1$
\begin{eqnarray}
&& \mu_{i}(p) \ra \lambda_j- \lambda_{2,\rad}  \label{convAuto}
\\
&& v_{i,p}\rightarrow {C_j}\varphi_j \ \mbox{ in } C(\bar B)\ \mbox{ up to a subsequence,}\label{convAutof}
\end{eqnarray}
for a certain $j=j(i)\in\{1,2,3,4,5,6\}$, where  $C_j:=\pm \norm{\varphi_j}^{-1}_\infty$. Moreover we also show that if $l\in\mathbb N$, $l\neq i$ and $\mu_l(p)\leq 0$ for $p\in (1,1+\delta)$, then   
\begin{equation}\label{distinti}j(l)\neq j(i)
\end{equation}
(we stress that under condition \eqref{distinti} it is nevertheless possible to have $\lambda_{j(l)}=\lambda_{j(i)}$). }
\\
Observe that the non-positive eigenvalue $\mu_{i}(p)$ is bounded for $p$ close to $1$, indeed by the standard variational characterization of $\mu_1(p)$
\begin{align*}
\mu_{i}(p)&>\mu_{1}(p)=\mu_{1,\rad}(p) \overset{\eqref{stessoLinearizzato}}{=}\inf_{\substack{v\in H^1_{0,\rad}(B)\\v\neq 0}}
\frac{\int_0^1\left(r\left(v'\right)^2-p\norm{ u_p}_{\infty}^{p-1}|\bar{u}_p|^{p-1}r v^2\right) dr}{\int_0^1 rv^2}\\
&\geq -p\norm{ u_p}_{\infty}^{p-1}\overset{\eqref{gammanp-1}}{\geq}-(\lambda_{2,\rad}+\ep) 
\end{align*}
for $p$ close to $1$. Let $p_n$ be a sequence converging to $1$, then the eigenfunction $v_{i,n}:=v_{i,p_n}$ satisfies
\begin{equation} \label{eigen}
\left\{
\begin{array}{ll}
L_pv_{i,n}\overset{\eqref{stessoLinearizzato}}{=}-\Delta v_{i,n} - p_n   \norm{ u_{p_n}}_{\infty}^{p_n-1}|\bar u_{p_n}|^{p_n-1} v_{i,n}  =\mu_i(p_n) v_{i,n} & \text{ in } B 
\\
\norm{ v_{i,n}}_{\infty}=1\\
v_{i,n}= 0 & \text{ on } \partial B.
\end{array}
\right. 
\end{equation} 
Moreover
  \[\big| p_n   \norm{ u_{p_n}}_{\infty}^{p_n-1}|\bar u_{p_n}|^{p_n-1} v_{i,n}  +\mu_i(p_n) v_{i,n}\big|\leq C\]
  and then, up to a subsequence, $v_{i,n}\to \tilde{\varphi_i}$ in $C(\bar B)$ where $\norm{\tilde \varphi_i}_{\infty}=1$ by the uniform convergence
and, using \eqref{gammanp-1} and \eqref{u-n}, it follows that $\tilde \varphi_i$ solves
\begin{equation} \label{new}
\begin{cases}
\begin{array}{ll}
-\Delta \tilde\varphi_{i} =\left(\lambda_{2,\rad} +\tilde \mu_{i}\right) \tilde \varphi_{i}\qquad & \text{ in } B \\
\norm{ \tilde \varphi_{i}}_{2}=1\\
\tilde \varphi_{i}= 0 & \text{ on } \partial B,
\end{array} 
\end{cases}
\end{equation} 
where $\tilde \mu_i=\lim_{n\rightarrow +\infty} \mu_{i}(p_n)\leq 0$. This means that $\tilde \varphi_i$ is an eigenfunction of the Laplace operator associated to the eigenvalue $ \lambda_{2,\rad} +\tilde \mu_{i}$, namely there exists $j=1,2,\ldots$ such that
\[\tilde \mu_{i}=\lambda_j-\lambda_{2,\rad}\]
and
\[\tilde \varphi_i={C_j}\varphi_j\]
where $C_j=\pm\norm{\varphi_j}_{\infty}^{-1}$.
Since $ \tilde \mu_{i}\leq 0$, by Lemma \ref{lemma:morseSecAutofRadialeLap} we have necessarily that $j\in\{1,2,3,4,5,6\}$.
Moreover, since the convergence in \eqref{convAuto} holds for any subsequence, then it also holds for the sequence.

Last we prove \eqref{distinti}. Let $l\neq i$ be such that $\mu_l(p)\leq 0$. We can take $v_{l,p}$ orthogonal in $L^2(B)$ to $v_{i,p}$.  The uniform convergence in $\bar B$ implies then that
$$0=\int_B v_{i,p}v_{l,p}={C_{j(i)}C_{j(l)}}\int _B \varphi_{j(i)} \varphi_{j(l)},$$
hence
$j(i)\neq j(l)$.

\

{\bf Step 3.} {\sl Conclusion}
\\
From  {\bf Step 2} we deduce that the operator $L_p$, for $p$ close to $1$, may have at most $6$ non-positive eigenvalues $\mu_i(p)\leq 0$, namely that $\mu_7(p)>0$.\\  Indeed   if we assume by contradiction that  
$\mu_7(p)\leq 0$ for $p$ close to $1$, then \eqref{convAuto} holds for all $i=1,2,\ldots, 7$ and so necessarily $j(7)=j(\hat i)$ for some $\hat i\in\{1,\ldots 6\}$, a contradiction with \eqref{distinti}.
\\
From {\bf Step 1},  we also know that the operator $L_p$ for $p$ close to $1$ has at least $6$ negative eigenvalues $\mu_i(p)< 0$.
\\ 
Combining both the information
we get:
\begin{equation}\label{dimostratoPrimaParte}
\mu_1(p)<\mu_2(p)\leq \mu_3(p)<\mu_4(p)\leq\mu_5(p)< \mu_6(p)<0<\mu_7(p)\leq\ldots
\end{equation}
(the strict inequalities are a consequence of  \eqref{ordinamentol} and of the convergence in \eqref{convAuto}), which proves  both \eqref{tesiMorse} and  the nondegeneracy of $u_p$ for $p$ close to $1$.\\ 
It remains to prove \eqref{iRadiali}. It is well known that $\mu_1(p)=\mu_{1,\rad}(p)$. Moreover  $m_{\rad}(u_p)=2$  by Lemma \ref{LemmaMorseIndexRadiale}, hence there exists a unique $l\in\{2,3,4,5,6\}$ such that $\mu_l(p)=\mu_{2,rad}(p)$. We denote by $v_{l,p}$ a radial eigenfunction associated to $\mu_l(p)$.
Next we show that $l=6$.\\
Observe that as a consequence of \eqref{dimostratoPrimaParte} and of the monotonicity property of the limit, we can take  $j=i$ in the convergences already proved in {\bf Step 2}, namely    \eqref{convAuto} and \eqref{convAutof} become respectively:
\begin{eqnarray}
&& \mu_{i}(p) \ra \lambda_i- \lambda_{2,\rad}  \label{convAutoDefinitiva}
\\
&& v_{i,p}\rightarrow {C_i}\varphi_i \label{convAutofDefinitiva}
\end{eqnarray}
as $p\rightarrow 1$, for any $i=1,\ldots, 6$.\\
Obviously $\varphi_1=\varphi_{1,\rad}$ and moreover, since $\lambda_6=\lambda_{2,\rad}$ by Lemma \ref{lemma:morseSecAutofRadialeLap}, we can take 
$\varphi_6=\varphi_{2,\rad}$,  while $\varphi_i$ is surely not radial for $i=2,3,4,5$.
Observe now that $\varphi_l$ is radial, being obtained in the limit of the radial eigenfunction $v_{l,p}$ in \eqref{convAutofDefinitiva}, this proves that $l=6$.
Last \eqref{convAutoDefinitiva} in the case $i=6$  also gives the limit
$ \mu_{6}(p)=\mu_{2,\rad}(p) \ra 0^-$
as $p\rightarrow 1$.
\end{proof}

\

\section{Regularity of the eigenvalues and the set $\mathcal P^j$}
By Proposition \ref{p4.7} the sets of the exponents $p$ at which $u_p$ is degenerate are
\begin{eqnarray}\label{setdegeneratep}
&&\{p\in(1,+\infty) : \beta_{1,\rad}(p)=-j^2\},\ \mbox{ for }j\in\N_0\nonumber
\\
&&\{p\in(1,+\infty) : \beta_{2,\rad}(p)=-1\}.
\end{eqnarray}
In particular we will be interested in the subset
\begin{equation}\label{def:S_j}
\mathcal P^j:=\left\{p\in(1,+\infty) :\ 
p\mapsto \beta_{1,\rad}(p)+j^2\  \mbox{ changes sign}\right\}, \quad j\in \N_0.
\end{equation}
Clearly $\mathcal P^j\subseteq \left\{p\in(1,+\infty) :\beta_{1,\rad}(p)=-j^2\right\}.$ 
\begin{lemma}\label{l5.1}
The maps $p\mapsto\beta_{i,\rad}(p)$  are analytic in $p$ and the sets of degenerate points in \eqref{setdegeneratep}, when not empty, consist of only isolated points. \\
Moreover $\mathcal P^j\not=\emptyset$, for $j=3,4,5$ 
and there exists an odd number $s_j(\geq 1)$ of isolated  values $p_1^j,\ldots, p^j_{s_j}\in (1+\delta,p^{\star})$ (where $\delta$ and $p^{\star}$ are as in Proposition \ref{risultatoMorse_pvicino1} and Proposition \ref{lemma:autovaloriRadialipgrande} respectively) such that
\[\mathcal P^j=\{p_1^j,\ldots, p^j_{s_j}\}  \qquad j=3,4,5.\]
\end{lemma}
\begin{proof} 
In \cite{D} it is proved that  for any smooth bounded domain $\Omega\subset \R^2$ for any $p>1$ except possibly for isolated $p$ the equation  $-\Delta u=u^p$ in $\Omega$, $u=0$ on $\partial \Omega$ has a non-degenerate positive solution.
The proof relies on the fact that the map $(u,p)\longrightarrow \left(-\Delta\right)^{-1}(u^p)$ is real analytic when considered in a suitable cone of positive weighted functions.
\\
This proof cannot be directly applied for sign-changing solutions, and so we need to adapt the proof of the analyticity for sign-changing radial fast decay solutions in the exterior of the ball used in \cite{DW}, which holds in $\R^N$, with $N\geq 3$.
\\
Following \cite{DW} we let $\tilde w_p(s)=r^{\frac 2{p-1}}u_p(r)$, for $r=e^s$. This function satisfies
\[
\tilde{w}_p''-\frac 4{p-1}\tilde{w}_p'+\left(\frac 2{p-1}\right)^2\tilde w_p+|\tilde w_p|^{p-1}\tilde w_p=0
\]
for $s\in(-\infty,0)$ with the conditions
\begin{equation}\label{***}
\tilde w_p(0)=0\quad ,\quad \lim_{s\to -\infty}\tilde w_p(s)=0.
\end{equation}
We consider, for $z>0$, the rescaled function $w(t)=\tilde w_p(z^{-1}t)$ that satisfies
\begin{equation}\label{riscalato}
{w}''-\frac 4{p-1}z{w}'+\left(\frac 2{p-1}\right)^2z^2 w+z^2| w|^{p-1} w=0
\end{equation}
in $(-\infty,0)$ with the boundary conditions in \eqref{***}. We let $s_1$ be the unique zero of $w(t)$ in $(-\infty, 0)$ and we consider problem \eqref{riscalato} in one of the intervals $(-\infty,s_1)$ or $(s_1,0)$ with Dirichlet boundary conditions (also at infinity). Of course we have that $r_1=e^{z^{-1}s_1}$ is the unique zero of $u_p$. 
Problem \eqref{riscalato} is equivalent to solve
\[
\begin{cases}
-\Delta u=u^p & \text{ in }\Omega_i\\
u>0  & \text{ in }\Omega_i\\
u=0 & \text{ on }\partial\Omega_i
\end{cases}\]
where $\Omega_1=B(0,e^{z^{-1}s_1})$ or $\Omega_2=B\setminus B(0,e^{z^{-1}s_1})$ and $u$ is radial. The  Dancer  result  for positive solutions in \cite{D} implies then that the positive solutions $w^1_{z,p}$ and $w^2_{z,p}$ to \eqref{riscalato}, in $(-\infty,s_1)$ and $(s_1,0)$ respectively, depend analytically on $p$ and $z$.
\\
Lastly, following the proof of Lemma 3.2 part c) in \cite{DW}, one can show the existence of $z_p$ close to $1$ and analytic in $p$ such that the function  
\[\tilde{w}_p(s) =\left\{
\begin{split}
w^1_{z_p,p}(z_ps) \quad \text{ for }s\in (-\infty,z_p^{-1}s_1]\\ 
-w^2_{z_p,p}(z_ps) \quad \text{ for }s\in (z_p^{-1}s_1,0)
\end{split}
\right.
\]
is $C^1$ in $s=z_p^{-1}s_1$. This proves that $p\mapsto u_p$ is analytic.
\\
The fact that $u_p$ is analytic with respect to $p$ implies that the eigenvalues $\beta_{1,\rad}(p)$, $\beta_{2,\rad}(p)$ are analytic \cite{Kato}. 
Moreover by \eqref{nu1controllatopgrande} and \eqref{nu1controllatopvicino1} it follows that $p\mapsto\beta_{1,\rad}(p)$ is not constant in $(1,+\infty)$ and so the solutions $p\in (1,+\infty)$ to  $\beta_{1,\rad}(p)=-j^2$ are isolated and can accumulate only at $+\infty$. 
By \eqref{nu2controllatopgrande} and \eqref{nu2controllatopvicino1} instead, either $p\mapsto\beta_{2,\rad}(p)$  is constant and there are no solutions $p\in (1,+\infty)$ to $\beta_{2,\rad}(p)=-1$ or it is not constant in $(1,+\infty)$  and in this case the solutions $p\in (1,+\infty)$ to  $\beta_{2,\rad}(p)=-1$ are isolated and can accumulate only at $+\infty$.
Finally  \eqref{nu1controllatopgrande} and  \eqref{nu1controllatopvicino1} imply also that $\beta_{1,\rad}(p)+j^2$ changes sign for some $p\in(1+\delta,p^{\star})$ (precisely at an odd number of values  of $p$), when $j=3,4,5$. 
\end{proof}

\

\section{Morse index and degeneracy in symmetric functions spaces}\label{se:symmspaces}

To prove the bifurcation result in Theorem \ref{teo1} and also to prove Theorem  \ref{prop1.4} we need to introduce some  spaces of symmetric functions. To this end we let $O(2)$ be the orthogonal group in $\R^2$, $O_k\subset O(2)$, for $k\in\mathbb N_0$, be the subgroup of rotations of angle $\frac {2\pi}k$ and $\tau\in O(2)$ be the reflection with respect to the $x$-axis, i.e. $\tau(x,y)=(x,-y)$ for any $(x,y)\in \R^2$. For any $k\in \N_0$, we denote by 
\begin{equation}\label{g-k}
\mathcal{G}_k\subset O(2) \ \text{ the subgroup generated by the elements of $O_k$ and by $\tau$ }
\end{equation}
and by 
\begin{equation}
\label{Hk}
H^1_{0,k}(B)\ : =\{v\in H^1_0(B)\ \text{ such that } v(g(x))=v(x), \  \  \forall g\in \mathcal{G}_k, \  \forall x\in B  \}.
\end{equation}

The functions in the spaces $H^1_{0,k}(B)$ clearly possess the following invariances (in polar coordinates $(x,y)=(r\cos \theta,r\sin \theta)$):
\begin{eqnarray} \label{fi00}
&v(r,\theta)=v(r,2\pi-\theta)\\\label{fi0}
&v(r,\theta)=v(r, \theta+\frac{2\pi}k) 
\end{eqnarray}
and so also
\begin{equation}\label{fi1}
v(r,\frac \pi k+\theta)=v(r,\frac \pi k-\theta) 
\end{equation}
for every $r\in (0,1]$ and for every $\theta\in [0,2\pi]$. Note that in general $\theta+\frac{2\pi}k\notin[0,2\pi]$, if this occurs we mean that $v(r,\theta)=v(r, \theta+\frac{2\pi}k-2\pi)$ and similarly we do when $\frac{\pi}k\pm\theta\notin[0,2\pi]$.  \\

Observe that when $k=1$ then $O_1$ is the trivial subgroup of $O(2)$ given by the identity map and the functions in  $H^1_{0,1}(B)$ are only invariant by the reflection $\tau$.

\

Clearly the radial solution $u_p\in  H^1_{0,k}(B)$, for every $k\in\N_0$. 
\\
As a consequence, letting as before $\left(\mu_i(p)\right)_{i\in\N_0}$ be the sequence of the eigenvalues of the linearized operator $L_p$ at $u_p$ (see Section \ref{section:linearizedOperator}), we can consider its subsequence  $\left( \mu_{i,k}(p)\right)_{i\in \N_0}$ of the $\mathcal G_k$-symmetric eigenvalues (i.e. eigenvalues associated to an eigenfunction that belongs to $H^1_{0,k}(B)$) for any $k\in\N_0$, which can be characterized as
\begin{eqnarray}\nonumber
\mu_{i,k}(p) &=& \min_{\substack{
W\subset H^1_{0,k}(B)\\ dim W=i}}   \max_{\substack{v\in W\\v\neq 0}}\ \ \
R_p[v],
\end{eqnarray}
where $R_p$ is the usual Rayleigh quotient as in \eqref{Rayleigh}. By the principle of symmetric criticality the functions $v_i$ that attains $\mu_{i,k}(p)$ are indeed solutions to the eigenvalue problem associated to the linearized operator, i.e. they satisfy
\[
\begin{cases}
\begin{array}{ll}
-\Delta v_i - p|u_p(x)|^{p-1} v_i =\mu_{i,k}(p) v_i\qquad & \text{ in } B, \\
v_i= 0 & \text{ on } \partial B
\end{array} 
\end{cases}\]
and are invariant by the action of $\mathcal{G}_k$. It is known that $\mu_{1,k}(p)=\mu_{1,\rad}(p)=\mu_1(p)$, for any $k\in\N_0$, since $v_1$ is a radial function.

\

We then define the  {\sl $k$-Morse index of $u_p$}, that we denote by $m_k(u_p)$, as the number of the negative $\mathcal G_k$-symmetric eigenvalues $\mu_{i,k}(p)$ of $L_p$  counted with multiplicity.

\

To compute the $k$-Morse index of $u_p$  it is useful the following result, analogous to the one in Lemma \ref{lemma:Morse=numeroAutovaloriesatoPalla}:
\begin{lemma}\label{lemma:Morse=numeroAutovaloriesatoPallaSIMMETRIA} 
The  $k$-Morse index of $u_p$ coincides with the number of the negative $\mathcal G_k$-symmetric eigenvalues  of the weighted problem \eqref{problemaPesatoPalla} counted according to their multiplicity.
\end{lemma}
The proof of the previous result is an easy adaptation of the arguments in \cite[Lemma 2.6]{GGN2} and relies on the variational characterization of the negative $\mathcal G_k$-symmetric eigenvalues  of the weighted problem \eqref{problemaPesatoPalla} (i.e. the eigenvalues whose eigenfuntions belong to $H^1_{0,k}(B)$). Indeed observe that they are a subsequence of the eigenvalues of the weighted problem \eqref{problemaPesatoPalla}
and that, as we have already seen in Section \ref{se:auxiliary}, they can be variationally characterized exactly when they are  {\sl negative}. More precisely,  by the principle of symmetric criticality, we can now restrict to  the subspace
$\mathcal H_{k}$ of the $\mathcal G_k$-symmetric functions of $\mathcal H$ ($\mathcal H_{k}\subset H^1_{0,k}(B)$) and define
\begin{eqnarray}\label{CourantCharEigenvpesatipallaSIMMETRICO1}
\beta_{1,k}(p) &:=&    \inf_{\substack{v\in \mathcal H_{k},\,\,v\neq 0}}\ \ \
\widetilde{R_p}[v]  \ (= \beta_1(p)=\beta_{1,\rad}(p))
\end{eqnarray}
and, if $\beta_{j,k}(p)<0$ for $j=1,\ldots, i-1$
\begin{eqnarray}\label{CourantCharEigenvpesatipallaSIMMETRICO}
\beta_{i,k}(p) &:=&   \inf_{\substack{v\in \mathcal H_{k},\,\,v\neq 0\\ v\perp_{\mathcal H}span\{\phi_{1},\ldots,\phi_{i-1}\}}}\ \ \
\widetilde{R_p}[v],\qquad i\in\N, \ i\geq 2,
\end{eqnarray}
where $\phi_j\in\mathcal H_{k}$ is the function where $\beta_{j,k}(p)$ is achieved for  $j=1,\ldots, i-1$ and solve 
\begin{equation}\label{problemaPesatoPallaDeboleiSIMMETRICO}
\int_B\nabla\phi_j\nabla v-p|u_p|^{p-1}\phi_jv\, dx= \beta_{j,k}(p)\int_B \frac{\phi_j
 v}{|x|^2}\, dx, \qquad \forall v\in \mathcal H.
\end{equation}
So similarly as in Lemma \ref{valoriVariaz=AutovaloriesatoPalla} one can prove the following variational characterization, which then gives the characterization of the $k$-Morse index in Lemma \ref{lemma:Morse=numeroAutovaloriesatoPallaSIMMETRIA} above:
\begin{lemma}
\label{valoriVariaz=AutovaloriesatoPallaSIMMETRICO}
The negative $\mathcal G_k$-symmetric eigenvalues  of problem \eqref{problemaPesatoPalla} coincide with the negative numbers  $\beta_{i,k}(p)$'s  in \eqref{CourantCharEigenvpesatipallaSIMMETRICO1}-\eqref{CourantCharEigenvpesatipallaSIMMETRICO}. Moreover the corresponding 
eigenfunctions, which solve \eqref{problemaPesatoPalla}, are in $\mathcal H_k$ and can be  chosen to be orthogonal in the sense of \eqref{perpH}. 
\end{lemma}

\begin{remark}[$\mathcal G_k$-invariance of the eigenfunctions]\label{remark:simmetriaAutofunzioni}
Recall that, according to  the spectral decomposition result in Lemma \ref{lemma:decomposizionePalla} and using Lemma \ref{2AutovRadNegativiPesatoPalla}, 
 we can decompose  the negative eigenvalues 
 of the weighted  problem \eqref{problemaPesatoPalla} as
\begin{equation}\label{autov-minori-zero}
\beta_{n,\rad}(p)+j^2<0
\end{equation}
for some $n=1,2$ and some $j\in \N $, where $\beta_{n,\rad}(p)$ are the negative radial weighted eigenvalues as defined in Section \ref{se:auxiliary}. 
\\
Moreover the  eigenfunctions  associated to each $(n,j)\in\{1,2\}\times\N$ in the decomposition \eqref{autov-minori-zero}  are explicitly known by Lemma  \ref{lemma:decomposizionePalla}, indeed they are:
\[
\phi_n(r)\cos (j\theta)\ \mbox{ and }\ \phi_n(r)\sin (j\theta)
\]
 where $\phi_n(r)$ is a radial eigenfunction associated to the simple  radial eigenvalue $\beta_{n,\rad}(p)$. 
 \\
Recall also that, by \eqref{eigenspaceSpectralDec}, the eigenspace related to each negative eigenvalue of problem \eqref{problemaPesatoPalla} is generated by these eigenfunctions, with $(n,j)$ varying among all the possible associated decompositions. \\
Hence the $\mathcal G_k$-invariance of the eigenfunctions is known, precisely one has that: 
\begin{itemize}
\item[$a)$] for $j=0$, the eigenvalues $\beta_{1,\rad}(p)<\beta_{2,\rad}(p)<0$ are simple in the space of the radial functions and each one produces $1$ radial eigenfunction $\phi_{n}$ ($n=1,2$ respectively) of problem \eqref{problemaPesatoPalla}, which belongs to $H^1_{0,k}(B)$ for every $k\geq1$;
\item[$b)$] 
for every $j\geq 1$, the eigenfunction $\phi_n(r)\sin (j \theta)$  doesn't belong to any space $H^1_{0,k}(B)$, $k\geq 1$ (since the reflection $\tau\in\mathcal G_k$);
\item[$c)$] 
for every $j\geq 1$, the eigenfunction $\phi_n(r)\cos (j \theta)$ is in $H^1_{0,j}(B)$;      
\item[$d)$]  
for every $j\geq 2$, the eigenfunction $\phi_n(r)\cos (j \theta)$  belongs also to the spaces $H^1_{0,k}(B)$ such that $k\in\N_0$ is a factor  of $j$ (we write $k\mid j$) (in particular it always belongs to $H^1_{0,1}(B)$), while it doesn't belong to the spaces $H^1_{0,k}(B)$ when $k\in\N_0$ is not a factor  of $j$.
\end{itemize}
\end{remark}

\

In the next section we will use the following result:
\begin{lemma}
\label{lemma:legameAutovSimmNegativoLinEPesato} 
Let $p\in (1,+\infty)$. The linearized operator
$L_p$ has a negative eigenvalue with eigenfunction in $H^1_{0,k}(B)\setminus H^1_{0,\rad}(B)$ if and only if 
\begin{equation}\label{negativo-k}
\beta_{1,\rad}(p)+k^2<0 \quad \end{equation}
\end{lemma}
\begin{proof}
Lemma \ref{lemma:Morse=numeroAutovaloriesatoPallaSIMMETRIA} implies that  $L_p$ has a negative eigenvalue in $H^1_{0,k}(B)\setminus H^1_{0,\rad}(B)$ if and only if the weighted problem \eqref{problemaPesatoPalla} has a negative eigenvalue in the space $\mathcal H_k\setminus \mathcal H_{\rad}$. By the spectral decomposition  given in Lemma \ref{lemma:decomposizionePalla} then, when \eqref{negativo-k} holds problem \eqref{problemaPesatoPalla} has the negative eigenvalue $\beta(p)=\beta_{1,\rad}(p)+k^2$ with corresponding eigenfunctions $\phi_1(r)\sin (k\theta)$ and $\phi_1(r)\cos (k\theta)$, the second of which  belonging to $\mathcal H_{k}\setminus \mathcal H_{\rad}$. When, instead $\beta_{1,\rad}(p)+k^2\geq 0$ the negative eigenvalues of problem \eqref{problemaPesatoPalla} are: $\beta_{i,\rad}(p)$, for $i=1,2$ with corresponding eigenfunctions $\phi_i(r)\in \mathcal H_{\rad}$ so that they do not belong to  $\mathcal H_k\setminus\mathcal H_{\rad}$ and
$\beta_{1,\rad}(p)+j^2$ for some  $j\in \{1,\dots,k-1\}$ with corresponding  eigenfunctions $\phi_1(r)\sin (j\theta)$ and $\phi_1(r)\cos (j\theta)$ neither of which belong to $\mathcal H_k$ since $j<k$, by Remark \ref{remark:simmetriaAutofunzioni}. This means that when \eqref{negativo-k} is not satisfied then the linearized operator does not admit any negative eigenvalue in $H^1_{0,k}(B)\setminus H^1_{0,\rad}(B)$ concluding the proof.
\end{proof}

By exploiting the information about the location of the weighted radial eigenvalues $\beta_{n,\rad}(p)$, $n=1,2$ obtained in the previous sections we can also derive information about the $k$-Morse index of the radial solution $u_p$ which will be useful to prove the non-radial part in Theorem \ref{prop1.4} (see Section \ref{section:proofTeoLeastEnergy}). 
%
%
%
%
Indeed using the results in Section \ref{section p large} and Section \ref{sse:pvicino1}, we can   explicitly compute the $k$-Morse index of $u_p$, for $p$ large enough and for $p$ close to $1$  respectively:
\begin{proposition}\label{Morse-simmetrico1}
Let $p^{\star}>1$ be as in Proposition \ref{lemma:autovaloriRadialipgrande}. Then for any $p\geq p^{\star}$ 
\begin{equation}\label{Morse-symm}
m_k(u_p)=\begin{cases}
7 & \text{ for }k=1\\
4 & \text{ for }k=2\\
3& \text{ for }k=3,4,5\\
2& \text{ for }k\geq 6
\end{cases}\end{equation}
\end{proposition}
\begin{proof} By Lemma \ref{lemma:Morse=numeroAutovaloriesatoPallaSIMMETRIA} in order to compute $m_k(u_p)$ we have to count the linearly independent eigenfunctions to the weighted problem \eqref{problemaPesatoPalla} which are associated to a negative eigenvalue and belong to the symmetric space $H^1_{0,k}(B)$.
\\
From Proposition \ref{lemma:autovaloriRadialipgrande} we know that for $p\geq p^{\star}$ it holds 
\[-36  < \beta_{1,\rad}(p)<-25   ,\  \  \ -1\leq \beta_{2,\rad}(p)<0.\] 
Then all the negative  eigenvalues are given  by \eqref{autov-minori-zero} with
\[j=\left\{
\begin{array}{lr}
0 & \mbox{ for }n=2\\
0,1,2,3,4,5 & \text{ for }n=1
\end{array}
\right.
\]
The conclusion follows by $a)$, $b)$, $c)$  and $d)$ in Remark \ref{remark:simmetriaAutofunzioni}.  
\end{proof}

Analogously for $p$ close to $1$ one has:
\begin{proposition}\label{Morse-simmetrico2}
Let $\delta>0$ be as in Proposition \ref{risultatoMorse_pvicino1}. Then for any $p\in (1,1+\delta)$ 
\begin{equation}\label{Morse-symm2}
m_k(u_p)=\begin{cases}
4& \text{ for }k=1\\
3 & \text{ for }k=2\\
2 & \text{ for }k\geq 3
\end{cases}\end{equation}
\end{proposition}
 \begin{proof}
We reason as in the proof of the previous lemma. 
From Corollary \ref{lemma:autovaloriRadialipvicino1} we know that for $p\in (1,1+\delta) $ it holds 
\[-9< \beta_{1,\rad}(p)<-4 , \ \  \ -1< \beta_{2,\rad}(p)<0.\]
Then all the negative  eigenvalues are given  by \eqref{autov-minori-zero} with
\[j=\left\{
\begin{array}{lr}
0 & \mbox{ for }n=2\\
0,1,2 & \text{ for }n=1
\end{array}
\right.
\]
The conclusion follows again by  Remark \ref{remark:simmetriaAutofunzioni}. 
%
%
%
%
\end{proof}

 \
 
 Finally we can characterize the degeneracy of $u_p$ in the symmetric spaces.
 We know from Proposition \ref{p4.7} that $u_p$ is degenerate if and only if 
\begin{align*}
&\beta_{1,\rad}(p)+j^2=0 \quad \text{ for some }j=j(p)>1 \ \text{ or }\\
&\beta_{2,\rad}(p)+1=0
\end{align*}
and these equalities can hold at the same time. As we will see in next result, the restriction to the symmetric spaces on one side rules out the degeneracy due to the second case, on the other side reduces the kernel of $L_p$ to be $1$-dimensional.

\begin{proposition}[Characterization of degeneracy in $H^1_{0,k}(B)$]\label{lemma:degenerazioneSimmetria}
Let $\delta>0$ and $p^{\star}>1$ be as in Proposition \ref{risultatoMorse_pvicino1} and Proposition \ref{lemma:autovaloriRadialipgrande} respectively. Let $k\in\N_0$. 
\begin{itemize}  
\item[$i)$] 
if  $p\in (1,1+\delta)$ then $u_p$ is  non-degenerate in $H^1_{0,k}(B)$ for any $k\geq 1$; 
\item[$ii)$] 
if $p\geq p^{\star}$ then $u_p$ is non-degenerate in $H^1_{0,k}(B)$ for any $k\geq 2$;  
\item[$iii)$]  
if $p\in (1+\delta,p^{\star})$ then $u_p$ is degenerate in $H^1_{0,k}(B)$ for $k\geq 2$
if and only if there exists $j\geq 2$ such that
\[\beta_{1,\rad}(p)=-j^2\quad \mbox{ and }\quad k\mid j.\]
In this case the kernel of $L_p$  in $H^1_{0,k}(B)$ is one dimensional and it is spanned by the function $\phi_1(r)\cos(j\theta)$.
\end{itemize}
\end{proposition}

\begin{proof}
$i)$ is obvious, since  $u_p$ is  non-degenerate in $H^1_0(B)$ when $p\in (1,1+\delta)$ (Proposition \ref{risultatoMorse_pvicino1}).
\\
$ii)$ follows from the characterization of the degeneracy of $u_p$ in $H^1_0(B)$ for $p$ large. Indeed  when $u_p$ is degenerate in $H^1_0(B)$ by Proposition \ref{lemma:autovaloriRadialipgrande} the kernel of $L_p$ is spanned by the two functions $\phi_2(r)\sin (\theta)$ and $\phi_2(r)\cos(\theta)$ and neither of the two belong to $H^1_{0,k}(B)$, when $k\geq 2$.
\\
$iii)$ follows from the characterization of the degeneracy of $u_p$ in $H^1_0(B)$ given in Proposition  \ref{p4.7}. Indeed, observe that the solution to the linearized equation $v$ in \eqref{autof2} do not belong to $H^1_{0,k}(B)$ when $k\geq 2$, hence $Ker(L_p)\not = \{0\}$ in $H^1_{0,k}(B)$ if and only if $p$ satisfies the equation \eqref{1deg}. 
To conclude let us recall that in this case $Ker( L_p)$ is spanned by the functions
$\phi_1(r) \sin(j\theta)$ and $\phi_1(r)\cos(j\theta)$  (see  \eqref{autof1}) and that  $\phi_1(r) \sin(j\theta)\not\in H^1_{0,k}(B)$ for $k\geq 2$, while $\phi_1(r)\cos(j\theta)\in H^1_{0,k}(B)$ for  any $k\mid  j$.
\end{proof}

\begin{remark}[Odd change in the $k$-Morse index]
\label{rmk:OddChange}
From Proposition \ref{lemma:degenerazioneSimmetria} - $iii)$, Lemma \ref{lemma:Morse=numeroAutovaloriesatoPallaSIMMETRIA} and the usual spectral decomposition of the negative eigenvalues of the weighted  problem \eqref{problemaPesatoPalla} it follows that $p\in (1, +\infty)$ is a value at which the $k$-Morse index $m_k(u_p)$, $k\geq 2$ changes if and only if there exists $j\geq 2$ such that $k\mid j$ and
$p\in \mathcal P^j$, where $\mathcal P^j$ is defined in \eqref{def:S_j}.\\
Moreover the change in the $k$-Morse index is always odd (precisely $\pm 1$).\\
In particular a sufficient condition for $p$ to be a $k$-Morse index odd changing point is that $p\in \mathcal P^k$.
\end{remark}

\

\section{The bifurcation result}\label{se:bifurcation}

In this section we will find nonradial solutions to \eqref{problem}  bifurcating from the curve of radial solutions $(p,u_p)$, looking for  fixed points of the operator $T\ :\ (1,+\infty)\times C^{1,\alpha}_0(\bar B)\longrightarrow C^{1,\alpha}_0(\bar B)$ defined by
\begin{equation}\label{T}
T(p,u)\ :=\left(-\Delta \right)^{-1}\left(|u|^{p-1}u\right).
\end{equation} 
We will restrict to  the $\mathcal G_k$-invariant functions introduced in Section \ref{se:symmspaces}, in particular let us define the spaces
\begin{equation}\label{X}
\mathcal X_k\ := C^{1,\alpha}(\bar B)\cap H^1_{0,k}(B),
%
\end{equation}
where $H^1_{0,k}(B)$ is the symmetric space in \eqref{Hk}.  We also set 
\begin{equation}\label{Xrad}\mathcal X_{\rad}:=C^{1,\alpha}(\bar B)\cap H^1_{0,\rad}(B).
\end{equation} 
Obviously $u_p\in\mathcal X_{\rad}\subset \mathcal X_k$, for every $p\in (1,\infty)$ and for every $k\geq 1$.
\\
We will look for  solutions in $\mathcal X_k$ which bifurcate at some point $(p^k, u_{p^k})$. Proposition \ref{lemma:degenerazioneSimmetria}-$iii)$ characterizes the values of $p$ at which $u_p$ is degenerate in $\mathcal X_k$, we will show bifurcation for any $p$ in the subset $\mathcal P_k$  (see  \eqref{def:S_j}) of degenerate values, for $k=3,4,5$.
 Observe that for any fixed $p$ the operator $T(p,\cdot)$ is compact and continuous in $p$ and that also its restriction to the subspaces $\mathcal X_k$, $k\geq 2$  is still compact (and continuous in $p$).

\

In particular we will prove that the  continuum of bifurcating solutions belongs to $\mathcal X_k\setminus\mathcal X_j$, $\forall j> k$  until they are non-radial, thus separating the branches related to different values of $k$. In order to get this property we restrict the operator $T$ to  suitable cones $\mathcal K_k$ in $\mathcal X_k$, defined, similarly as in  \cite{D2}, by imposing some  angular monotonicity to the $\mathcal G_k$-symmetric functions. Hence for $k\in\mathbb N_0$  let us  define the cone:
\begin{equation}\label{K-k}
\mathcal K_k\ : \ =\{v\in \mathcal X_k \ \text{ s.t. $v_{\theta}(r,\theta)\leq 0$ for $0\leq \theta\leq\frac{\pi}k$ , $0<r<1$} \},
\end{equation}
where $ v_{\theta}$ denotes the derivative with respect to the angle $\theta$ of the polar coordinates.
By definition $\mathcal X_{\rad}\subset\mathcal K_k\subset \mathcal X_k$ for any $k\geq 1$ and the monotonicity in the definition implies the following \emph{separation property}: 
\begin{equation}\label{quelloCheGuadagnoConConi}\mathcal K_k\cap \mathcal K_h=\mathcal X_{\rad},\quad \forall h\neq k,
\end{equation}
which will be crucial in order to separate the branches.

\

The complete statement of our bifurcation result is contained in Theorem \ref{teo-bif} below, which is the main result of the section, Theorem \ref{teo1} in the introduction follows from it.

\

Let $\mathcal P^k$, $k\in\mathbb N_0$ be the subset of degenerate exponents defined in \eqref{def:S_j}. By Lemma \ref{l5.1} we know that
\[\emptyset\neq \mathcal P^k=\{p_1^k,\ldots, p_{s_k}^k\}, \ \mbox{ when  $\ k=3,4,5$}\]
 (where $s_k\geq 1$ is an odd integer). We then have:
\begin{theorem}
\label{teo-bif} 
The points $(p^k_h, u_{p^k_h})$, $h\in \{1,\dots,s_k\}$ for $k=3,4,5$  are nonradial bifurcation points from the curve of radial solutions $(p,u_p)$ and
the bifurcating solutions belong to the cone $\mathcal K_k$. The bifurcation is global and the Rabinowitz alternative holds.
Moreover, for every $k=3,4,5$ there exists at least one exponent $p^k\in\{p^k_1,\dots,p^k_{s_k}\}$ such that,
letting $\mathcal{C}_k$ be the continuum that branches out
of $(p^k, u_{p^k})$ then either it is unbounded in $(1,+\infty)\times \mathcal K_k$ or it intersects $\{1\}\times \mathcal K_k$. Finally 
$\mathcal C_k\cap\mathcal C_j\subset \mathcal X_{\rad}$ for any $j=3,4,5$, $j\neq k$. 
\end{theorem}

The proof of Theorem \ref{teo-bif} can be found at the end of the section.
The core of the proof consists in getting  bifurcation at the degenerate points  at which there is a change in the  fixed point index of $T(p,\cdot)$  at $u_p$ relative to the cone $\mathcal K_k$ (index introduced in \cite{D83}). These degenerate points $(p,u_p)$ are given by any $p\in\mathcal P^k$ (see Proposition \ref{prop:ChangeIndexCone}). Observe that at  $p\in\mathcal P^k$ also the $k$-Morse index of $u_p$ has a (odd) change (see Remark \ref{rmk:OddChange}).
First we show that:
\begin{lemma}\label{lemma:TMandaConoInCono}
The operator $T(p,\cdot)$ maps $\mathcal X_k$ into $\mathcal X_k$ and in particular  $\mathcal K_k$ into $\mathcal K_k$.
\end{lemma}
\begin{proof}
Let $w\in \mathcal X_k$ and let $z=T(p,w)$. Since $w\in C^{1,\alpha}(B)$ then $z\in C^{3,\alpha}(B)$ and  by definition of $T$,  it is a classical solution to 
\begin{equation}\label{z}
\begin{cases}
-\Delta z= |w|^{p-1}w & \text{ in } B, \\
z= 0 & \text{ on } \partial B.
\end{cases} 
\end{equation} 
Let $\tilde z(x)=z(g(x))$, for $g\in\mathcal G_k$. Then $\tilde z$ is a solution to \eqref{z}, because $w\in \mathcal X_k$ and $-\Delta$ is invariant by the action of $\mathcal G_k$. This implies $\tilde z=z$ getting that $z\in \mathcal X_k$.
\\
It remains to show that when $w\in\mathcal K_k$ also the monotonicity assumption on $w$ is preserved by $T$. Since $z\in C^{3,\alpha}(B)$ we can compute 
 $z_\theta=\frac{\partial z}{\partial \theta}$ and letting $w_\theta=\frac{\partial w}{\partial \theta}$, we have that $z_\theta$ is a classical solution to 
\[
\begin{cases}
-\Delta z_\theta = p|w|^{p-1}w_\theta & \text{ in } (0,1)\times (0,\frac \pi k), \\
z_\theta(1,\theta)= 0 & \text{ on } \partial B.
\end{cases} 
\]
By assumption $w\in \mathcal K_k$ so that $w_\theta\leq 0$ in $(0,1)\times (0,\frac \pi k)$. Moreover $z_\theta(r,0)=0$ since $z$ is even
in $\theta$ (see \eqref{fi00}) and moreover $z_\theta(r,\frac \pi k)=0$  by \eqref{fi1}. The maximum principle then yields $z_{\theta}\leq 0$ for $0\leq \theta\leq\frac{\pi}k$ , $0<r<1$, concluding the proof.
\end{proof}

\

When $u_p$ is an isolated fixed point for $T(p,\cdot)$ we can consider its index relative to the cone $\mathcal K_k$ (see \cite{D83}), which we denote by 
$\mathit{ind}_{\mathcal K_k}\left(T(p,\cdot),u_p\right)$. \\
 We can compute  $\mathit{ind}_{\mathcal K_k}\left(T(p,\cdot),u_p\right)$ when  $u_p$ is non-degenerate in $\mathcal X_k$. In this case the characterization in Proposition \ref{lemma:degenerazioneSimmetria}-$iii)$  implies  in particular  that $\beta_{1,\rad}(p)+k^2\neq 0$, we then have:
\begin{lemma}\label{lemma:calcoloIndexCono} Let $k\geq 2$ and $p$ be such that $u_p$ is non-degenerate in $\mathcal X_k$
then 
\[
\mathit{ind}_{\mathcal K_k}\left(T(p,\cdot),u_p\right)=\left\{
\begin{array}{ll}
0 & \text{ if }\beta_{1,\rad}(p)+k^2<0\\
{1}& \text{ if }\beta_{1,\rad}(p)+k^2>0
\end{array}\right.
\] 
\end{lemma}
\begin{proof}
By Lemma \ref{lemma:TMandaConoInCono} we can consider the operator $T$ restricted to the space $\mathcal X_k$, namely $T : (1,+\infty)\times \mathcal X_k\longrightarrow \mathcal X_k$ for some $k\geq 2$. Let us denote by $T'_u$ the Frech\'et derivative of $T$ with respect to $u$. Since $u_p$ is non-degenerate in $\mathcal X_k$,
then $I-T'_u(p,u_p) : \mathcal X_k\longrightarrow \mathcal X_k$ is invertible. We can then apply Theorem 1 in \cite{D83} getting that
\begin{equation}\label{o2}
\!\!\!\!\!\!\!\!\mathit{ind}_{\mathcal K_k}\left(T(p,\cdot),u_p\right)=\left\{
\begin{array}{ll}
0 & \text{ if } T'_u \text{ has the property } \alpha\\
\mathit{ind}_{\mathcal X_k}\left(T'_u(p,u_p),0\right) & \text{ if } T'_u \text{ does not have  the property } \alpha
\end{array}\right.
\end{equation}
where we refer to \cite{D83} for the definition of the property $\alpha$.
Moreover, since $u_p$ is isolated in $\mathcal X_k$ (again by its nondegeneracy) and since $I-T'_u(p,u_p)$ is invertible we have
\begin{equation}\label{o1}
\mathit{ind}_{\mathcal X_k}\left(T'_u(p,u_p),0\right)=\lim_{r\to 0}\mathit{deg}_{\mathcal X_k}
\left(I-T'_u(p,\cdot), U_r(u_p),0\right)=(-1)^{m_k(u_p)}\end{equation}
where $\mathit{deg}$ is the usual Leray-Schauder degree in the Banach space $\mathcal X_k$, $U_r(u_p):=\{w\in \mathcal X_k\ : \ \|u_p-w\|<r\}$ and the last equality follows by standard results for the Leray Schauder degree of linear, compact, invertible maps (see for instance \cite{AM}). 
The characterization of the degeneracy in $\mathcal X_k$ (see Proposition \ref{lemma:degenerazioneSimmetria}-$iii)$)  implies  in particular  that  $\beta_{1,\rad}(p)+k^2\neq 0$ at the non-degenerate point $p$, the rest of the proof is devoted to show that 
\begin{equation}\label{propertyalpha}
\mbox{$T'_u$ has the property $\alpha$ if and only if $\beta_{1,\rad}(p)+k^2<0$.}
\end{equation} 
In this case indeed \eqref{o2} and \eqref{o1} implies the result since by Lemma \ref{lemma:legameAutovSimmNegativoLinEPesato} and Lemma \ref{LemmaMorseIndexRadiale} one has
\[m_k(u_p)=2,\mbox{ when }\beta_{1,\rad}(p)+k^2>0.\]
The property $\alpha$ in \eqref{o2} is stated in \cite[Lemma 2]{D83}. Following the same notations we have that the linear map $T'_u(p,u_p)$ has the property $\alpha$ if and only if (Lemma 2-(a) of \cite{D83}) the spectral radius of $T'_u(p,u_p)$ is greater than 1 when restricted to the orthogonal complement to $\mathcal X_{\rad}$ in $\mathcal X_k$, which we denote by    
$\mathcal X_{\rad}^\perp$
(observe that in our case the subspace $S_{u_p}$ in \cite{D83} is $\mathcal X_{\rad}$). Equivalently, as observed also in \cite[proof of Theorem 1]{D2}, $T'_u(p,u_p)$ has the property $\alpha$ if and only if there exist  $t\in(0,1)$ and $h\in \mathcal X_{\rad}^\perp$ such that $h=tT'_u(p,u_p)h$, namely, recalling the definition of $T$, such that the linear equation
\begin{equation}\label{o3}
\begin{cases}
-\Delta h-tp|u_p|^{p-1}h=0 & \text{ in }B\\
h=0 &\text{ on }\partial B
\end{cases}\end{equation}
admits a nontrivial solution $h\in \mathcal X_{\rad}^{\perp}$ for some $t\in(0,1)$. This is equivalent to say that zero is an eigenvalue of the problem 
\[
\begin{cases}
-\Delta h-tp|u_p|^{p-1}h=\mu h & \text{ in }B\\
h=0 &\text{ on }\partial B
\end{cases}\]
with eigenfunction in
$\mathcal X_{\rad}^\perp$ for some $t\in(0,1)$. We denote by $\mu_t$ the smallest eigenvalue of this problem in $\mathcal X_{\rad}^\perp$, which depends on $t$.   By the variational characterization of the eigenvalues $\mu_t$ is decreasing in $t$. Moreover $\mu_0>0$, since when $t=0$ $\mu_0$ is the first Dirichlet eigenvalue in $\mathcal X_{\rad}^\perp$ of the Laplace operator in $B$ which is strictly positive. When $t=1$ instead $\mu_1$ is the smallest eigenvalue in $\mathcal X_{\rad}^\perp$ of the linearized operator $L_p$. 
When $\mu_1$ is negative then there exists a $t\in(0,1)$ such that \eqref{o3} has a solution in $\mathcal X_k\setminus\mathcal X _{\rad}$. When $\mu_1$ is positive instead then $\mu_t>\mu_1>0$ for any $t\in (0,1)$ and equation \eqref{o3} does not have a solution in $\mathcal X_k\setminus\mathcal X _{\rad}$.
Finally from Lemma \ref{lemma:legameAutovSimmNegativoLinEPesato} we have that $\mu_1<0$  if and only if $\beta_{1,\rad}(p)+k^2<0$ and this concludes the proof of \eqref{propertyalpha}.
\end{proof}

As a consequence one can characterize the set of the points $p$ at which  the index $\mathit{ind}_{\mathcal K_k}\left(T(p,\cdot),u_p\right)$ changes:
\begin{proposition}[Change in the fixed point index relative to $\mathcal K_k$] \label{prop:ChangeIndexCone} 
$p\in (1,+\infty)$ is a  value at which $\mathit{ind}_{\mathcal K_k}\left(T(p,\cdot),u_p\right)$ changes, for $k\geq 2$ if and only if $p\in\mathcal P^k$,
where the set $\mathcal P^k$ is the one defined in \eqref{def:S_j}.
\end{proposition}
\begin{proof}
If $p\in \mathcal P^k$ then  $(p,u_p)$ is an isolated degenerate point  (Lemma \ref{l5.1}), as a consequence
the values $p=p_h^k\pm\delta$  are non-degenerate  for any $\delta>0$ small
and by definition of $\mathcal P_k$ we also have $[\beta_{1,\rad}(p+\delta)+k^2][\beta_{1,\rad}(p-\delta)+k^2]<0$. The conclusion then follows by Lemma \ref{lemma:calcoloIndexCono} applied at the points $p=p_h^k\pm\delta$.
\\
 Viceversa if $\mathit{ind}_{\mathcal K_k}\left(T(p,\cdot),u_p\right)$ changes at $p$ then  by Lemma \ref{lemma:calcoloIndexCono} $p$ satisfies $\beta_{1,\rad}(p)=-k^2$ and $\beta_{1,\rad}(p)+k^2$ changes sign at $p$. 
This implies that necessarily $p\in\mathcal P^k$.
\end{proof}

\

\subsection{Proof of Theorem \ref{teo-bif}}
\begin{proof}
{\bf Step 1.}
{\sl Non-radial local bifurcation in $\mathcal K_k$}\\
Let us consider $p_h^k$ for a certain $h\in\{1,\ldots, s_k\}$. By Proposition \ref{prop:ChangeIndexCone} we know that  $ \mathit{ind}_{\mathcal K_k}\left(T(p,\cdot),u_p\right)$ changes  as $p$ crosses $p_h^k$,  namely that for any $\delta>0$ small 
\begin{equation}\label{indiceDiv}
\mathit{ind}_{\mathcal K_k}
\left(T(  p^k_h-\delta ,\cdot),u_{p^k_h-\delta}\right)
\neq
\mathit{ind}_{\mathcal K_k}
\left(T(  p^k_h+\delta ,\cdot),u_{p^k_h+\delta}\right),
\end{equation}
 we now show that $(p^k_h, u_{ p^k_h})$ is a bifurcation point in $(1,+\infty)\times \mathcal K_k$.\\
Hence let us  assume by contradiction that  $(p^k_h, u_{ p^k_h})$ is not a
bifurcation point in $(1,+\infty)\times \mathcal K_k$, then we can find $\delta>0$ and a neighborhood $\mathcal O$ of  $\{(p,u_p): p\in ( p^k_h-\delta, p^k_h+\delta)\} $ in $( p^k_h-\delta, p^k_h+\delta)\times \mathcal K_k$ such that $u-T(p,u)\neq 0$ for every $(p,u)$ in $\mathcal O$ different from $(p,u_p)$. We can choose $\delta>0$ such that \eqref{indiceDiv} holds.
Letting $\mathcal O_p:=\{v\in \mathcal K_k\ : \ (p,v)\in \mathcal O\}$, it then follows that  there are no solutions to  $u-T(p,u)= 0$ on $\cup_{p\in  (p^k_h-\delta, p^k_h+\delta) }\{p\} \times \partial  \mathcal{O}_p$ 
and there is only the radial solution $(p,u_p)$ in $\left(\{ p^k_h-\delta\}\times \mathcal{O}_{  p^k_h-\delta }\right)\cup  \left(\{ p^k_h+\delta \}\times \mathcal{O}_{  p^k_h+\delta }\right)$. By the  homotopy invariance of the fixed point index in the cone, see \cite{D83}, then we have that 
\[
\mathit{ind}_{\mathcal K_k}\left(T(p,\cdot),u_p\right) \ \ \hbox{ is constant for } \  p\in(  p^k_h-\delta,  p^k_h+\delta),
\]
which is in contradiction with \eqref{indiceDiv}.
This proves the local bifurcation.
The bifurcating solutions belong to $\mathcal K_k$ since $T$ maps the cone in itself (Lemma  \ref{lemma:TMandaConoInCono}) and are non-radial for $p$ close to $p_h^k$
since
$u_p$ is radially non-degenerate by Lemma \ref{lemma:radiallyNonDeg}.

\

{\bf Step 2.}
{\sl Global bifurcation and Rabinowitz alternative}
\\ 
We can adapt the proof of Theorem 3.3 in \cite{G10}. One of the main differences is that now, since the cone $\mathcal K_k$ is not a Banach space, we substitute  the Leray-Schauder degree used in \cite{G10} with the degree in the convex cone $\mathcal K_k$, which we  denote by $\mathit{deg}_{\mathcal K_k}(I-T(p,\cdot), \mathcal O,0)$, for any  open (with the induced topology) set $\mathcal O$  in $\mathcal K_k$. 
The degree in the convex cone has been introduced in \cite{Amman} (where it is called {\sl index}),  its definition arises directly from the Leray-Schauder degree (to which it coincides when the cone is a Banach space) and in particular it admits the same properties of the Leray-Schauder degree (normalization, additivity, homotopy invariance, permanence, excision, solution property, etc, see \cite[Theorem 11.1 and 11.2]{Amman}).
\\
Following  \cite{G10}, let $\mathcal S\ : =\left\{(p,u_p): p\in (1,+\infty)\right\} \subseteq (1,+\infty)\times \mathcal K_k$  be the curve of radial least-energy solutions, let $\Sigma_k$ be the closure of the set
$  \{(p,v)\in  \left((1,+\infty)\times \mathcal K_k\right)\setminus \mathcal S : v \text{ solves }\eqref{problem}\}$ 
and let $\mathcal C_k$ be the closed connected component of $\Sigma_k$ bifurcating from $( p^k_h , u_{ p^k_h})$.
Assume by contradiction that
the Rabinowitz alternative, namely one of the following, does not occur:
\begin{itemize}
\item[$i)$] $\mathcal C_k$ is unbounded in $(1,+\infty)\times \mathcal K_k$;
\item[$ii)$]  $\mathcal C_k$ intersects $\{1\}\times \mathcal K_k$;
\item[$iii)$]  there exists $ p_l^k$ with $l\neq h$ such that  $( p^k_l, u_{ p^k_l})\in\mathcal C_k\cap \mathcal S$.
  \end{itemize}
Then as in Step 2 in the proof of \cite[Theorem 3.3]{G10} we can then construct a suitable neighborhood $\mathcal O$  of $\mathcal C_k$ in $\mathcal K_k$ such that $\partial \mathcal O\cap \Sigma_k=\emptyset$, $\mathcal O\cap \mathcal S\subset (p^k_h-\delta, p^k_h+\delta)\times \mathcal K_k$ for $\delta$  such that $
u_{p_h^k\pm \delta}$ is nondegenerate and  moreover there exists $c_0>0$ such that  $\| v-u_p\|_{\mathcal X_k}\geq c_0$ for $(p,v)\in \mathcal O$ such that $|p-p^k_h|\geq \delta$. Then we can follow the proof of Step 3 and Step 4 in \cite[Theorem 3.3]{G10}, recalling now that, for $\Lambda_c:=\{(p,v)\in (1,+\infty)\times \mathcal X_k :  \| v-u_p\|_{\mathcal X_k}< c\}
$
one has
\[\mathit{deg}_{\mathcal K_k}(I-T(p_h^k\pm \delta,\cdot), \left(\mathcal O\cap \Lambda_c\right)_{p_h^k\pm \delta},0)=\mathit{ind}_{\mathcal K_k}(T(p_h^k\pm \delta,\cdot),u_{p_h^k\pm \delta})\]
for any $c<c_0$.
The fixed point index relative to the cone $\mathcal K_k$ can be then computed in $p_h^k\pm \delta$  and it assumes either the value $0$ or $1$ (Lemma \ref{lemma:calcoloIndexCono}). The proof of Step 3 and 4 of \cite[Theorem 3.3]{G10} can be repeated and so we get a contradiction.
\\
We can now adapt the proof of \cite[Proposition 2.3]{G16}, again using the degree in the convex cone $\mathcal K_k$ which is, as already observed,  either $0$ or $1$ in a neighborhood of the isolated (in $\mathcal X_k$) solution $u_p$. 
The main difference is that, in the final part of the proof of  \cite[Proposition 2.3]{G16} we now obtain, following the notations of \cite{G16}, that
\[\mathit{deg}_{\mathcal K_k}(S_r(p,v), \mathcal O\cap B_r(p_l^k,u_{ p_l^k}),0)=\pm 1\]
for every $p_l^k\in \mathcal P^k$. This implies again that the number of points $p_l^k\in \mathcal P^k$ which belong to $\mathcal C_k$, including $( p^k_h, u_{ p^k_h})$, has to be even if $\mathcal C_k$ is bounded. 
 Since the total number $s_k$ of  points in $\mathcal P^k$ is odd (see Lemma \ref{l5.1}), then there exist at least one value $p^k\in\{p^k_h\}_{h=1,\ldots, s_k}$ at which either $i)$ or $ii)$ holds.
 
 \
 
{\bf Step 3.}
{\sl Conclusion}\\
Since the bifurcating solutions are not radial for $p$ close to $p_h^k$, the separation property  \eqref{quelloCheGuadagnoConConi} implies that near the bifurcation points $\mathcal{C}_k\neq \mathcal{C}_i$  if $k\neq i$. Moreover $(\mathcal{C}_k\cap \mathcal{C}_i)\subset(\mathcal K_k\cap\mathcal K_j)$ hence it contains only radial solutions.
\end{proof}

\

\begin{remark}[Shape of the bifurcating solutions]
 \label{remark:shape}
Observe that from the definition of the space $\mathcal X_h$ and from the separation property \eqref{quelloCheGuadagnoConConi} of $\mathcal K_k$ it follows that
\begin{equation}\label{quelloCheGuadagnoConConiSeparazione}\mathcal K_k\cap \mathcal X_h=\mathcal X_{\rad},\quad \forall h> k
\end{equation}
 and so, as stated in Theorem \ref{teo1} in the introduction, either the bifurcating solution belongs to $\mathcal X_k\setminus\mathcal X_j$, $\forall j>k$ or it is radial.
\\
Moreover, since the kernel of the linearized operator is one dimensional when restricted to the spaces $\mathcal X_k$ (Proposition \ref{lemma:degenerazioneSimmetria}-iii)), we can get an expansion of the bifurcating solution found in Theorem \ref{teo-bif} near the bifurcation point $(p^k, u_{p^k})$, even if we cannot apply the Crandall-Rabinowitz result to obtain some regularity on the solutions set. Indeed, applying Proposition 2.4 in \cite{G16} we know that there exists $\varepsilon_0>0$ such that for any $0<\varepsilon<\varepsilon_0$ if $(p,v) \in \mathcal{C}_k\cap \left(B_{\varepsilon}( p^k, u_{p^k})\setminus\{( p^k, u_{p^k})\}\right)$, then  
\[ v(r,\theta)= u_p(r)+\alpha_{\varepsilon} \phi_1(r) \cos \left(k \theta\right)+ \psi_{\varepsilon}(r,\theta)\]
where $\alpha_{\varepsilon}\to 0$ as $\varepsilon\to 0$, $\phi_1(r)>0$ is a first eigenfunction of the weighted eigenvalue problem as  defined in Proposition \ref{p4.7} and $\psi_{\varepsilon}(r,\theta)\in \mathcal X_k$ is such that $\|\psi_{\varepsilon}\|_{\infty}=o(\alpha_{\varepsilon})$ as $\varepsilon\to 0$. 
\begin{figure}[h]
  \centering\Huge
\resizebox{240pt}{!}{\input{almostradial.pdf_tex}}
\end{figure}
As a consequence, near the bifurcation point, the solutions we found not only are in $\mathcal X_k\setminus \mathcal X_{\rad}$ but, being small perturbation of the radial least energy solution $u_p$, they also inherit from $u_p$ the property of having two nodal domains and of being quasi-radial in the sense of Definition \ref{def:quasiradial}.
\\
We remark that along the branch the number of nodal regions of the solutions may change and that moreover far from the bifurcation point they may also loose the \emph{quasi-radial shape} and their nodal line could touch the boundary.
\end{remark}

\begin{remark}[Multiple bifurcation]
Observe that we can obtain a solution to \eqref{problem} by rotating the solution $v$ in Theorem \ref{teo-bif} of an angle $\alpha$. This solution coincides with the one bifurcating from $u_p$ in the direction 
\[   w(r,\theta)=\phi_1(r)\left(a\sin (k\theta)-b\cos (k\theta)\right) \in Ker(L_p)\]
with $\alpha=\arctan (-a/b)$, letting $\hat \tau$ be the reflection with respect to the hyperplane $ax+by=0$ and restrincting to the spaces
\[\widehat{\mathcal X}_k \ := C^{1,\alpha}_0( B)\cap \widehat{H}^1_{0,k}(B),\]
where
$\widehat{H}^1_{0,k}(B)\ : =\{v\in H^1_0(B)\ \text{ such that } v(g(x))=v(x), \  \  \forall g\in \widehat{\mathcal{G}}_k , \ \forall x\in B  \}$
and $\widehat{\mathcal{G}}_k\subset O(2)$ is the group generated by $O_k$ and by the reflection $\hat \tau$.

\end{remark}

 \begin{remark}[Bifurcation via odd change in the k-Morse index of $u_p$]
We stress that in order to get the bifurcation result one could work directly in the space $\mathcal X_k$, $k=3,4,5$ without restricting to the cones $\mathcal K_k\subset\mathcal X_k$ substituting the degree  in the cone $\mathcal K_k$ with the usual Leray-Schauder degree  in $\mathcal X_k$.
\\
Anyway the bifurcation  result obtained in this way is only partial, since a priori different branches of solutions could coincide.
\\
The advantage of restricting to the cones  $\mathcal K_k$ in the proof of Theorem \ref{teo-bif}  is that set $\mathcal K_k\cap\mathcal K_j$ contains only radial functions when $k\neq j$, and this allow to separate the branches.
 \end{remark}

\

\section{The proof of Theorem \ref{prop1.4}} \label{section:proofTeoLeastEnergy}

Let us consider the functional $E_p$ in \eqref{functional} restricted to the space $H^1_{0,k}(B)$, $k\in\mathbb N_0$ defined in  \eqref{Hk}. 
By the principle of symmetric criticality critical points of  $E_p$ in $H^1_{0,k}(B)$ are solutions to \eqref{problem} which are invariant by the action of $\mathcal{G}_k$. In particular (see \cite{CastroCossioNeuberger}) to produce nodal solutions to \eqref{problem} which are invariant by the action of $\mathcal{G}_k$  one can minimize the functional $E_p$ on the nodal symmetric Nehari set 
\[M_k\ :=\{v\in H^1_{0,k}(B)\ : \ v^+\neq 0,\ v^-\neq 0,\  E_p'(u)u^+=E_p'(u)u^-=0\}\]
where $E_p'$ is the Fr\'echet derivative of $E_p$. 
Then a function $\bar u$ such that
\[E_p(\bar u)=\inf_{u\in M_k}E_p(u)\]
is a least energy sign-changing $\mathcal{G}_k$-invariant solution to \eqref{problem} and we  denote it by $u_p^k$, for $k=1,2,\dots$.

\

The proof of the {\sl non-radial part} in Theorem \ref{prop1.4} follows directly by comparing the $k$-Morse index of the radial solution $u_p$ (computed in Section \ref{se:symmspaces}) with the $k$-Morse index of the least energy symmetric solution $u_p^k$.

Indeed, following the same arguments in  \cite[Theorem 1.3]{BartschWeth} and working in the space of symmetric functions $H^1_{0,k}(B)$, one can prove  the following result:
\begin{lemma}\label{lem-10-1}
Let $u_p^k$ be a least energy sign-changing solution to \eqref{problem} in the space $H^1_{0,k}(B)$. Then 
\begin{equation}\label{k-morse-upk}
m_k\left( u_p^k\right)=2, \quad\forall p\in (1,+\infty), 
\end{equation}
where $m_k$ denotes the $k$-Morse index of $u_p^k$.
\end{lemma}
By comparing \eqref{k-morse-upk} with the information in Propositions \ref{Morse-simmetrico1} and  \ref{Morse-simmetrico2} we get then the proof of the non-radial part of Theorem 	\ref{prop1.4}
since necessarily  $u_p^k$ is not radial for any  $p$ and $k$ such that $m_k(u_p)>2$.

\

\

The proof of the {\sl radial part} of Theorem \ref{prop1.4} is more involved and is the goal of the rest of this section where first we show an $L^{\infty}$ bound for the solution $u_p^k$ for $p$ close to $1$ (Proposition \ref{prop7.4}) and then, using this bound, we deduce the result by studying the asymptotic behavior of the solutions $u_p^k$ as $p\rightarrow 1$ (this is done in the proof of Proposition \ref{leastRadiale}). 

\

As already discussed in the introduction we do not have  a bound for the full Morse index of $u_p^k$, but {\sl  only for the $k$-Morse index} (Lemma \ref{lem-10-1} above), for this reason,  exploiting the symmetry of $u_p^k$, we reduce problem \eqref{problem} from the ball $B$ to the circular sector $S_k$ of the ball defined in polar coordinates as
\[S_k:=\{(r,\theta)\ : \ 0<r<1 \ \ ,\ \  0<\theta<\frac{\pi}k\}.\]
\begin{figure}[h]
  \centering
  \resizebox{170pt}{!}{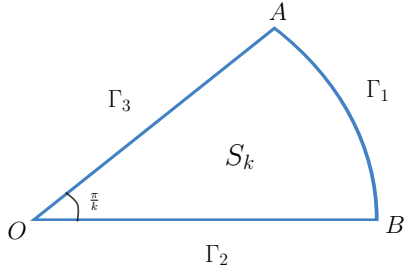}
  \caption{Sector $S_k$}
\end{figure}
Indeed setting  $\Gamma_1:=\{(r,\theta): \ r=1, \ \theta\in(0,\frac \pi k)\}$, $\Gamma_2 :=\{(r,\theta): \ \theta=0, \ r\in(0,1)\}$,   $\Gamma_3:= \{(r,\theta): \ \theta=\frac \pi k,\ r\in(0,1)\}$, $A=(\cos \frac{\pi}k, \sin \frac{\pi}k)$ and $B=(1,0)$, one has $\partial S_k=\Gamma_1\cup \Gamma_2\cup \Gamma_3\cup \{O,A,B\}$ and any regular function $v$ to \eqref{problem} which is invariant by the action of the group $\mathcal{G}_k$, satisfies 
\[v\in C^1(S_k\cup\Gamma_2\cup \Gamma_3\cup O)\ \ , \ \ \frac{\partial  v }{\partial \nu}=0 \ \ \ \text{ on } \Gamma_2\cup \Gamma_3\] 
where $\nu$ denotes the outer normal vector to the boundary of $S_k$.
Hence $u_p^k$ is a classical solution to 
\begin{equation}\label{sector}
\begin{cases}
-\Delta u_p^k=|u_p^k|^{p-1}u_p^k & \text{ in }S_k\\
u_p^k=0 & \text{ on } \Gamma_1\\
\frac{\partial  u_p^k }{\partial \nu}=0 & \text{ on } \Gamma_2\cup \Gamma_3.
\end{cases}\end{equation}

\

In next result we convert the bound on the  $k$-Morse index in \eqref{k-morse-upk}  into a bound on the full {\em{mixed-Morse index}} of $u_p^k$ in the sector $S_k$.

\begin{lemma}\label{lem-10-nuovo}
Let $u_p^k$ be the least energy sign-changing solution to \eqref{problem} in  the space $H^1_{0,k}(B)$. 
Then for any $p\in (1,+\infty)$ the mixed eigenvalue problem 
\begin{equation}\label{sector-lin}
\begin{cases}
-\Delta v=p|u_p^k|^{p-1}v +\mu v& \text{ in }S_k\\
v=0 & \text{ on } \Gamma_1\\
\frac{\partial v}{\partial \nu}=0 & \text{ on } \Gamma_2\cup \Gamma_3
\end{cases}\end{equation}
admits only $2$ negative eigenvalues $\mu$.
\end{lemma}

\begin{proof}
Because of Lemma \ref{lem-10-1} the Dirichlet eigenvalue problem
\begin{equation}\label{nuovo-lin}
\begin{cases}
-\Delta v=p|u_p^k|^{p-1}v +\mu v& \text{ in }B\\
v=0 & \text{ on } \partial B
\end{cases}\end{equation}
admits only two linearly independent eigenfunctions $\tilde \psi_1$ and $\tilde\psi_2$ which are invariant by the action of $\mathcal{G}_k$, are regular, by elliptic regularity theory,
and which correspond to a negative eigenvalue, say $\mu_1^k$ and $\mu_2^k$. By the symmetry properties of $\tilde \psi_i$ it is straightforward to see, that, the restriction of  $\tilde \psi_i$ to the sector $S_k$ satisfies \eqref{sector-lin} corresponding to the same eigenvalue $\mu_i^k<0$ for $i=1,2$. This shows that the number of negative eigenvalues of \eqref{sector-lin} is at least two. Viceversa, if problem \eqref{sector-lin} possess $m>2$ negative eigenvalues $\mu_i$ corresponding to the eigenfunctions $\psi_1,\dots,\psi_m$ (that we take orthogonal in $L^2(S_k)$), then, denoting by $\tilde \psi_1,\dots,\tilde \psi_m$ the extension  of  $\psi_1,\dots,\psi_m$ to $B$ under the action of $\mathcal{G}_k$, it is easy to see that $\tilde \psi_1,\dots,\tilde \psi_m\in H^1_{0,k}(B)$ solve \eqref{nuovo-lin} corresponding to the eigenvalues $\mu_1<\dots\leq \mu_m<0$ and are orthogonal in $L^2(B)$ contradicting Lemma \ref{lem-10-1}. This shows that the number of negative eigenvalues for problem \eqref{sector-lin} is at most two concluding the proof.
\end{proof}

\

In order to get an uniform $L^{\infty}$ bound for the solution $u_p^k$ we want to perform a blow-up argument in the sector  $S_k$ exploiting the uniform bound of the mixed Morse index in Lemma \ref{lem-10-nuovo}. 
\\
This blow-up procedure in $S_k$ requires special care, since we have to deal with mixed boundary conditions and above all with the angular points of $S_k$. For these reasons  the analysis of the rescaled solutions includes several different cases, depending on the location of the maximum points in the sector which gives different shapes of the limiting domain. Anyway in all the cases we end-up with solutions to  a limit linear problem in  unbounded domains 
with either Dirichlet or Neumann or mixed boundary conditions, whose Morse index (or symmetric Morse index) is finite. In order to rule-out this possibility we will need the following  symmetric version of a well known non-existence result:

\begin{proposition}\label{prop7.3}
 Let $\Sigma$ be either $\R^2$ or $\R^2_{+}:=\{(x,y)\in \R^2\ : y> 0\}$ and let 
$\mathcal{G}$ be any subgroup of $O(2)$ which preserves $\Sigma$. 
Let $u$ be any nontrivial solution to the problem
\begin{equation}\label{prob-limite}
-\Delta u-u=0 \ \ \text{ in }\Sigma
\end{equation}
and when $ \Sigma=\R^2_{+}$ assume also that 
\begin{equation}
\label{boundary-cond}
u=0 \ \ \text{ on }\partial \Sigma. 
\end{equation}
Then, the  $\mathcal{G}$-Morse index of $u$ is not finite.
\\
Here
the
$\mathcal{G}$-Morse index of a solution $u$ to \eqref{prob-limite} is the maximal dimension of a subspace $X\subseteq C^{\infty}_{0,\mathcal{G}}(\Sigma)$ such that 
\begin{equation}\label{quad}
Q(v):=\int_{\Sigma} \left[|\nabla v|^2-|v|^2\right]dx <0, \ \forall v\in X\setminus\{0\},
\end{equation}
where $C^{\infty}_{0,\mathcal{G}}(\Sigma)$ denotes the subspace of $C^{\infty}_0(\Sigma)$ of the  functions invariant with respect to the action of $\mathcal{G}$.
\end{proposition}

\begin{proof}
Let us consider first the case of $\Sigma=\R^2$.
Let us denote, as usual, by  $\lambda_j$, $j\in\mathbb N$, the Dirichlet eigenvalues of $-\Delta$ in $B$, since $\mathcal{G}  $ preserves $B$, we can consider among them
the subsequence $\lambda_j^{\mathcal{G}}$ of the eigenvalues corresponding to $\mathcal G$-invariant eigenfunctions. 
\\
Let $\psi_j^\mathcal{G}$ be the $\mathcal G$-invariant eigenfunction associated to $\lambda_j^{\mathcal{G}}$, then it is easy to see that the function $\widehat \psi_j^\mathcal{G}(x):=\psi_j^\mathcal{G}\left(\frac xR\right)$, where $R>0$, solves
\begin{equation}\label{autof-riscalate}
\begin{cases}
-\Delta\widehat \psi_j^{\mathcal{G}} =\frac{\lambda_j^{\mathcal{G}}}{R^2} \widehat \psi_j^\mathcal{G}  & \text{ in }B_R\\
\widehat \psi_j^{\mathcal{G}}=0 & \text{ on }\partial B_R,
\end{cases}
\end{equation}
where $B_R$ is the ball centered at the origin with radius $R$. 
\\
Observe that for any integer $m>0$ and for any  subgroup $\mathcal{G}$ of $ O(2)$ there exists $R>0$ such that $\frac{\lambda_1^{\mathcal{G}}}{R^2}< \dots \leq \frac{\lambda_m^{\mathcal{G}}}{R^2}<1$, so that by \eqref{autof-riscalate} we get 
\[Q\left(\widehat\psi_j^{\mathcal{G}}\right)=\int_{\Sigma} \left[|\nabla \widehat\psi_j^{\mathcal{G}}   |^2-|\widehat \psi_j^{\mathcal{G}} |^2\right]dx=\left(\frac{\lambda_j^{\mathcal{G}}}{R^2}-1\right) \int_{\Sigma}| \widehat\psi_j^{\mathcal{G}} |^2dx<0, \mbox{ for }j=1,\dots,m\]
Since the functions $\widehat \psi_1^{\mathcal{G}}, \dots, \widehat \psi_m^{\mathcal{G}}\in C^{\infty}_{0,\mathcal G}(\Sigma)$ and are linearly independent (and orthogonal in $L^2(B_R)$), this means that the $\mathcal{G}$-Morse index of any nontrivial solution $u$ to \eqref{prob-limite} is greater or equal than $m$, for any $m\in\mathbb N$ showing the result in case of $\Sigma=\R^2$.

\

When $\Sigma=\R^2_{+}$ we let $\lambda_j^+$ be the sequence of Dirichlet eigenvalues of $-\Delta$ in $B\cap  \R^2_+$ and $(\lambda_j^+)^{\mathcal G}$ the subsequence of the eigenvalues invariant with respect to the action of $\mathcal{G}$ with associated $\mathcal G$-invariant eigenfunctions $\psi_j^{\mathcal{G}}$. 
Then defining as before the rescaled function $\widehat \psi_j^\mathcal{G}$, it  solves
\[
\begin{cases}
-\Delta\widehat \psi_j^{\mathcal{G}} =\frac{(\lambda_j^+)^{\mathcal{G}}}{R^2} \widehat \psi_j^\mathcal{G}  & \text{ in }B_R\cap \R^2_+ \\
\widehat \psi_j^{\mathcal{G}}=0 & \text{ on }\partial \left(B_R\cap \R^2_+\right)
\end{cases}
\]
and the thesis follows  similarly as in the previous case.
\end{proof}

\

\

We are now ready to perform the blow-up analysis in $S_k$  to  get a uniform $L^{\infty}$ bound for the solutions $u_p^k$.

\

\begin{proposition}\label{prop7.4}
Let $u_p^k$ be a least energy sign-changing solution to \eqref{problem} in the space $H^1_{0,k}(B)$ and let $\delta>0$. Then there exists $C>0$ such that 
\[\norm{u_p^k}_{\infty}^{p-1}\leq C,\quad \mbox{ for any }p\in (1,1+\delta).\]

\end{proposition}
\begin{proof}

  Assume by contradiction that there exists a sequence $p_n\to 1$ such that, letting $M_n:=\norm{u_n}_{\infty}$ with $u_n:=u_{p_n}^k$,
  $M_n^{p_n-1}\to \infty$ as $n\to \infty$.  
Let $P_n=(x_n,y_n)$ be the points at which $|u_n(P_n)|=M_n$. W.l.o.g. we can assume $u_n(P_n)=M_n$ and, by the symmetry properties of $u_n$, also that $P_n\in  S_k\cup \Gamma_2\cup \Gamma_3\cup\{O\}$. We may also assume that
\[P_n\to P_0:=(x_0,y_0)\in \bar S_k.\] 
We restrict the functions $u_n$ to the sector $S_k$ and 
define the functions
\[\widetilde{u}_n(x,y):=\frac 1{M_n}{u_n(M_n^{\frac {1-p_n}2}(x,y)+P_n)},\]
that satisfy
\[-\Delta \widetilde{u}_n=|\widetilde{u}_n|^{p_n-1}\widetilde{u}_n\]
in $\Omega_n:=M_n^{\frac {p_n-1}2}\left( S_k-P_n\right)$.

\

In the sequel we analyze the asymptotic behavior of the rescaled functions $\widetilde u_n$ and get a contradiction by mean of  Proposition \ref{prop7.3}.
We need to consider several cases depending upon the localization of the limit point $P_0$ in $\bar S_k$.
 The underlying idea of each case is that the sequence of solutions $\widetilde{u}_n$ converges to a non-trivial solution $\widetilde u$ to \eqref{prob-limite} either in $\R^2$ or in a halfplane with Dirichlet boundary conditions. Moreover the bound on the Morse index of $\widetilde{u}_n$ obtained in Lemma \ref{lem-10-nuovo} is preserved when passing to the limit problem. This last property, together with Proposition \ref{prop7.3}, implies $\widetilde u=0$ giving always a contradiction. Thus  $M_n^{p_n-1}$ is bounded and this ends the proof. 

\

Observe that by definition  $(\widetilde x, \widetilde y)\in \Omega_n$ if and only if
\[\widetilde x= M_n^{\frac {p_n-1}2}(x-x_n) \ \ \ \text{ and }\ \ \ \widetilde y= M_n^{\frac {p_n-1}2}(y-y_n)\]
for some $(x,y)\in S_k$, moreover a point $(x,y)$ belongs to $ S_k$ if and only if
\begin{equation}\label{AppS_k}x>0 \, , \ \ \ y> 0\,  ,\ \ \ \frac yx<\tan \frac \pi k \ \ \ \text{ and }\ \ \ 0<x^2+y^2<1.
\end{equation}
As a consequence we deduce that $(\widetilde x, \widetilde y)\in \Omega_n$ if and only if the following inequalities are all satisfied:
\begin{eqnarray}
&&M_n^{\frac {1-p_n}2}\widetilde x+x_n> 0 \, , \label{10-cc}\\
&&M_n^{\frac {1-p_n}2}\widetilde y+y_n> 0\, , \label{10-c}\\
&&\frac{M_n^{\frac {1-p_n}2}\widetilde y+y_n }{M_n^{\frac {1-p_n}2}\widetilde x+x_n }<\tan \frac \pi k
\label{10-a}\\
&&0<x_n^2+y_n^2+M_n^{{1-p_n}}\left(\widetilde x^2+ \widetilde y^2\right)+2M_n^{\frac {1-p_n}2}\left(\widetilde xx_n+\widetilde yy_n\right)<1\label{10-b}
\end{eqnarray}
From now on we denote by $d_n$ the distance between $P_n$ and $\partial S_k$, namely
\begin{equation}\label{dn}
d_n:=\min_{P\in\partial S_k}|P_n-P|.
\end{equation}
{\bf Step 1.} {\emph{$P_0\in S_k$}}
\\
Observe that in this case $d_nM_n^{\frac {p_n-1}2}\rightarrow +\infty$ as $n\rightarrow +\infty$. Indeed, since $P_0\in S_k$, by \eqref{AppS_k} $x_0>0$, $y_0>0$, $x_0^2+y_0^2<1$ and $\frac{y_0}{x_0}<\tan \frac \pi k$, so that, since $M_n^{p_n-1}\to \infty$ as $n\rightarrow +\infty$, any point $(\widetilde x,\widetilde y)\in B_R$  satisfies \eqref{10-cc}, \eqref{10-c}, \eqref{10-a} and \eqref{10-b}, for $n$ large enough, namely for any $R>0$ $B_R\subseteq \Omega_n$ for $n$ large enough. \\
Elliptic estimates imply that, up to a subsequence $ \widetilde{u}_n\to \widetilde u$ uniformly on compact sets of $\R^2$. By the argument in \cite{GS}
 $\widetilde u$  is defined in all of $\R^2$, it is a nontrivial weak solution to \eqref{prob-limite} in $\Sigma=\R^2$ and satisfies $\widetilde u(0)=1$. 
\\
Finally we show that the Morse index of the limit function $\widetilde u$ is less or equal than $2$, this contradicts Proposition \ref{prop7.3} and proves the thesis in the case  $P_0\in S_k$. \\
Assume, by contradiction, that the Morse index of $\widetilde u$ as a solution to \eqref{prob-limite} is greater than $2$. Then there exist at least $3$ functions $\widetilde\psi_1,\widetilde\psi_2,\widetilde\psi_3\in C^{\infty}_0(\R^2)$ such that $\widetilde\psi_i$ are linearly independent (orthogonal in $L^2(\R^2)$) and
\[Q(\widetilde\psi_i)<0\]
where $Q$ is the quadratic form as defined in \eqref{quad}. 
Since $ \widetilde\psi_i$ are supported in a ball $B_R$ then, the uniform convergence of $\widetilde{u}_n\to \widetilde u$ on compact sets of $\R^2$ implies that 
\[\int_{\R^2}|\nabla \widetilde\psi_i|^2 -  p_n|\widetilde{u}_n|^{p_n-1} \widetilde\psi_i^2 <0\]
for $n$ large enough. 
Then the functions $\widehat\psi_i(x,y):=\widetilde \psi_i\left(\frac{(x,y)-P_n}{M_n^{\frac{p_n-1}2}}\right)$ belong to $C^{\infty}_0(S_k)$ for $n$ large enough, are orthogonal in $L^2(S_k)$ and satisfy
 \[\int_{S_k}|\nabla\widehat\psi_i|^2 -  p_n|{u}_n|^{p_n-1}\widehat\psi_i^2 <0\]
for $i=1,2,3$. Then, letting $\psi_i\in C^{\infty}_0(B) $ be the $\mathcal{G}_k$-invariant extension of $\widehat\psi_i$ to the ball $B$, it holds
 \[\int_{B}|\nabla\psi_i|^2 -  p_n|{u}_n|^{p_n-1}\psi_i^2 <0\]
for $i=1,2,3$ contradicting the fact that the $k$-Morse index of ${u}_n$ is two (Lemma \ref{lem-10-1}). 

\

{\bf Step 2.}{\emph{ $P_0\in \Gamma_1$}}
\\
In this case we have to consider the two possibilities either $d_nM_n^{\frac{p_n-1}2}\to\infty$ or $d_nM_n^{\frac{p_n-1}2}\to s>0$, for $d_n$ as in \eqref{dn} (the fact that $s>0$ is a consequence of the Dirichlet boundary conditions on $\Gamma_1$ and can be deduced exactly as in the paper \cite{GS}).
Then, as in the proof in \cite{GS} the rescaled functions $ \widetilde{u}_n^k\to \widetilde u$ as $n\to \infty$ uniformly on compact sets of $\Sigma$, where $\widetilde u$ is a nontrivial solution (recall that $\widetilde{u}(0)=1$) either to \eqref{prob-limite} in $\Sigma=\R^2$ in the first case or in $\Sigma=\R^2_{+}$ in the second case (up to a rotation and a translation) satisfying \eqref{boundary-cond}.  Moreover one can prove similarly as in {\bf Step 1} that $\widetilde u$ has finite Morse index, contradicting again  Proposition \ref{prop7.3}.

\

{\bf Step 3.} {\emph{$P_0\in \Gamma_2\cup \Gamma_3$}}
\\
We give the details of the proof only in the case $P_0\in \Gamma_2$ since the case $P_0\in \Gamma_3$ can be handled in a similar way.
In this case $d_n=y_n\to 0$ ($d_n$ as in \eqref{dn}) and $x_n\to x_0$ as $n\to \infty$ with $0<x_0<1$, hence a point $(\widetilde x,\widetilde y)\in B_R$ satisfies  \eqref{10-c}, \eqref{10-a} and \eqref{10-b}  for $n$ large enough, and so it belongs to $\Omega_n$ if and only if  \eqref{10-cc} holds, namely when
\[\widetilde y>-y_nM_n^{\frac{p_n-1}2}.\] Two possibilities may hold: either $y_nM_n^{\frac{p_n-1}2}\to\infty$ or $y_nM_n^{\frac {p_n-1}2}\to s\geq 0$.

\

{\emph{Case 1: $y_nM_n^{\frac{p_n-1}2}\to\infty$.}}\\ 
In the first case it follows that any ball $B_R\subset \Omega_n$ for $n$ large enough, namely $\Omega_n\to \Sigma=\R^2$ and so, as in {\bf Step 1}, $\widetilde{u}_n\to \widetilde u$ uniformly on compact sets of $\Sigma$, where $\widetilde u$ is a nontrivial solution to \eqref{prob-limite} in $\R^2$ that satisfies $\widetilde u(0)=1$ and that has finite Morse index, getting a contradiction. 

\

{\emph{Case 2: $y_nM_n^{\frac{p_n-1}2}\to s\geq 0$.}}\\
In this case instead $\Omega_n\to  \Sigma:=\{(x,y)\in \R^2 : \ y> -s\} $ for some $s\geq 0$ and $ \widetilde{u}_n\to \widetilde u$ on compact sets of $\Sigma$ where $\widetilde u$ is a solution to \eqref{prob-limite} in $\Sigma:=\{(x,y)\in \R^2 : \ y> -s\}$ that satisfies a Neumann boundary condition on $\partial \Sigma$.\\
When $s>0$, $0\in \Omega_n$ for $n$ large enough, hence $\widetilde{u}$ is nontrivial since $\widetilde{u}(0)=1$ by the uniform convergence on compact sets. Finally by translating this limit nontrivial solution in the $y$-direction we then end-up, when $s>0$, with a nontrivial solution $\widetilde u$ to \eqref{prob-limite} in $\Sigma=\mathbb R^2_+$ with  Neumann boundary conditions on $\partial \Sigma$.\\
Next we treat the case $s=0$ and show that again the limit solution $\widetilde u$ is  non-trivial. Observe that $\widetilde y=-M_n^{\frac{p_n-1}2}y_n\in\partial\Omega_n$ and that in the case $s=0$  it belongs to a neighborhood of $0$ for $n$ large. By the elliptic regularity up to the boundary (see Lemma 6.18 in \cite{GT}) for the equation $-\Delta \widetilde u_n=f_n$ with $f_n=|\widetilde u_n|^{p_n-1}\widetilde u_n$,
we obtain a uniform bound on the gradient of $\widetilde u_n$ in $\overline \Omega_n\cap B_{\rho}$, for $\rho$ sufficiently small (indeed by definition  $|\widetilde u_n|\leq 1$ on $\partial \Omega_n$, hence $|f_n(x)|\leq 1$ and we use the fact that $u_n\in C^{2,\gamma}(\Gamma_2)$).   
This implies that
\[\widetilde u_n(F)\geq \widetilde u_n(0)-C|F-0|=1-C|F|, \quad  \forall F\in \Omega_n\cap B_{\rho}\]
where $C$ is the uniform bound on the gradient. Choosing $F$ in the set $\Sigma= \{(x,y)\in\R^2   : \ y> 0\}$ and sufficiently close to $0$ and passing to the limit in the previous inequality one then has $\widetilde u(F)>0$, 
namely $\widetilde u$ is  non-trivial.\\
Summarizing, for any $s\geq 0$, we have obtained a non-trivial solution $\widetilde u$ to  \eqref{prob-limite} in $\Sigma:=\mathbb R^2_+$ that satisfies a Neumann boundary condition on $\partial \Sigma$.
Moreover, as a consequence of Lemma \ref{lem-10-nuovo}, similarly  
 as in {\bf Step 1}, one can easily prove that the maximal number of linearly independent functions $\widetilde\psi_i$
in the space 
$C^{\infty}_0(\overline{\R^2_+})\cap \{\frac{\partial \widetilde\psi_i}{\partial y}{\big|_{y=0}}=0\}$ 
that make negative the quadratic form $Q$ 
is at most $2$. 
As a consequence, the even extension of $\widetilde u$ to the whole $\R^2$ is a nontrivial solution to \eqref{prob-limite} in $\Sigma=\R^2$ which has finite $\mathcal G$-Morse index,  where $\mathcal{G}$ here is the group generated by the reflection with respect to the $x$-axis. 
Again this is not possible by Proposition \ref{prop7.3}.

\

{\bf Step 4.} {\emph{$P_0=B$ ($P_0=A$ follows similarly)}}. 
\\
Since we are assuming that $M_n^{p_n-1}\to \infty$ and $(x_n,y_n)\to (1,0)$ it is straightforward to see that a point $(\widetilde x, \widetilde y)\in B_R$ satisfies \eqref{10-cc}, \eqref{10-a} and the first inequality in \eqref{10-b}  for large values of $n$ and so it belongs to $\Omega_n$ for large $n$ if and only if \eqref{10-c}  and the second inequality in \eqref{10-b} are satisfied, namely:
\begin{equation}\label{10-5}
\widetilde y>- y_nM_n^{\frac {p_n-1}2}
\end{equation} 
\begin{equation}\label{10-6}
M_n^{\frac {1-p_n}2}\left(\widetilde x^2+ \widetilde y^2\right)+2\left(\widetilde xx_n+\widetilde yy_n\right)<\left(1-x_n^2-y_n^2\right)M_n^{\frac {p_n-1}2}
\end{equation}
Hence we have to to distinguish several  possibilities: \begin{eqnarray}
& \mbox{ either } &  y_nM_n^{\frac {p_n-1}2}\to \infty\label{10-1}\\
& \mbox{ or } & y_nM_n^{\frac {p_n-1}2}\to \alpha\geq 0\label{10-2}
\end{eqnarray}
as $n\to \infty$ and also 
\begin{eqnarray}
& \mbox{ either } &  \left(1-x_n^2-y_n^2\right)M_n^{\frac {p_n-1}2}\to \infty\label{10-3}\\
  & \mbox{ or } & \left(1-x_n^2-y_n^2\right)M_n^{\frac {p_n-1}2}\to \beta> 0\label{10-4}
\end{eqnarray}
as $n\to \infty$, where the case $\beta=0$ is ruled-out by the Dirichlet boundary conditions on $\Gamma_1$ (as in {\bf Step 2}).\\
Observe that 
\eqref{10-1} implies \eqref{10-5}  for large $n$, while when \eqref{10-2} holds then \eqref{10-5} is satisfied  for $n$ large if and only if 
$\widetilde y>-\alpha$. Similarly if \eqref{10-3} holds then \eqref{10-6} is satisfied when $n$ is large, while if \eqref{10-4} holds then \eqref{10-6} is satisfied for $n$ large if and only if 
$\widetilde x<\frac {\beta}2 $.\\
Summarizing we have that $ \widetilde{u}_n\to \widetilde u$ uniformly on compact sets of $\Sigma$, where $\widetilde u$ is a solution to \eqref{prob-limite} in $\Sigma$, more precisely:   

\

{\emph{Case 1: \eqref{10-1} and \eqref{10-3} hold.}}\\
In this case $\Sigma =\R^2$, $\widetilde u$ is nontrivial (since $\widetilde u(0)=1$) and moreover, as in {\bf Step 1} one can prove that $\widetilde u$ has finite Morse index  contradicting Proposition  \ref{prop7.3}. 

\

{\emph{Case 2: \eqref{10-1} and \eqref{10-4} hold.}}\\
In this case $\Sigma=\{(x,y)\in \R^2: \ x<\frac {\beta}2\}$, $\widetilde u$ is  nontrivial  (again $0\in \Omega_n$ when $n$ is large enough and then $\widetilde u(0)=1$), it satisfies  Dirichlet boundary conditions on the hyperplane $x=\frac {\beta}2$ and has finite Morse index. This (up to a translation) contradicts again Proposition  \ref{prop7.3}.

\

{\emph{Case 3: \eqref{10-2} and \eqref{10-3}  hold.}}\\
Now $\Sigma=\{(x,y)\in \R^2: \ y>-\alpha\}$, $\widetilde u$  satisfies 
Neumann boundary conditions on the hyperplane $y=-\alpha$. If $\alpha>0$ then, as before, $\widetilde u(0)=1$ and so it is nontrivial. In this case we translate this solution in the $y$-direction getting a solution to \eqref{prob-limite} in $\R^2_+$ that satisfies 
Neumann boundary conditions and we obtain a contradiction as in {\bf Step 3}-{\emph{Case 2}}.
In the case $\alpha=0$ we observe that $d_n=y_n$ (where $d_n$ as usual is the distance  in \eqref{dn}). Indeed 
$P_0=B$ implies that $d_n= \min\{dist (P_n,\Gamma_2), dist (P_n,\Gamma_1)\}$, where $dist (P_n,\Gamma_2)=y_n$ and $dist (P_n,\Gamma_1)= 1-\sqrt{x_n^2+y_n^2}$,  
  moreover 
$1-\sqrt{x_n^2+y_n^2}\geq y_n$ if and only if
\begin{equation}\label{intermedia}
y_n(2-y_n)\leq 1-x_n^2-y_n^2,\end{equation}
and \eqref{intermedia}  holds for $n$ large, under the assumptions \eqref{10-2} with $\alpha=0$ and \eqref{10-3}.
Since $d_n=y_n$, then $\widetilde y=-M_n^{\frac{p_n-1}2}y_n\in\partial\Omega_n$ and moreover it belongs to a neighborhood of $0$ for $n$ large, hence we can reason as in {\bf Step 3}-{\emph{Case 2}} and use the elliptic regularity up to the boundary to obtain a uniform estimate on the gradient of $\widetilde u_n$ in a neighborhood of $0$, showing that $\widetilde u$ is nontrivial. Again we obtain a contradiction as at the end of {\bf Step 3}-{\emph{Case 2}}.

\

{\emph{Case 4: \eqref{10-2} and \eqref{10-4}  hold.}}\\
Now  $\Sigma=\{(x,y)\in \R^2: \ y>-\alpha, \ x<\frac {\beta}2\}$, $\widetilde u$  satisfies  Dirichlet boundary conditions on the hyperplane $x=\frac {\beta}2$ and  Neumann boundary conditions on the hyperplane $y=-\alpha$. As before when $\alpha>0$ we have that $0\in \Omega_n$ when $n$ is large enough and then $\widetilde u(0)=1$, namely  $\widetilde u$ is nontrivial and so we translate it ending with a nontrivial solution $\bar u$ to \eqref{prob-limite} in $\bar{\Sigma}=\{(x,y)\in \R^2: \ y>0, \ x<0\}$, with Dirichlet boundary conditions on $x=0$ and Neumann boundary conditions on $y=0$. When $\alpha=0$ one proves 
\eqref{intermedia} as in the previous case, so again $d_n=y_n$ for large $n$. Then $\widetilde y=- M_n^{\frac {p_n-1}2}y_n\in\partial\Omega_n$ and it belongs to a neighborhood of $0$ for large $n$, so we can prove that $\widetilde u$ is nontrivial using again the elliptic regularity up to the boundary as in the previous situation. Also in this case we translate $\widetilde u$ ending with a nontrivial solution $\bar u$ to \eqref{prob-limite} in $\bar{\Sigma}=\{(x,y)\in \R^2: \ y>0, \ x<0\}$, with Dirichlet boundary conditions on $x=0$ and Neumann boundary conditions on $y=0$.
\\
Finally observe that as a consequence of Lemma \ref{lem-10-nuovo}, using arguments similar to the ones in {\bf{Step 1}}, one can prove that  the maximal number of linearly independent functions 
$\widetilde\psi_i\in C^{\infty}_0(\{(x,y)\in \R^2: \ y\geq 0, \ x<0\})\cap \{\frac{\partial \widetilde\psi_i}{\partial y}{\big|_{y=0}}=0\}$ 
that make negative the quadratic form $Q$ 
is at most $2$.
Thus, by extending $\bar u$ to $\widetilde \Sigma :=\{(x,y)\in \R^2:  \ x<0\}$ in an even way,  we obtain a solution to \eqref{prob-limite} in $\widetilde \Sigma$ which has finite $\mathcal{G}$-Morse index, where $\mathcal{G}$ here is the group generated by the reflection with respect to the $x$-axis. This is again in contradiction  with Proposition  \ref{prop7.3}.

\

{\bf Step 5.} {\emph{$P_0=O$}}
\\
In this case we can assume w.l.o.g. that $d_n=y_n$, since $P_0=O$ implies that $d_n=\min \{dist(P_n,\Gamma_2),dist(P_n,\Gamma_3)\}$, $dist(P_n,\Gamma_2)=y_n$ and w.l.o.g (up to rotation) we may consider only the case $dist(P_n,\Gamma_2)\leq dist(P_n,\Gamma_3)$. 
We may also assume that $y_n\leq x_n$ and $\frac{y_n}{x_n}\leq \tan \frac{\pi}{2k}$ (if $x_n\neq 0$). 
Then a point $(\widetilde x, \widetilde y)\in B_R(0)$ for some $R>0$ belongs to $\Omega_n$ if and only if conditions \eqref{10-c} and \eqref{10-a} are satisfied.
Indeed \eqref{10-b} is easily verified. 
We have to distinguish different cases, since
\begin{eqnarray}
&\mbox{either } & y_nM_n^{\frac {p_n-1}2}\to \infty\label{10-7}\\
&\mbox{or } &y_nM_n^{\frac {p_n-1}2}\to \alpha\geq 0\label{10-8}
\end{eqnarray}
and
\begin{eqnarray}
& \mbox{either } & x_nM_n^{\frac {p_n-1}2}\to \infty\label{10-9}\\
&\mbox{or } &x_nM_n^{\frac {p_n-1}2}\to \beta\geq 0,\label{10-10}
\end{eqnarray}
where it is obvious that \eqref{10-7} implies \eqref{10-9} and that \eqref{10-10} implies \eqref{10-8} with $\alpha\leq\beta$  (since $y_n\leq x_n$).

\

{\emph{Case 1: \eqref{10-7}  holds.}}\\
In this case also  \eqref{10-9}  holds and $d_nM_n^{\frac {p_n-1}2}\to \infty$, hence \eqref{10-c} and \eqref{10-a} are satisfied for large $n$ and so $\Omega_n\to \R^2$. 
Then $ \widetilde{u}_n\to \widetilde u$ uniformly on compact sets of $\R^2$ where $\widetilde u$ is a nontrivial (since $\widetilde u(0)=1$) solution to \eqref{prob-limite} in $\R^2$ of finite Morse index, giving a contradiction to the results of Proposition \ref{prop7.3}.

\

{\emph{Case 2: \eqref{10-8} and \eqref{10-9}  hold.}}\\
\eqref{10-a} is  satisfied  for large $n$ while \eqref{10-c} is satisfied for large $n$ if and only if $\widetilde y>-\alpha$. Hence the limit domain is  $\Sigma=\{(x,y)\in \R^2: \ y>-\alpha\}$ and $\widetilde{u}_n\to \widetilde u$ uniformly on compact sets of $\Sigma$ where $\widetilde u$ is a solution to \eqref{prob-limite} in $\Sigma$ that satisfies a Neumann boundary condition on $y=-\alpha$ of finite Morse index, in the sense of {\bf Step 3}. 
Moreover when $\alpha>0$ then $0\in \Omega_n$ and this implies that $\widetilde u$ is nontrivial getting a contradiction. When $\alpha=0$ we observe that  $\widetilde y=- M_n^{\frac {p_n-1}2}y_n\in\partial\Omega_n$ and it belongs to
a neighborhood of $0$. 
We can  therefore apply the elliptic regularity up to the boundary as in {\bf Step 3} getting that $\widetilde u$ is nontrivial. Thus a contradiction arises as in the previous case.

\

{\emph{Case 3: \eqref{10-10}  holds with $\beta>0$.}}\\
In this case also condition  \eqref{10-8} holds with $0\leq\alpha\leq\beta$, which implies that \eqref{10-c} is satisfied for large $n$ if and only if  $\widetilde y>-\alpha$. Moreover
by \eqref{10-cc} and \eqref{10-10}  it follows that $\widetilde x>- \beta$. Condition \eqref{10-a} is satisfied for large $n$, instead, if and only if
\[ \frac{\widetilde y+ \alpha}{\widetilde x+ \beta}< \tan \frac{\pi}k.\]
Then the limiting domain $\Sigma$ is a positive cone in $\R^2$ with vertex in $(-\beta, -\alpha)$ and with amplitude $\frac{\pi}k$ (the same of $S_k$) 
\[
\Sigma=\left\{
(r\cos\theta -\beta,r\sin\theta-\alpha)\ : \ r\in (0,+\infty), \ \theta \in [0,\frac{\pi}{k}]
\right\}
\]
Then  $\widetilde{u}_n\to \widetilde u$ uniformly on compact sets of $\Sigma$ where $\widetilde u$ is a solution to \eqref{prob-limite} in $\Sigma$ that satisfies a Neumann boundary condition on $\partial \Sigma$. When $\alpha,\beta\neq 0$ then $0\in \Sigma$ and we can infer that $\widetilde u$ is nontrivial. The same is true when $\alpha=0$, since $\beta>0$ and in this case we have that $\widetilde y=- M_n^{\frac {p_n-1}2}y_n\in\partial\Omega_n$ and belongs to a neighborhood of $0$, so we can reason as in {\bf Step 3} the and show that $\widetilde u$ is nontrivial. Moreover in both the cases $\widetilde u$ has finite Morse index, since the maximal number of linearly independent functions $\widetilde \psi_i$ in $C^{\infty}_0(\overline \Sigma)\cap\{\frac{\partial\widetilde \psi_i}{\partial \nu}|_{\partial \Sigma}=0\}$ ($\nu$ denotes the outer normal to $\partial \Sigma$) that make negative the quadratic form $Q$ is at most two due to Lemma \ref{lem-10-nuovo}. Translating  $\widetilde u$ with respect to one or both the axes we end-up with a function $\bar u$ that satisfies \eqref{prob-limite} in $\{(x,y)\in \R^2: x>0, y>0, \frac yx<\tan \frac \pi k\}$ and  Neumann boundary conditions. Finally the $\mathcal{G}_k$ extension of $ \bar u$ to the whole $\R^2$ (which is well defined  due to the Neumann boundary conditions) is a non trivial $k$-symmetric solution to \eqref{prob-limite} in $\mathbb R^2$ which has $k$-Morse index at most $2$. This contradicts the result in Proposition \ref{prop7.3}.

\

{\emph{Case 4: \eqref{10-10}  holds with $\beta=0$.}}
In this case also condition  \eqref{10-8} holds with $\alpha=0$.
We consider the solution $u_n$ in the whole ball $B$ (without restricting  it to the sector $S_k$) and we define
\[\widetilde{v}_n(x,y):=\frac 1{M_n}{u_n(M_n^{\frac {1-p_n}2}(x,y))}\]
that satisfies
\[-\Delta \widetilde{v}_n=|\widetilde{v}_n|^{p_n-1}\widetilde{v}_n\]
in $\widetilde B_n:=M_n^{\frac {p_n-1}2} B$ and also $|\widetilde{v}_n|\leq 1$. The rescaled domain $\widetilde B_n\to \R^2$ and $\widetilde{v}_n\to \widetilde v$  uniformly on compact sets of $\R^2$ where $\widetilde v$ is a solution to \eqref{prob-limite} which has $k$-Morse index at most $2$ (observe that since we are rescaling with respect to the origin the symmetries are preserved). To obtain a contradiction via Proposition \ref{prop7.3} we need to show that $\widetilde v$ is nontrivial. This easily follows since  $\widetilde{v}_n(\widetilde P_n)=1$, where $\widetilde P_n=(M_n^{\frac {p_n-1}2}x_n, M_n^{\frac {p_n-1}2}y_n)$ and by assumption $\widetilde P_n\to 0$, so that $\widetilde v(0)=1$. This end the proof.
\end{proof}

\

Now we are in the position to consider the asymptotic behavior of the nodal least energy solutions $u_p^k$  as $p\to 1$
and to conclude the proof of the radial part of Theorem \ref{prop1.4}.

\begin{proposition}\label{leastRadiale}
The least energy nodal solutions $u_p^k$ are radial for any $k\geq 3$ when $p$ is close to $1$.
\end{proposition}
\begin{proof}
{\bf Step 1.} {\emph{We show that for any sequence $p_n>1$  converging to $1$ 
	\begin{equation}\label{u-nkappaconv}
	\bar{u}_n^k:=\frac {u_{p_n}^k}{\norm{u_{p_n}^k}_{\infty}}\rightarrow 
	C\varphi_{2,\rad} =J_0(\nu_{02}|x|)\quad \text{ in }C(\bar B)
	\end{equation}
	up to a subsequence, where $C=\pm1$ and
	\begin{equation}\label{8.14} 
	\norm{u_{p_n}^k}_{\infty}^{p_n-1}= \lambda_{2,\rad}\left(1-\widetilde c (p_n-1)\right)+o(p_n-1) \text{ as }n\to \infty
	\end{equation}
	where $\widetilde c$ is as in \eqref{c-tilde}.
	}}	
	
	\
		
Let $M_n:=\norm{u_{p_n}^k}_\infty$, we have shown in Proposition \ref{prop7.4} that $M_n^{p_n-1}$ is bounded, we can then repeat the proof of Lemma \ref{lemma-pvicino1} proving  that
\[M_n^{p_n-1}\to \lambda \ \ \text{and} \ \  \bar u_n^k\to C\varphi  \text{ in } C(\bar B) \, \text{ up to a subsequence,  with }C=\pm 1\]
where $\lambda$ is an eigenvalue of $-\Delta$ in $B$ with Dirichlet boundary conditions, $\varphi$ is a corresponding eigenfunction with $\|\varphi\|_{\infty}=1$. Moreover $\varphi$ is invariant by the action of $\mathcal{G}_k$ (since $\bar u_{n}^k$ are for every $n$) and, following the ideas in Step 1 in the proof of Proposition \ref{risultatoMorse_pvicino1} we can show that $m_k(\varphi)\leq m_k(u_{p_n}^k)$, hence $m_k(\varphi)\leq 2$ by Lemma \ref{lem-10-1}. Since the $k$-symmetric eigenvalues of $-\Delta$ are known and since we are assuming $k\geq 3$, this means that necessarily either $\lambda=\lambda_{1,\rad}$ or  $\lambda=\lambda_{2,\rad}$.
We show that the case $\lambda=\lambda_{1,\rad}$ cannot hold. Indeed, following similar ideas as in Step 2 of the proof of Proposition \ref{risultatoMorse_pvicino1}, since $\varphi_{1,\rad}$ has Morse index $0$, one gets that the $2$ negative $k$-symmetric  eigenvalues of the linearized operator at $u_{p_n}^k$ (recall $m_k(u_{p_n}^k)=2$ by Lemma \ref{k-morse-upk})  converge both to $0$ and that the corresponding eigenfunctions (that we can take to be orthogonal in $L^2(B)$) converge to two orthogonal solutions of
\[
\left\{
\begin{array}{lr}
-\Delta v=\lambda_1 v & \text{ in }B\\
v=0 & \mbox{ on }\partial B.
\end{array}
\right.
\]
This is not possible, since $\lambda_1$ is simple, so $\lambda=\lambda_{2,\rad}$.  Reasoning exactly as in the proof of Lemma \ref{lemma-pvicino1}, we can then prove \eqref{8.14}.
Assuming w.l.o.g. that $\bar u_n^k(0)>0$ for $n$ large, we also have
\[\bar u_n^k\to\varphi_{2,\rad}=J_0(\nu_{02}|x|) \text{ as }n\to \infty \ \text{ in } C(\bar B),
\]
getting \eqref{u-nkappaconv}.

\

{\bf Step 2.} {\emph{We show that $u_p^k=u_p$ for $p$ close to $1$, where as usual $u_p$ is the least energy nodal radial solution to \eqref{problem}.}}

\

Assume by contradiction that there exists a sequence $p_n>1$, $p_n\rightarrow 1$ as $n\rightarrow +\infty$ such that $u_n^k\neq u_n$, where $u_n^k:=u_{p_n}^k$ and $u_n:=u_{p_n}$, and define
$w_n:=\frac{u_n^k- u_n}{\norm{u_{n}^k- u_n}_{\infty}}$. $w_n$  satisfies
\begin{equation}\label{8.12}
\left\{\begin{array}{lr}
-\Delta w_n= p_nc_n(x)w_n \qquad\  \mbox{ in }B\\
w_n =0\qquad\qquad\qquad\qquad\mbox{ on }\partial B\\
 \norm{w_n}_{\infty}=1
\end{array}\right.
\end{equation}
where, by the Mean value Theorem,
\begin{equation}\label{cn}
c_n(x)=\int_0^1 |tu_n^k+(1-t) u_n|^{p_n-1}\, dt\leq \norm{u_n^k}_{\infty}^{p_n-1}+\norm{u_n}_{\infty}^{p_n-1}\leq \overset{{\eqref{8.14}-\eqref{gammanp-1}}}{C}\lambda_{2,\rad}.
\end{equation}
We show that 
\begin{equation}\label{infinitob}
c_n(x)\to \lambda_{2,\rad}\ \text{ almost everywhere in } B\ \mbox{ as } \ n\rightarrow \infty.
\end{equation}

Indeed from  \eqref{gammanp-1} and \eqref{u-n} we have that
\begin{eqnarray*}
\frac{u_n}{\lambda_{2,\rad}^{\frac 1{p_n-1}}}&=&\frac{u_n}{\norm{u_n}_{\infty}}\left(\frac{\norm{u_n}_{\infty}^{p_n-1}}{\lambda_{2,\rad}}\right)^{\frac 1{p_n-1}}=\bar{u}_n\left(1-\widetilde c(p_n-1)+o(p_n-1)\right)^{\frac 1{p_n-1}}\\
&=&\varphi_{2,\rad}e^{-\widetilde c}(1+o(1))
\end{eqnarray*}  as $n\to \infty$,  where $\widetilde c$ is as in \eqref{c-tilde},
and the same holds for $u_n^k$ using  \eqref{8.14} and \eqref{u-nkappaconv}. Namely
\[
\frac{u_n}{e^{-\widetilde c}\lambda_{2,\rad}^{\frac 1{p_n-1}}}\to \varphi_{2,\rad}\ \text{ and }\ \frac{u_n^k}{e^{-\widetilde c}\lambda_{2,\rad}^{\frac 1{p_n-1}}}\to \varphi_{2,\rad}\quad  \text{ in $C(\bar B)$ as }n\to \infty.
\]
As a consequence, for any $x\in B$  we have 
\begin{equation}\label{cuoreb}
t \frac{u_n^k}{e^{-\widetilde c}\lambda_{2,\rad}^{\frac 1{p_n-1}}}   +(1-t) \frac{u_n}{e^{-\widetilde c}\lambda_{2,\rad}^{\frac 1{p_n-1}}} \to \varphi_{2,\rad}
\end{equation} 
and \eqref{infinitob} follows then from \eqref{cuoreb} observing that
\begin{equation}\label{8.20}\begin{split}
\frac {c_n(x)}{\lambda_{2,\rad}}& =\int_0^1 \Big|t \frac{u_n^k}{\lambda_{2,\rad}^{\frac 1{p_n-1}}}   +(1-t) \frac{u_n}{\lambda_{2,\rad}^{\frac 1{p_n-1}}} \Big|^{p_n-1}\, dt=\\
&=e^{-\widetilde c(p_n-1)}\int_0^1 \Big|t \frac{u_n^k}{e^{-\widetilde c}\lambda_{2,\rad}^{\frac 1{p_n-1}}}   +(1-t) \frac{u_n}{e^{-\widetilde c}\lambda_{2,\rad}^{\frac 1{p_n-1}}} \Big|^{p_n-1}\, dt.
\end{split}.
\end{equation}
Passing to the limit in \eqref{8.12} and using \eqref{infinitob} get that
$w_n$ converges,  up to a subsequence, in $C(\bar B)$ to a function $w$ which  solves
\begin{equation}\label{8.13}
\left\{\begin{array}{lr}
-\Delta w= \lambda_{2,\rad}w \qquad  \mbox{ in }B\\
w =0\qquad\qquad\qquad\mbox{ on }\partial B\\
\norm{w}_{\infty}=1 
\end{array}\right.
\end{equation}
so that 
\begin{equation}\label{defC}
w=C\varphi_{2,\rad}, \quad\mbox{ with }\ C=\pm1\ \mbox{ depending on the sign of }w(0).\end{equation} 
On the other side, multiplying \eqref{8.12} by $\varphi_{2,\rad}$ and integrating over $B$ we find
\begin{equation}\label{8.18}
\begin{split}
&\lambda_{2,\rad} \int_B w_n \varphi_{2,\rad}=\int_B \nabla  w_n\nabla \varphi_{2,\rad}=\lambda_{2,\rad} p_n\int_B \frac{c_n(x)}{\lambda_{2,\rad}}  w_n\varphi_{2,\rad}\\
&=\lambda_{2,\rad}\int_B \frac{c_n(x)}{\lambda_{2,\rad}}  w_n\varphi_{2,\rad}+\lambda_{2,\rad}(p_n-1)\int_B \frac{c_n(x)}{\lambda_{2,\rad}}  w_n\varphi_{2,\rad}.
\end{split}
\end{equation}
Using  the trivial  equality $e^x-1=x\int_0^1e^{sx}ds $ and \eqref{8.20}, we write
\begin{eqnarray*}
\frac{c_n(x)}{\lambda_{2,\rad}}&=&\int_0^1 1+(p_n-1)\log \Big| t\frac{u_n^k}{\lambda_{2,\rad}^{\frac 1{p_n-1}}} +(1-t) \frac{u_n} {\lambda_{2,\rad}^{\frac 1{p_n-1}}}\Big|
\int_0^1 \Big| t\frac{u_n^k}{\lambda_{2,\rad}^{\frac 1{p_n-1}}} +(1-t) \frac{u_n} {\lambda_{2,\rad}^{\frac 1{p_n-1}}}\Big|^{s(p_n-1)}\ ds\, dt\\
&=&1+(p_n-1)g_n(x),
\end{eqnarray*}
where 
\[g_n(x):=\int_0^1\log \Big| t\frac{u_n^k}{\lambda_{2,\rad}^{\frac 1{p_n-1}}} +(1-t) \frac{u_n} {\lambda_{2,\rad}^{\frac 1{p_n-1}}}\Big|\int_0^1 \Big| t\frac{u_n^k}{\lambda_{2,\rad}^{\frac 1{p_n-1}}} +(1-t) \frac{u_n} {\lambda_{2,\rad}^{\frac 1{p_n-1}}}\Big|^{s(p_n-1)}\ ds\, dt.\]
Equation \eqref{8.18} then becomes
\[
\lambda_{2,\rad} \int_B w_n \varphi_{2,\rad}=
\lambda_{2,\rad}\int_B \left( 1+(p_n-1)g_n(x) \right)  w_n\varphi_{2,\rad}+\lambda_{2,\rad}(p_n-1)\int_B \frac{c_n(x)}{\lambda_{2,\rad}}  w_n\varphi_{2,\rad}.
\]
so that, dividing  by $\lambda_{2,\rad}(p_n-1)$ we obtain 
\begin{equation}\label{8.18bis}
0=\int_B g_n(x)w_n\varphi_{2,\rad}+\int_B \frac{c_n(x)}{\lambda_{2,\rad}}  w_n\varphi_{2,\rad}.
\end{equation}
Observe now that, by \eqref{cuoreb}, for any $x\in B$ such that $\varphi_{2,\rad}\neq 0$ we have that 
\begin{equation}\label{quadratinob} 
g_n(x)\to \log \big|\varphi_{2,\rad}e^{-\widetilde c}\big|=\log\big|\varphi_{2,\rad}\big|  -\widetilde c\ \text{  as }n\to \infty.
\end{equation}
This implies that $g_n(x)\varphi_{2,\rad}\in L^{\infty}(B)$ and \[\norm{g_n(x)\varphi_{2,\rad}}_{\infty}\leq C.\]
We can then pass to the limit as $n\to \infty$ into \eqref{8.18bis} and using  \eqref{quadratinob}  and \eqref{infinitob} we get
\[
0=C\int_B\left( \log\big|\varphi_{2,\rad}\big|  -\widetilde c\right) \varphi_{2,\rad}^2+C\int_B \varphi_{2,\rad}^2
\]
which implies, using the definition of $\widetilde{c}$ in \eqref{c-tilde}, that
\[
0=C\int_B \varphi_{2,\rad}^2
\]
namely that $C=0$, contradicting the definition of $C$ in \eqref{defC} and ending the proof. 
\end{proof}

\begin{remark} \label{RemarkNoQuaiRadk=2}
One could prove, reasoning as in the proof of Proposition \ref{leastRadiale}, that $\bar u_p^2\rightarrow \varphi$ in $C^1(\bar B)$ as $p\rightarrow 1$, where $\varphi$ is an eigenfunction of $-\Delta$ corresponding to the eigenvalue $\lambda_4=\lambda_5$, which is not quasi-radial.
\begin{figure}[h]
  \centering\Huge
\resizebox{100pt}{!}{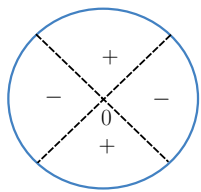}
\caption{Eigenfunction associated to $\lambda_4=\lambda_5$}
\end{figure}The convergence in $C^1(\bar B)$, by the Hopf lemma then implies that $u_p^2$ is not quasi-radial  for $p$
close to $1$.\end{remark}

\end{document}